\documentclass[11pt,a4paper]{article}

\usepackage[utf8]{inputenc}
\usepackage[T1]{fontenc}
\usepackage{epsfig}
\usepackage{amsmath, amsthm}
\usepackage{amsfonts,amssymb}
\usepackage{mathtools}
\usepackage{hyperref}
\hypersetup{colorlinks=true,pdftitle="The convexification effect of Minkowski summation",pdftex}

\newtheorem{conj}{Conjecture}[section]
\newtheorem{thm}[conj]{Theorem}

\newtheorem{rem}[conj]{Remark}
\newtheorem{lem}[conj]{Lemma}
\newtheorem{prop}[conj]{Proposition}
\newtheorem{coro}[conj]{Corollary}

\newtheorem{defn}[conj]{Definition}
\newtheorem{cor}[conj]{Corollary}
\newtheorem{ex}[conj]{Example}

\newcommand{\conv}{\mathrm{conv}}

\newcommand{\vol}{\mathrm{Vol}}
\newcommand{\diam}{\mathrm{diam}}

\newcommand{\supp}{\mathrm{supp}}

\newcommand{\inr}{\mathrm{inr}}

\newcommand{\R}{\mathbb{R}}
\newcommand{\N}{\mathbb{N}}
\newcommand{\Z}{\mathbb{Z}}

\newcommand{\bc}{\begin{center}}
\newcommand{\ec}{\end{center}}
\newcommand{\bt}{\begin{tabular}}
\newcommand{\et}{\end{tabular}} 
\newcommand{\bea}{\begin{eqnarray}}
\newcommand{\eea}{\end{eqnarray}}
\newcommand{\bean}{\begin{eqnarray*}}
\newcommand{\eean}{\end{eqnarray*}}

\newcommand{\ba}{\begin{array}}
\newcommand{\ea}{\end{array}}

\def\be{\begin{eqnarray}}
\def\ee{\end{eqnarray}}
\def\ben{\begin{eqnarray*}}
\def\een{\end{eqnarray*}}

\newcommand{\ra} {\rightarrow}
\newcommand{\la} {\leftarrow}

\newcommand{\nth}{\frac{1}{n}}

\newcommand{\Nat}{\mathbb{N}}

\newcommand{\lam}{\lambda}
\newcommand{\eps}{\epsilon}

\def\elabel#1{\label{e:#1}}

\def\sq{$\Box$}

\def\qed{\ifmmode\sq\else{\unskip\nobreak\hfil
\penalty50\hskip1em\null\nobreak\hfil\sq
\parfillskip=0pt\finalhyphendemerits=0\endgraf}\fi\par\medbreak}

\def\supp{{\rm supp\,}}

\newsavebox{\junk}
\savebox{\junk}[1.6mm]{\hbox{$|\!|\!|$}}

\def\det{{\mathop{\rm det}}}
\def\limsup{\mathop{\rm lim\ sup}}
\def\liminf{\mathop{\rm lim\ inf}}

\def\til={{\widetilde =}}

\def\half{{\mathchoice{\textstyle \frac{1}{2}}%
{\frac{1}{2}}%
{\hbox{\tiny $\frac{1}{2}$}}%
{\hbox{\tiny $\frac{1}{2}$}} }}

 \def\eq#1/{(\ref{#1})}

\def\eq#1/{(\ref{e:#1})}

\newcommand{\beqn}[1]{\notes{#1}%
\begin{eqnarray} \elabel{#1}}

\newcommand{\eeqn}{\end{eqnarray} }

\newcommand{\beq}[1]{\notes{#1}%
\begin{equation}\elabel{#1}}

\newcommand{\eeq}{\end{equation}} 

\def\bdes{\begin{description}}
\def\edes{\end{description}}

\def\notes#1{}

\newcommand{\setS}{s}
\newcommand{\setT}{t}

\newcommand{\collS}{\mathcal{C}}

\newcommand{\sumS}{\sum_{\setS\in\collS}}

\newcommand{\as}{\alpha_{\setS}}

\newcommand{\bs}{\beta_{\setS}}

\def\E{\mathbb{E}}
\def\R{{\bf R}}

\def\P{{\bf P}}
\def\H{{\cal H}}

\def\phi{\varphi}

\def\bee{\begin{eqnarray*}}
\def\ene{\end{eqnarray*}}

\newcommand{\rad}{\xrightarrow{d_H}}

\renewcommand{\la}{\lambda}
\renewcommand{\eps}{\varepsilon}

\begin{document}

\title{The convexification effect of Minkowski summation}
\author{Matthieu Fradelizi\thanks{supported in part by the Agence Nationale de la Recherche, project GeMeCoD (ANR 2011 BS01 007 01).}, 
Mokshay Madiman\thanks{supported in part by the U.S. National Science Foundation through grants DMS-1409504 (CAREER) and CCF-1346564.}, 
Arnaud Marsiglietti\thanks{supported in part by the Institute for Mathematics and its Applications with funds provided by the National Science Foundation.} and Artem Zvavitch\thanks{ supported in part by the U.S. National Science
Foundation Grant DMS-1101636, Simons Foundation and the B\'ezout Labex.}}

\date{}

\maketitle

\begin{abstract}
Let us define for a compact set $A \subset \R^n$ the sequence
$$ A(k) = \left\{\frac{a_1+\cdots +a_k}{k}: a_1, \ldots, a_k\in A\right\}=\frac{1}{k}\Big(\underset{k\ {\rm times}}{\underbrace{A + \cdots + A}}\Big). $$
It was independently proved by Shapley, Folkman and Starr (1969) and by Emerson and Greenleaf (1969)
that  $A(k)$ approaches the convex hull of $A$ in the Hausdorff distance induced by the Euclidean norm as $k$ goes to $\infty$.
We explore in this survey how exactly $A(k)$ approaches the convex hull of $A$, and more generally,
how a Minkowski sum of possibly different compact sets approaches convexity,
as measured by various indices of non-convexity. The non-convexity indices considered
include the Hausdorff distance induced by any norm on $\R^n$,
the volume deficit (the difference of volumes), a non-convexity index introduced by Schneider (1975),
and the effective standard deviation or inner radius. After first clarifying the interrelationships between these 
various indices of non-convexity, which were previously either unknown or scattered in the literature, we show that 
the volume deficit of $A(k)$ does not monotonically decrease to 0 in dimension 12 or above, thus falsifying a conjecture of
Bobkov et al. (2011), even though their conjecture is proved to be true in dimension 1 and for certain sets $A$ with
special structure. On the other hand, Schneider's index possesses a strong monotonicity property along the sequence $A(k)$, and 
both the Hausdorff distance and effective standard deviation are eventually monotone (once $k$ exceeds $n$).
Along the way, we obtain new inequalities for the volume of the Minkowski sum of compact sets (showing that this is 
fractionally superadditive but not supermodular in general, but is indeed supermodular when the sets are convex),
falsify a conjecture of Dyn and Farkhi (2004),
demonstrate applications of our results to combinatorial discrepancy theory,
and suggest some questions worthy of further investigation.
\end{abstract}



\vspace{1cm}

\newpage
\tableofcontents
\newpage

\section{Introduction}
\label{sec:intro}

Minkowski summation is a basic and ubiquitous operation on sets. Indeed, the Minkowski sum
$A+B = \{a+b: a \in A, b \in B \}$ of sets $A$ and $B$ makes sense as long as $A$ and $B$ are subsets
of an ambient set in which the operation + is defined. In particular, this notion makes sense
in any group, and  there are multiple fields of mathematics that are preoccupied with 
studying what exactly this operation does. For example, much of classical additive combinatorics
studies the cardinality of Minkowski sums (called sumsets in this context) of finite subsets of a group and their interaction
with additive structure of the concerned sets, while the study of the Lebesgue measure of Minkowski sums in $\R^n$
is central to much of convex geometry and geometric functional analysis. 
In this survey paper, which also contains a number of original results, our goal is to understand 
better the qualitative effect of Minkowski summation in $\R^n$-- specifically, the ``convexifying''
effect that it has. Somewhat surprisingly, while the existence of such an effect has long been known,
several rather basic questions about its nature do not seem to have been addressed, and we undertake 
to fill the gap.

The fact that Minkowski summation produces sets that look ``more convex''  is easy to visualize
by drawing a non-convex set\footnote{The simplest nontrivial example is three non-collinear points
in the plane, so that $A(k)$ is the original set $A$ of vertices of a triangle together with those
convex combinations of the vertices formed by rational coefficients with denominator $k$.} in the plane and its self-averages $A(k)$ defined by
\begin{eqnarray}\label{defAk}
A(k) = \left\{\frac{a_1+\cdots +a_k}{k} : a_1, \ldots, a_k\in A\right\}=\frac{1}{k}\Big(\underset{k\ {\rm times}}{\underbrace{A + \cdots + A}}\Big).
\end{eqnarray}
This intuition was first made precise in the late 1960's independently\footnote{Both the papers
of Starr \cite{Sta69} and Emerson and Greenleaf \cite{EG69} were submitted in 1967 and published in 1969,
but in very different communities (economics and algebra); so it is not surprising that the authors of these papers were
unaware of each other. Perhaps more surprising is that the relationship between these papers does not seem
to have ever been noticed in the almost 5 decades since. The fact that $A(k)$  converges 
to the convex hull of $A$, at an $O(1/k)$ rate in the Hausdorff metric when dimension $n$ is fixed, 
should perhaps properly be called the Emerson-Folkman-Greenleaf-Shapley-Starr theorem,
but in keeping with the old mathematical tradition of not worrying too much about names of theorems (cf., Arnold's principle),
we will simply use the nomenclature that has become standard.} by Starr \cite{Sta69} (see also \cite{Sta81}), who credited
Shapley and Folkman for the main result, and by Emerson and Greenleaf \cite{EG69}. Denoting by $\conv(A)$ the convex hull of $A$, by $B_2^n$ the $n$-dimensional Euclidean ball of radius $1$, and by $d(A)= \inf\{r>0: \conv(A)\subset A+rB_2^n\}$ the Hausdorff distance between a 
set $A$ and its convex hull, it follows from the Shapley-Folkman-Starr theorem that 
if $A_1, \ldots, A_k$ are compact sets in $\R^n$ contained inside some ball, then
\ben
d(A_1 +\cdots + A_k) = O\big( \sqrt{\min\{k,n\}} \big) .
\een
By considering $A_1=\cdots=A_k=A$, one concludes that 
$d(A(k))=O \big( \frac{\sqrt{n}}{k} \big)$. In other words,
when $A$ is a compact subset of $\R^n$ for fixed dimension $n$, 
$A(k)$ converges in Hausdorff distance to $\conv(A)$ as $k\ra\infty$,
at rate at least $O(1/k)$.

Our geometric intuition would suggest that in some sense, as $k$ increases,
the set $A(k)$ is getting progressively more convex, or in other words, that
the convergence of $A(k)$ to $\conv(A)$ is, in some sense, monotone.
The main goal of this paper is to examine this intuition, and explore whether it can
be made rigorous. 

One motivation for our goal of exploring monotonicity in the Shapley-Folkman-Starr theorem is that it
was the key tool allowing Starr \cite{Sta69} to prove that in an economy
with a sufficiently large number of traders, there are (under some natural
conditions) configurations arbitrarily close
to equilibrium even without making any convexity assumptions on preferences of the traders;
thus investigations of monotonicity in this theorem speak to the question
of whether these quasi-equilibrium configurations in fact get ``closer'' to a true equilibrium
as the number of traders increases. 
A related result is the core convergence result of Anderson \cite{And78},
which states under very general conditions that 
the discrepancy between a core allocation
and the corresponding competitive equilibrium price vector in a
pure exchange economy becomes arbitrarily small as the number
of agents gets large. These results are central results in mathematical economics,
and continue to attract attention (see, e.g., \cite{Sch12}).

Our original motivation, however, came from a conjecture made 
by Bobkov, Madiman and Wang \cite{BMW11}. To state it,
let us introduce the volume deficit $\Delta(A)$ of a compact set $A$ in $\R^n$: 
$\Delta(A):= \vol_n(\conv(A)\setminus A) = \vol_n(\conv(A))-\vol_n(A)$,
where $\vol_n$ denotes the Lebesgue measure in $\R^n$.

\begin{conj}[Bobkov-Madiman-Wang \cite{BMW11}]\label{weakconj}
Let $A$ be a compact set in $\R^n$ for some $n\in\Nat$, and let $A(k)$
be defined as in \eqref{defAk}. 
Then  the sequence $\{\Delta(A(k))\}_{k \geq 1}$ is non-increasing in $k$,
or equivalently, $\{\vol_n(A(k))\}_{k \geq 1}$ is non-decreasing. 
\end{conj}

In fact, the authors of \cite{BMW11} proposed a number of related conjectures,
of which Conjecture~\ref{weakconj} is the weakest. Indeed, 
they conjectured a monotonicity property in a probabilistic limit theorem, namely the law of 
large numbers for random sets due to Z.~Artstein and Vitale \cite{AV75};
when this conjectured monotonicity property of \cite{BMW11}  is restricted to deterministic
(i.e., non-random) sets, one obtains Conjecture \ref{weakconj}. They showed
in turn that this conjectured monotonicity property in the law of large numbers for
random sets is implied by the following volume inequality for Minkowski sums.
For $k\ge1$ being an integer, we set $[k]=\{1, \ldots, k\}$. 

\begin{conj}[Bobkov-Madiman-Wang \cite{BMW11}]\label{strongconj}
Let $n\ge1$, $k\ge2$ be integers and let $A_1, \dots, A_k$ be $k$ compact sets in $\R^n$. Then 
\begin{eqnarray}\label{conjdimn}
\vol_n\left(\sum_{i=1}^kA_i\right)^\frac{1}{n}\ge \frac{1}{k-1}\sum_{i=1}^k\vol_n\left(\sum_{j\in[k]\setminus\{i\}}A_j\right)^\frac{1}{n}.
\end{eqnarray}
\end{conj}

Apart from the fact that Conjecture~\ref{strongconj} implies Conjecture~\ref{weakconj} (which can be seen simply by applying the former to $A_1=\cdots=A_k=A$, where $A$ is a fixed compact set),
Conjecture~\ref{strongconj} is particularly interesting because of its close connections to
an important inequality in Geometry, namely the Brunn-Minkowski inequality,
and a fundamental inequality in Information Theory, namely the entropy power inequality.
Since the conjectures in \cite{BMW11} were largely motivated by these connections,
we now briefly explain them.

The Brunn-Minkowski inequality (or strictly speaking, the Brunn-Minkowski-Lyusternik inequality)
states that for all compact sets $A,B$ in $\R^n$,
\begin{eqnarray}\label{BMI}
\vol_n(A+B)^{1/n} \geq \vol_n(A)^{1/n} + \vol_n(B)^{1/n} .
\end{eqnarray}
It is, of course, a cornerstone of Convex Geometry, and has beautiful relations to many areas
of Mathematics (see, e.g., \cite{Gar02, Sch14:book}).
The case $k=2$ of Conjecture~\ref{strongconj} is exactly the Brunn-Minkowski inequality (\ref{BMI}). Whereas Conjecture~\ref{strongconj} yields the monotonicity described in Conjecture~\ref{weakconj},
the Brunn-Minkowski inequality only allows one to deduce that the subsequence 
$\{\vol_n(A(2^k))\}_{k \in \N}$ is non-decreasing
(one may also deduce this fact from the trivial inclusion $A \subset \frac{A+A}{2}$). 

The entropy power inequality states that for all independent random vectors $X, Y$ in $\R^n$,
\begin{eqnarray}\label{EPI}
N(X+Y) \geq N(X) + N(Y),
\end{eqnarray}
where
$$N(X) = \frac{1}{2\pi e} e^{\frac{2h(X)}{n}}$$
denotes the entropy power of $X$. Let us recall that the entropy of a random vector $X$ with density function $f_X$ 
(with respect to Lebesgue measure $dx$) is $h(X) = -\int f_X(x) \log f_X(x) dx$ if the integral exists and $- \infty$ otherwise (see, e.g., \cite{CT91:book}). As a consequence, one may deduce that for independent and identically distributed random vectors $X_i$, $i \geq 0$, the sequence
$$ \left\{ N \left( \frac{X_1 + \cdots + X_{2^k}}{\sqrt{2^k}} \right) \right\}_{k \in \N} $$
is non-decreasing. 
S.~Artstein, Ball, Barthe and Naor \cite{ABBN04:1}  
generalized the entropy power inequality (\ref{EPI}) by proving that for any independent random vectors $X_1, \dots, X_k$,
\begin{eqnarray}\label{newEPI}
N\left(\sum_{i=1}^k X_i\right) \geq \frac{1}{k-1} \sum_{i=1}^k N\left(\sum_{j \in [k] \setminus \{i\}} X_j\right).
\end{eqnarray}
In particular, if all $X_i$ in the above inequality are identically distributed, then one may deduce that the sequence
$$ \left\{ N \left( \frac{X_1 + \cdots + X_{k}}{\sqrt{k}} \right) \right\}_{k \geq 1} $$
is non-decreasing. This fact is usually referred to as ``the monotonicity of entropy in the Central Limit Theorem'',
since the sequence of entropies of these normalized sums converges to that of a Gaussian distribution as shown earlier
by Barron \cite{Bar86}.
Later, simpler proofs of the inequality \eqref{newEPI} were given by \cite{MB06:isit, TV06}; more general inequalities
were developed in \cite{MB07, Shl07, MG17}.

There is a formal resemblance between inequalities (\ref{EPI}) and (\ref{BMI}) that was noticed in a pioneering work of Costa and Cover \cite{CC84} and later explained by Dembo, Cover and Thomas \cite{DCT91} (see also \cite{SV00, WM14} %\cite{Lie78}
for other aspects of this connection). In the last decade, several further developments have been made that link 
Information Theory to the Brunn-Minkowski theory, including entropy analogues of the Blaschke-Santal\'o inequality \cite{LYZ04},
the reverse Brunn-Minkowski inequality \cite{BM11:cras, BM12:jfa}, the Rogers-Shephard inequality \cite{BM13:goetze, MK18} and the Busemann inequality \cite{BNT16}.
Indeed, volume inequalities and entropy inequalities (and also certain small ball inequalities \cite{MMX17:1}) can be unified using
the framework of R\'enyi entropies; this framework and the relevant literature is surveyed in \cite{MMX17:0}. 
On the other hand, natural analogues in the Brunn-Minkowski theory of Fisher information inequalities 
hold sometimes but not always  \cite{FGM03, AFO14, FM14}.
 %(see, e.g., \cite{FGM03, LYZ04, BM11:cras, BM11:it, BM12:jfa, AFO14, FM14, CW14}). 
In particular, it is now well understood that the functional $A\mapsto \vol_n(A)^{1/n}$ in the geometry of compact subsets of $\R^n$, and the functional $f_X\mapsto N(X)$ in probability are analogous to each other in many (but not all) ways. Thus, for example, the monotonicity property desired in Conjecture~\ref{weakconj}
is in a sense analogous to the monotonicity property in the Central Limit Theorem implied by inequality \eqref{newEPI},
and Conjecture \ref{strongconj} from \cite{BMW11} generalizes the Brunn-Minkowski inequality (\ref{BMI})
exactly as inequality (\ref{newEPI}) generalizes the entropy power inequality (\ref{EPI}).

The starting point of this work was the observation that although Conjecture \ref{strongconj} holds for
certain special classes of sets (namely, one dimensional compact sets, convex sets and their Cartesian product, as shown in subsection \ref{sec:1d}),
both Conjecture~\ref{weakconj} and Conjecture~\ref{strongconj} fail to hold in general even for moderately high dimension 
(Theorem~\ref{thm:counter} constructs a counterexample in dimension 12). These results,
which consider the question of the monotonicity of $\Delta(A(k))$ are stated and proved 
in Section~\ref{sec:Delta}. We also discuss there the question of when one has
convergence of $\Delta(A(k))$ to 0, and at what rate, drawing on the work of the
\cite{EG69} (which seems not to be well known in the contemporary literature on convexity). 

Section~\ref{sec:vol} is devoted to developing some new volume inequalities for Minkowski sums.
In particular, we observe in Theorem~\ref{thm:fsa} that if the exponents of $1/n$ in Conjecture~\ref{strongconj} are removed,
then the modified inequality is true for general compact sets
(though unfortunately one can no longer directly relate this to 
a law of large numbers for sets). Furthermore, in the case of convex sets, Theorem~\ref{thm:smod}
proves an even stronger fact, namely that the volume of the Minkowski sum of 
convex sets is supermodular. Various other facts surrounding these observations are also 
discussed in Section~\ref{sec:vol}.

Even though the conjecture about $A(k)$ becoming progressively more convex in the sense
of $\Delta$ is false thanks to Theorem~\ref{thm:counter}, one can ask the same question when we measure the extent of non-convexity
using functionals other than $\Delta$. In Section~\ref{sec:meas}, we survey the existing literature
on measures of non-convexity of sets, also making some possibly new observations about
these various measures and the relations between them. The functionals we consider
include a non-convexity index $c(A)$ introduced by Schneider \cite{Sch75}, the notion of
inner radius $r(A)$ introduced by Starr \cite{Sta69} (and studied in an equivalent form
as the effective standard deviation $v(A)$ by Cassels \cite{Cas75}, though the equivalence was only understood later by Wegmann \cite{Weg80}),
and the Hausdorff distance $d(A)$ to the convex hull, which we already introduced when
describing the Shapley-Folkman-Starr theorem. We also consider the generalized Hausdorff distance $d^{(K)}(A)$
corresponding to using a non-Euclidean norm whose unit ball is the convex body $K$.
The rest of the paper is devoted to
the examination of whether $A(k)$ becomes progressively more convex as $k$ increases,
when measured through these other functionals.

In Section~\ref{sec:c}, we develop the main positive result of this paper, Theorem~\ref{thm:c-quant-mono},
which shows that $c(A(k))$ is monotonically (strictly) decreasing in $k$, unless
$A(k)$ is already convex. Various other properties of Schneider's non-convexity index
and its behavior for Minkowski sums are also established here, including the optimal
$O(1/k)$ convergence rate for $c(A(k))$. We remark that even the question of convergence of $c(A(k))$
to 0 does not seem to have been explored in the literature.

Section~\ref{sec:r} considers the behavior of $v(A(k))$ (or equivalently $r(A(k))$). 
For this sequence, we show that monotonicity holds in dimensions 1 and 2, 
and in general dimension, monotonicity holds eventually (in particular, once $k$ exceeds $n$).
The convergence rate of $r(A(k))$ to 0 was already established
in Starr's original paper \cite{Sta69}; we review the classical proof of Cassels \cite{Cas75}
of this result.
%{\color{red} Should we include the Starr proof as well (there is a 1981 version
%and the original version)? That is also nice,
%and more geometric. Matthieu may remember discussing this with Mokshay.}

Section~\ref{sec:d} considers the question of monotonicity of $d(A(k))$,
as well as its generalizations $d^{(K)}(A(k))$ when we consider $\R^n$ equipped with norms
other than the Euclidean norm (indeed, following \cite{BG81}, we even consider
so-called ``nonsymmetric norms''). Again here, we show that
monotonicity holds in dimensions 1 and 2, 
and in general dimension, monotonicity holds eventually (in particular, once $k$ exceeds $n$).
In fact, more general inequalities are proved that hold for Minkowski sums of different sets. 
The convergence rate of $d(A(k))$ to 0 was already established
in Starr's original paper \cite{Sta69}; we review both a classical proof,
and also provide a new very simple proof of a rate result
that is suboptimal  in dimension for the Euclidean norm but sharp in both
dimension and number $k$ of summands given that it holds for arbitrary norms. In 2004  Dyn and Farkhi \cite{DF04} conjectured that
$d^2(A+B) \leq d^2(A)+d^2(B).$ We show that this conjecture is false in $\R^n$,  $n \ge 3$.

In Section~\ref{sec:discrep}, we show that a number of results from combinatorial discrepancy theory
can be seen as consequences of the convexifying effect of Minkowski summation.
In particular, we obtain a new bound on the discrepancy for finite-dimensional Banach spaces
in terms of the Banach-Mazur distance of the space from a Euclidean one.

Finally, in Section~\ref{sec:disc}, we make various additional remarks, including on notions
of non-convexity not considered in this paper.

\vspace{.1in}
\noindent
{\bf Acknowledgments.}
Franck Barthe had independently observed that Conjecture~\ref{strongconj} holds in dimension 1, using the same proof, by 2011.
We are indebted to Fedor Nazarov for valuable discussions, in particular for the help in the construction of the 
counterexamples in Theorem~\ref{thm:counter} and Theorem \ref{thm:DF}. We would like to thank Victor Grinberg for many enlightening
discussions on the connections with discrepancy theory, which were an enormous help with putting Section~\ref{sec:discrep}
together. We also thank Franck Barthe, Dario Cordero-Erausquin, Uri Grupel, 
Bo'az Klartag, Joseph Lehec, Paul-Marie Samson, Sreekar Vadlamani, and Murali Vemuri for interesting discussions.
Some of the original results developed in this work were announced in \cite{FMMZ16}; we are grateful to Gilles Pisier
for curating that announcement. Finally we are grateful to the anonymous referee for a careful reading of the
paper and constructive comments.

\section{Measures of non-convexity}
\label{sec:meas}

\subsection{Preliminaries and Definitions}
\label{sec:defns}

%Cassels \cite{Cas75} gave an elegant probabilistic interpretation of 
%the Shapley-Folkman-Starr theorem, by introducing and using the parameter $v$ instead of $r$,
%and he also observed that $d(A)\leq v(A)\leq r(A)$.
%The relationship between the various parameters was fully settled by Wegmann \cite{Weg80},
%who showed that $v(A)=r(A)$ for arbitrary sets $A$, and moreover that $d(A)=v(A)=r(A)$ when the supremum
%in the definition of $v$ is achieved at a point in the relative interior of $\conv(A)$.
%%see also \cite{DF04}
%%Cassels version also has the (minor) advantage of not requiring compactness in the formulation.

Throughout this paper, we only deal with compact sets, since several of the measures
of non-convexity we consider can have rather unpleasant behavior if we do not make this assumption.

The convex hull operation interacts nicely with Minkowski summation.

\begin{lem}\label{lem:conv-sum}
Let $A,B$ be nonempty subsets of $\R^n$. Then,
$$ \conv(A+B)=\conv(A)+\conv(B). $$
\end{lem}

\begin{proof}
Let $x \in \conv(A) + \conv(B) $. Then $x = \sum_{i=1}^k \lambda_i a_i + \sum_{j=1}^l \mu_j b_j$, where $a_i \in A$, $b_j \in B$, $\lambda_i \geq 0$, $\mu_j \geq 0$ and $\sum_{i=1}^k \lambda_i=1$, $\sum_{j=1}^l \mu_j=1$. Thus, $x = \sum_{i=1}^k \sum_{j=1}^l \lambda_i \mu_j (a_i + b_j)$. Hence $x \in \conv(A+B)$. The other inclusion is clear.
\end{proof}

Lemma \ref{lem:conv-sum} will be used throughout the paper without necessarily referring to it. 
A useful consequence of Lemma  \ref{lem:conv-sum} is the following remark.

\begin{rem}\label{rk:referee}
If $A+\lambda\conv(A)$ is convex then 
\[
A+\lambda\conv(A)=\conv(A+\lambda\conv(A))=\conv(A)+\lambda\conv(A)=(1+\lambda)\conv(A).
\]
\end{rem}

The Shapley-Folkman lemma, which is closely related to the classical Carath\'eodory theorem, is key to our development.

\begin{lem}[Shapley-Folkman]\label{lem:SF}

Let $A_1, \dots, A_k$ be nonempty subsets of $\R^n$, with $k \geq n + 1$. Let $a\in \sum_{i \in [k]} \conv(A_i)$. Then there exists a set $I$ of cardinality at most $n$ such that
$$  a\in \sum_{i \in I} \conv(A_i) + \sum_{i \in [k] \setminus I} A_i. $$

\end{lem}

\begin{proof}
We present below a proof taken from Proposition 5.7.1 of \cite{Ber09:book}.
Let $a \in \sum_{i \in [k]} \conv(A_i)$. Then
$$ a = \sum_{i \in [k]} a_i =  \sum_{i \in [k]}\sum_{j=1}^{t_i} \lambda_{ij} a_{ij}, $$
where $\lambda_{ij} \geq 0$, $\sum_{j=1}^{t_i} \lambda_{ij} = 1$, and $a_{ij} \in A_i$. Let us consider the following vectors of $\R^{n+k}$,
\begin{eqnarray*}
z & = & (a, 1, \cdots, 1), \\ z_{1j} & = & (a_{1j}, 1, 0, \cdots, 0), \quad j \in [t_1], \\ & \vdots & \\ z_{kj} & = & (a_{kj}, 0, \cdots, 0, 1), \quad j \in [t_k].
\end{eqnarray*}
Notice that $z = \sum_{i=1}^k \sum_{j=1}^{t_i} \lambda_{ij} z_{ij}$. Using Carath\'eodory's theorem in the positive cone generated by $z_{ij}$ in $\R^{n+k}$, one has
$$ z = \sum_{i=1}^k \sum_{j=1}^{t_i} \mu_{ij} z_{ij}, $$
for some nonnegative scalars $\mu_{ij}$ where at most $n+k$ of them are non zero. This implies that $a = \sum_{i=1}^k \sum_{j=1}^{t_i} \mu_{ij} a_{ij}$ and that $\sum_{j=1}^{t_i} \mu_{ij} = 1$, for all $i \in [k]$. Thus for each $i \in [k]$, there exists $j_i \in [t_i]$ such that $\mu_{ij_i} > 0$. But at most $n+k$ scalars $\mu_{ij}$ are positive. Hence there are at most $n$ additional $\mu_{ij}$ that are positive. One deduces that there are at least $k-n$ indices $i$ such that $\mu_{i\ell_i}=1$ for some $\ell_i \in [t_i]$, and thus $\mu_{ij} = 0$ for $j \neq \ell_i$. For these indices, one has $a_i \in A_i$. The other inclusion is clear.
\end{proof}

The Shapley-Folkman lemma may alternatively be written as the statement that, for $k \geq n + 1$,
\be
\conv(\sum_{i \in [k]} A_i) = \bigcup_{I \subset [k]: |I| \leq n} \bigg[ \sum_{i \in I} \conv(A_i) + \sum_{i \in [k] \setminus I} A_i \bigg],
\ee
where $|I|$ denotes the cardinality of $I$.
When all the sets involved are identical, and $k >n$, this reduces to the identity
\be
k\,\conv(A) = n\, \conv(A) + (k- n) \,A(k-n).
\ee

It should be noted that the Shapley-Folkman lemma is in the center of a rich vein of investigation in convex analysis
and its applications. As explained by Z.~Artstein \cite{Art80}, It may be seen as a discrete manifestation of a key lemma about extreme points that is related
to a number of ``bang-bang'' type results. It also plays an important role in the theory of vector-valued measures;
for example, it can be used as an ingredient in the proof of Lyapunov's theorem on the range of vector measures (see \cite{KR13},  \cite{DU77:book}
and references therein).

For a compact set $A$ in $\R^n$, denote by
$$R(A)=\min_{x}\{r>0: A\subset x+rB_2^n\}$$ 
the radius of the smallest ball containing $A$. By Jung's theorem \cite{Jun01}, this parameter is close to the diameter, namely one has 
\ben
\frac{\diam(A)}{2}\le R(A)\le \diam(A)\sqrt{\frac{n}{2(n+1)}}\le\frac{\diam(A)}{\sqrt{2}},
\een
where $\diam(A) = \sup_{x,y \in A} |x-y|$ is the Euclidean diameter of $A$. We also denote by 
\ben
\inr(A)=\max_{x}\{r\ge0: x+rB_2^n\subset A\}
\een
the inradius of $A$, i.e. the radius of a largest Euclidean ball included in $A$. 
There are several ways of measuring non-convexity of a set:
\begin{enumerate}
\item The Hausdorff distance from the convex hull is perhaps the most obvious measure to consider:
\ben
d(A)= d_H(A, \conv(A))= \inf\{r>0: \conv(A)\subset A+rB_2^n\} .
\een
A variant of this is to consider the Hausdorff distance when the ambient metric space is  $\R^n$ equipped with a norm different from the Euclidean norm. If $K$ is the closed unit ball of this norm
(i.e., any symmetric\footnote{We always use ``symmetric'' to mean centrally symmetric, i.e., $x\in K$ if and only if $-x\in K$.}, compact, convex set with nonempty interior), we define
\be\label{eq:dK}
d^{(K)}(A)=  \inf\{r>0: \conv(A)\subset A+rK\} .
\ee
In fact, the quantity \eqref{eq:dK} makes sense for any compact convex set containing 0 in its interior --
then it is sometimes called the Hausdorff distance with respect to a ``nonsymmetric norm''.
\item Another natural measure of non-convexity is the ``volume deficit'':
\ben
\Delta(A)=\vol_n(\conv(A)\setminus A) = \vol_n(\conv(A))- \vol_n(A).
\een
Of course, this notion is interesting only when $\vol_n(\conv(A))\not=0$.
There are many variants of this that one could consider, such as
$\log\vol_n(\conv(A))- \log\vol_n(A)$, or relative versions such as
$\Delta(A)/\vol_n(\conv(A))$ that are automatically bounded.
\item The ``inner radius'' of a compact set was defined by Starr \cite{Sta69} as follows: 
\ben
r(A) = \sup_{x\in\conv(A)} \inf \{R(T): T\subset A, x\in \conv(T) \} .
\een
\item The ``effective standard deviation'' was defined by Cassels \cite{Cas75}. For a random vector $X$ in $\R^n$, let $V(X)$ be the trace of its covariance matrix. Then the effective standard deviation of a compact set $A$ of $\R^n$ is
\ben
v^2(A) = \sup_{x\in\conv(A)} \inf \{ V(X) :  \supp(X)\subset A, |\supp(X)|<\infty, \E X=x \}.
\een
Let us notice the equivalent geometric definition of $v$:
\ben
v^2(A) &=&  \sup_{x\in\conv(A)}\inf\{\sum p_i |a_i-x|^2: x=\sum p_i a_i; p_i >0; \sum p_i=1, a_i \in A  \}\\
&=& \sup_{x\in\conv(A)}\inf\{\sum p_i |a_i|^2-|x|^2: x=\sum p_i a_i; p_i >0; \sum p_i=1, a_i \in A  \}.
\een
\item In analogy with the effective standard deviation, we define the ``effective absolute deviation'' by 
\ben
w(A) &=&  \sup_{x\in\conv(A)}\inf\bigg\{\sum p_i |a_i-x|: \, x=\sum p_i a_i; p_i >0; \sum p_i=1, a_i \in A  \bigg\}\\
&=& \sup_{x\in\conv(A)} \inf \{ \E|X-x| :  \supp(X)\subset A, |\supp(X)|<\infty, \E X=x \}.
\een

\item Another non-convexity measure was defined by Cassels \cite{Cas75} as follows:
\ben
\rho(A) = \sup_{x \in \conv(A)} \inf_{a \in A_x} |x-a|,
%=\sup_{x\in\conv(A)} \inf_{T\subset A, x\in \conv(T), } R(T) .
\een
where $A_x = \{a \in A : \exists b \in \conv(A), \exists \theta \in (0,1) \mbox{ such that } x = (1-\theta)a + \theta b\}$.
\item The ``non-convexity index'' was defined by Schneider \cite{Sch75} as follows: 
\ben
c(A) = \inf \{ \lam\geq 0: A+\lam\, \conv(A) \text{ is convex}  \}.
\een
%Observe that this may be thought of as a self-normalized version of the Hausdorff ``distance''
%corresponding to a nonsymmetric norm, i.e., $c(A)=d^{(\conv(A))}(A)$.
\end{enumerate}

\subsection{Basic properties of non-convexity measures}
\label{sec:properties}

All of these functionals are 0 when $A$ is a convex set; this justifies calling them ``measures of non-convexity''. 
In fact, we have the following stronger statement
since we restrict our attention to compact sets.

\begin{lem}\label{lem:meas0}
Let $A$ be a compact set in $\R^n$. Then:
\begin{enumerate}
\item $c(A)=0$ if and only if $A$ is convex.
\item $d(A)=0$ if and only if $A$ is convex.
\item $r(A)=0$ if and only if $A$ is convex.
\item $\rho(A)=0$ if and only if $A$ is convex.
\item $v(A)=0$ if and only if $A$ is convex.
\item $w(A)=0$ if and only if $A$ is convex.
\item Under the additional assumption that $\conv(A)$ has nonempty interior, $\Delta(A)=0$ if and only if $A$ is convex.
\end{enumerate}
\end{lem}

\begin{proof}  Directly from the definition of $c(A)$ we get that $c(A)=0$ if $A$ is convex  (just select $\lambda=0$). Now assume that $c(A)=0$, then  $\{A+\frac{1}{m} \conv(A)\}_{m=1}^\infty$
is a sequence of compact convex sets, converging in Hausdorff metric to $A$, thus $A$ must be convex. Notice that this observation is due to Schneider \cite{Sch75}.
 
 The assertion about $d(A)$ follows immediately from the definition and the limiting argument similar to the above one.
 
 If $A$ is convex then, clearly $r(A)=0$, indeed we can always take $T=(rB_2^n+x) \cap A \not =\emptyset$  with $r \to 0$. Next, if $r(A)=0$, then using Theorem \ref{thm:weg} below we have $d(A)\le r(A)=0$ thus $d(A)=0$ and therefore $A$ is convex.
 
 The statements about $\rho(A)$, $v(A)$ and $w(A)$ can be deduced from the definitions, but they will also follow immediately from the Theorem \ref{thm:weg} below.
 
 Assume that $A$ is convex, then $\conv(A)=A$ and $\Delta(A)=0$. Next, assume that $\Delta(A)=0$. Assume, towards a contradiction, that $\conv(A) \not =A$. Then there exists $x \in \conv(A)$ and $r>0$ such that $(x + r B_2^n) \cap A =\emptyset$. Since $\conv(A)$ is convex and has nonempty interior, there exists a ball $y+sB_2^n\subset\conv(A)$ and one has 
 $$
\Delta(A)\ge \vol_n(\conv(A) \cap (x + r B_2^n))\ge \vol_n(\conv(x, y+sB_2^n) \cap (x + r B_2^n)) >0,
 $$ 
 which contradicts $\Delta(A)=0$.
\end{proof}

The following lemmata capture some basic properties of all these measures of non-convexity
(note that we need not separately discuss $v$, $w$ and $\rho$ henceforth owing to Theorem~\ref{thm:weg}).
The first lemma concerns the behavior of these functionals on scaling of the argument set.

\begin{lem}\label{lem:scaling}
Let $A$ be a compact subset of $\R^n$, $x \in \R^n$, and $\lam\in (0,\infty)$.
\begin{enumerate}
\item $c(\lam A + x)= c(A)$. In fact, $c$ is affine-invariant.
\item $d(\lam A + x)= \lam  d(A)$.
\item $r(\lam A + x)=\lam  r(A)$.
\item $\Delta(\lam A + x)= \lam^n \Delta(A)$. In fact, if $T(x)=Mx+b$, where
$M$ is an invertible linear transformation and $b \in \R^n$, then $\Delta(T(A))=|\det(M)| \Delta(A)$.
\end{enumerate}
\end{lem}

\begin{proof}
To see that $c$ is affine-invariant, we first notice that $\conv(TA)=T\conv(A)$. Moreover writing $Tx=Mx+b$, where
$M$ is an invertible linear transformation and $b \in \R^n$, we get that 
$$
TA+ \lam \conv(TA)=M(A+\lam \conv(A))+(1+\lam)b,
$$
which is convex if and only if $A+\lam \conv(A)$ is convex. 

It is easy to see from the definitions that $d$, $r$ and $\Delta$ are translation-invariant, and that $d$ and $r$ are 1-homogeneous and $\Delta$ is $n$-homogeneous with respect to dilation.
\end{proof}

The next lemma concerns the monotonicity of non-convexity measures with respect to the inclusion relation.

\begin{lem}\label{lem:dincl}
Let $A, B$ be compact sets in $\R^n$ such that $A\subset B$ and  $\conv(A)=\conv(B)$. Then:
\begin{enumerate}
\item $c(A)\geq c(B)$.
\item $d(A)\geq d(B)$.
\item $r(A)\geq r(B)$.
\item $\Delta(A)\geq \Delta(B)$.
\end{enumerate}
\end{lem}

\begin{proof}
For the first part, observe that if $\lam=c(A)$,
\ben
(1+\lam)\conv(B) \supset B+\lam\conv(B) = B+\lam\conv(A) \supset A+\lam\conv(A) = (1+\lam)\conv(B),
\een
where in the last equation we used that $A + \lambda \conv(A)$ is convex and Remark \ref{rk:referee}. 
Hence all relations in the above display must be equalities, and $B+\lam\conv(B)$ must be convex, which means $c(A) = \lambda \geq c(B)$.

For the second part, observe that 
\ben
d(A)=\sup_{x\in \conv(A)}d(x,A)=\sup_{x\in \conv(B)}d(x,A)\ge\sup_{x\in \conv(B)}d(x,B)=d(B).
\een

For the third part, observe that
\ben
\inf \{R(T): T\subset A, x\in \conv(T) \} \geq \inf \{R(T): T\subset B, x\in \conv(T) \}.
\een
Hence $r(A) \geq r(B)$.

For the fourth part, observe that
\ben
\Delta(A) = \vol_n(\conv(B)) - \vol_n(A) \geq \vol_n(\conv(B)) - \vol_n(B) = \Delta(B).
\een
\end{proof}

As a consequence of Lemma \ref{lem:dincl}, we deduce that $A(k)$ is monotone along the 
subsequence of powers of 2, when measured through all these measures of non-convexity.

Finally we discuss topological aspects of these non-convexity functionals,
specifically, whether they have continuity properties with respect to the topology on the 
class of compact sets induced by Hausdorff distance.

\begin{lem}\label{lem:semicont}
Suppose $A_k\rad A$, where all the sets involved are compact subsets of $\R^n$. Then:
\begin{enumerate}
\item $\lim_{k\ra\infty} d(A_k)= d(A)$, i.e., $d$ is continuous.
\item $\liminf_{k\ra\infty} \Delta(A_k)\geq \Delta(A)$, i.e., $\Delta$ is lower semicontinuous.
\item $\liminf_{k\ra\infty}  c(A_k) \geq c(A)$, i.e., $c$ is lower semicontinuous. 
\item $\liminf_{k\ra\infty} r(A_k) \geq r(A)$, i.e., $r$ is lower semicontinuous. 
\end{enumerate}
\end{lem}

\begin{proof}
%Let us first observe that $A_k\rad A$ implies $\conv(A_k)\rad \conv(A)$. Indeed, just by
%applying the convex hull operation to the inclusions
%$A_k \subset A+ \eps B_2^n$ and $A \subset A_k+ \eps B_2^n$,
%and invoking Lemma~\ref{lem:conv-sum} yields this implication.
%
Let us first observe that for any compact sets $A,B$
\be\label{eq:lip} 
d_H(\conv(A),\conv(B))\le d_H(A,B),
\ee
by applying the convex hull operation to the inclusions
$B \subset A+ d B_2^n$ and $A \subset B+ d B_2^n$,
and invoking Lemma~\ref{lem:conv-sum}. 
Thus $A_k\rad A$ implies $\conv(A_k)\rad \conv(A)$.
\\

%For the first part, 
1. Observe that by the triangle inequality for the Hausdorff metric, we have the inequality
\ben
d(B) = d_H(B, \conv(B)) 
\leq d_H(B, A) + d_H(A, \conv(A)) + d_H(\conv(A), \conv(B)). 
\een
Using \eqref{eq:lip} one deduces that $d(B)-d(A) \leq 2d_H(B, A)$.
Changing the role of $A$ and $B$, we get
\ben
|d(B)-d(A)| \leq 2d_H(B, A).
\een
This proves the continuity of $d$.

%For the second part, 
2. Recall that, with respect to the Hausdorff distance, the volume is upper semicontinuous on the class of compact sets  (see, e.g., \cite[Theorem 12.3.6]{SW08:book})
and continuous on the class of compact convex sets  (see, e.g., \cite[Theorem 1.8.20]{Sch14:book}). Thus
\ben
\limsup_{k\ra\infty} \vol_n(A_k) \leq \vol_n(A)
\een
and
\ben
\lim_{k\ra\infty} \vol_n(\conv(A_k)) = \vol_n(\conv(A)) ,
\een
so that subtracting the former from the latter yields the desired semicontinuity of $\Delta$.

%For the third part, 
3. Observe that by definition,
\ben
A_k+\lam_k \conv(A_k)= (1+\lam_k) \conv(A_k) ,
\een
where $\lam_k=c(A_k)$.
Note that from Theorem \ref{thm:sch75} below due to Schneider \cite{Sch75} one has $\lam_k \in [0, n]$, thus there exists a convergent subsequence $\lambda_{k_n} \to \lam_*$ and 
\ben
A+\lam_* \conv(A)= (1+\lam_*) \conv(A) ,
\een
 Thus $\lam_* \ge c(A)$, which is the desired semicontinuity of $c$.

%Next we will study convergence of $r(A_k)$. 
4. Using $A_k\rad A$ we get that $R(A_k)$ is bounded and thus $r(A_k)$ is bounded and there is a convergent subsequence $r(A_{k_m}) \to l$.
Our goal is to show that $r(A) \le l$. Let $x \in \conv(A)$. Then there exits $x_m \in A_{k_m}$ such that $x_m \to x$. From the definition of $r(A_{k_m})$ we get that there exists $T_m \subset A_{k_m}$ such that $x_m \in \conv(T_m)$ and $R(T_m) \le r(A_{k_m})$. We can select a convergent subsequence  $T_{m_i} \to T$, where $T$ is compact (see \cite[Theorem  1.8.4]{Sch14:book}), then $T\subset A$ and $x \in \conv (T)$ and $R(T_{m_i}) \to R(T)$ therefore $R(T)\le l$. Thus $r(A) \le l$.
\end{proof}

We emphasize that the semicontinuity assertions in Lemma~\ref{lem:semicont} are not
continuity assertions for a reason and even adding the assumption of nestedness of the sets would not help. 

\begin{ex}\label{ex:sch}
Schneider \cite{Sch75}  observed that $c$ is not continuous with respect to the Hausdorff distance, 
even if restricted to the compact sets with nonempty interior. His example consists of taking a triangle in the plane,
and replacing one of its edges by the two segments which join the endpoints of the edge to an interior point (see Figure \ref{fig:sch}). 
More precisely, let $a_k=(\frac{1}{2}-\frac{1}{k},\frac{1}{2}-\frac{1}{k})$, $A_k=\conv((0,0); (1,0); a_k)\cup \conv((0,0); (0,1); a_k)$, and $A= \conv((0,0) ; (0,1); (1,0))=\conv(A_k)$. Then $d_H(A_k,A)\to 0$. But one has $r(A)=c(A)=0$ since $A$ is convex. Moreover one can notice that $c(A_k)=1$. Indeed on one hand $A\subset \frac{A+A_k}{2}$, which implies that $c(A_k)\le1$, on the other hand for every $\lambda<1$ the point $(\frac{1}{2},\frac{1}{2})\in A\setminus\frac{A_k+\lambda A}{1+\lambda}$, thus $c(A_k)=1$. Notice also that $r(A_k)=1/\sqrt{2}$. Indeed $A_k\subset (\frac{1}{2},\frac{1}{2})+\frac{1}{\sqrt{2}}B_2^2$ hence $r(A_k)\le\frac{1}{\sqrt{2}}$ and the opposite inequality is not difficult to see since the supremum in the definition of $r$ is attained at the point $(\frac{1}{2},\frac{1}{2})$.
\end{ex}

\begin{figure}[h!]
\begin{center}
\includegraphics[scale=0.4]{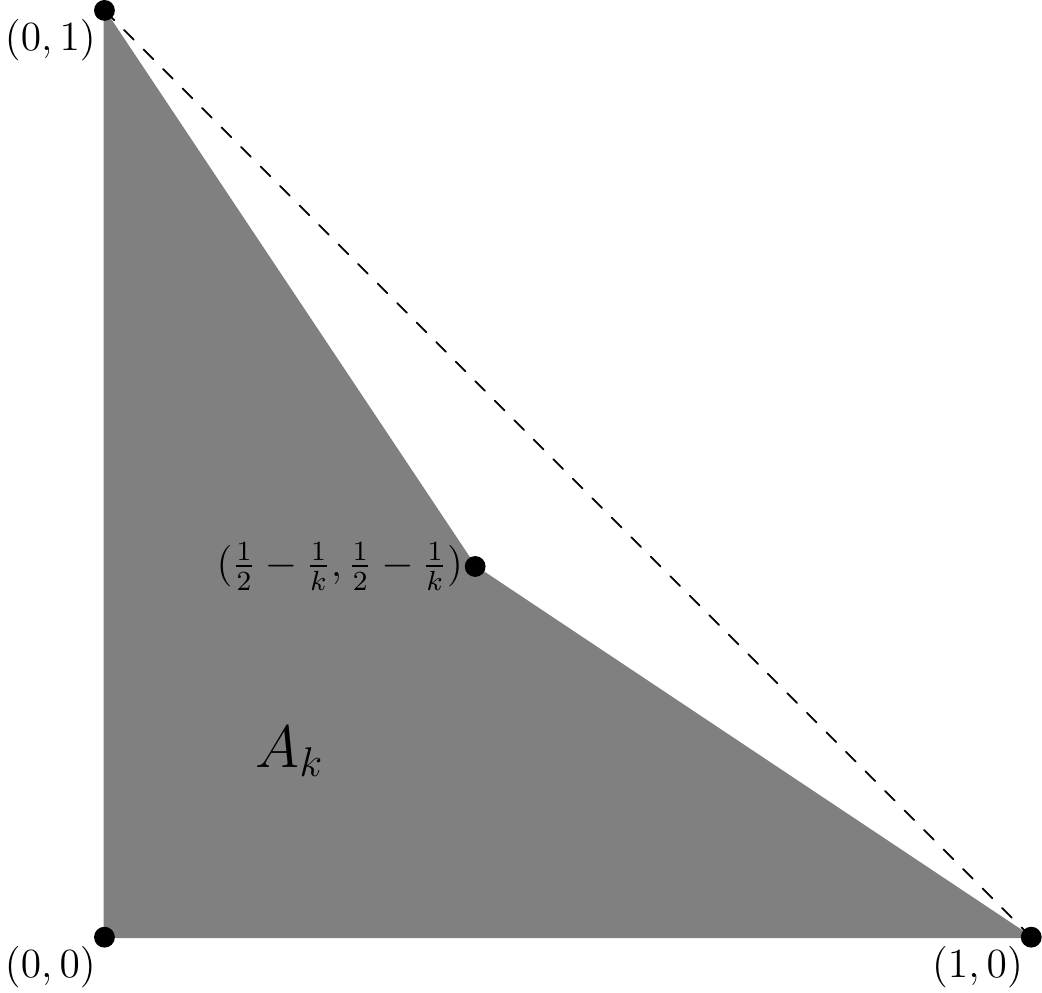}
\end{center}
\caption{Discontinuity of $c$ and $r$ with respect to Hausdorff distance (Example \ref{ex:sch}). }\label{fig:sch}
\end{figure}

\begin{ex}\label{ex:dis}
To see that there is no continuity for $\Delta$, consider a sequence of discrete nested sets converging in $d$ to $[0,1]$, more precisely: $A_k=\{\frac{m}{2^k}; 0\le m\le 2^k\}$.
\end{ex}

%For example, if one considers the nested sequence of sets
%\ben
%A_k= \{0\} \cup \bigg( \bigg[\frac{1}{k}, 1\bigg]\times [-1,1] \bigg), 
%\een
%which tends to $A=[0,1]\times [-1,1]$ in the Hausdorff distance,
%then $\limsup_{k\ra\infty} r(A_k) = \half > 0=r(A)$. 
%Another example is triangles with picks, i.e.,
%$A_k=(0,0),(0,1)) \cup ((0,0), (1,0)) \cup \conv ((0,0),(0,1-1/k)) \cup ((0,0), (1-1/k,0))$.

\subsection{Special properties of Schneider's index}
\label{sec:schneider}

All these functionals other than $c$ can be unbounded. The boundedness of $c$
follows for the following nice inequality due to Schneider \cite{Sch75}.

\begin{thm}\label{thm:sch75}\cite{Sch75}
For any subset $A$ of $\R^n$,
\ben
c(A) \leq n.
\een
\end{thm}

\begin{proof}
Applying the Shapley-Folkman lemma (Lemma \ref{lem:SF}) to $A_1 = \cdots = A_{n+1} = A$, where $A \subset \R^n$ is a fixed compact set, one deduces that $(n+1)\conv(A) = A + n\conv(A)$. Thus $c(A) \leq n$.
%We want to prove that $A+n\conv(A)$ is convex. We only need to prove that $(n+1)\conv(A)\subset n\conv(A)+A$. Let $x\in (n+1)\conv(A)$. From 
%Carath\'eodory's theorem, there exists $a_1,\dots, a_{n+1}\in A$ and $\lam_1, \dots, \lam_{n+1}\ge0$ such that $x=(n+1)\sum_{i=1}^{n+1}\lam_ia_i$ with $\sum_{i=1}^{n+1}\lam_i=1$. Since $1=\sum_{i=1}^{n+1}\lam_i\ge (n+1)\min_i\lam_i$, there exists $i_0$ such that $\lam_{i_0}\ge \frac{1}{n+1}$. Define $\mu_{i_0}=\frac{(n+1)\lam_{i_0}-1}{n}$ and $\mu_i=\frac{n+1}{n}\lam_i$ for $i\neq i_0$. Then $\mu_i\ge0$ for all $i$, $\sum_i\mu_i=1$ and 
%$$ x=a_{i_0}+ ((n+1)\lam_{i_0}-1)a_{i_0}+(n+1)\sum_{i\neq i_0}\lam_ia_i=a_{i_0}+n\sum_i\mu_ia_i\in A+n\conv(A).$$
%{\color{red} ADD PROOF OF EQUALITY CASE? ALSO SHOULD WE REPLACE ABOVE PROOF BY ONE-LINE PROOF USING Shapley-Folkman lemma?}
\end{proof}

Schneider \cite{Sch75} showed that $c(A)=n$ if and only if $A$ consists of $n+1$ affinely independent points. Schneider also showed that if $A$ is unbounded or connected, one has the sharp bound $c(A)\leq n-1$.

Let us note some alternative representations of Schneider's non-convexity index. First, we would like to remind the definition of the {\it Minkowski functional} of a compact convex set $K$ containing zero:
$$
\|x\|_K=\inf\{t > 0: x \in  tK\},
$$
with the usual convention that $\|x\|_K=+\infty$ if $\{t > 0: x \in  tK\}=\emptyset$. Note that $K=\{x\in \R^n: \|x\|_K\le 1\}$ and  $\|x\|_K$ is a norm if $K$ is symmetric with non empty interior.

For any compact set $A\subset\R^n$, define
\ben
A_\lam= \frac{1}{1+\lam} [A+\lam\, \conv(A)],
\een
and observe that
\ben
\conv(A_\lam) = \frac{1}{1+\lam} \conv(A+\lam\, \conv(A)) = \frac{1}{1+\lam} [\conv(A) + \lambda \conv(A)] = \conv(A).
\een
Hence, we can express
\be\label{eq:c-alam}
c(A)= \inf \{ \lam\geq 0: A_\lam \text{ is convex}  \} = \inf \{ \lam\geq 0: A_\lam = \conv(A) \} .
\ee

Rewriting this yet another way, we see that if $c(A)<t$, then for each $x\in\conv(A)$, there
exists $a\in A$ and $b\in\conv(A)$ such that
\ben
x=\frac{a+tb}{1+t} ,
\een
or equivalently, $x-a=t(b-x)$. In other words, 
$x-a\in t K_x$ where $K_x=\conv(A)-x$,
which can be written as $\|x-a\|_{K_x} \leq t$ using the Minkowski functional.
Thus
\ben
c(A)=\sup_{x\in\conv(A)} \inf_{a\in A} \|x-a\|_{K_x} .
%= \sup_{x\in\conv(A)} d_{K_x} (x, A) .
\een
This representation is nice since it allows for comparison with the representation of $d(A)$
in the same form but with $K_x$ replaced by the Euclidean unit ball.

\begin{rem}
 Schneider \cite{Sch75} observed that there are many closed {\it unbounded} sets $A\subset \R^n$ that satisfy $c(A)=0$, 
but are not convex. Examples he gave include the set of integers in $\R$, or a parabola in the plane. 
This makes it very clear that if we are to use $c$ as a measure of non-convexity, we should
restrict attention to compact sets.
\end{rem}

\subsection{Unconditional relationships}
\label{sec:uncond}

It is natural to ask how these various measures of non-convexity are related. First we note that $d$ and $d^{(K)}$ are equivalent. To prove this we would like to present an elementary but useful observation:

\begin{lem}\label{lem:d^K-d^L} 
Let $K\subset \R^n$ be an arbitrary convex body containing $0$ in its interior. 
Consider a convex body $L \subset \R^n$ such that $K \subset L$ and $t>0$. Then  for any compact set $A \subset \R^n$, 
$$
d^{(K)} (A) \geq d^{(L)} (A)
$$
and
$$
d^{(tK)}(A)=\frac{1}{t}d^{(K)}(A).
$$
\end{lem}
\begin{proof}
Notice that
$$ A+d^{(K)}(A) L \supset A+d^{(K)}(A)K \supset \conv(A). $$
Hence, $d^{(K)} (A) \geq d^{(L)} (A)$. In addition, one has
$$ A + d^{(K)}(A) K = A + \frac{1}{t}d^{(K)}(A) t K. $$
Hence,  $d^{(tK)}(A)=\frac{1}{t}d^{(K)}(A)$.
\end{proof}

The next lemma follows immediately from  Lemma \ref{lem:d^K-d^L}:

\begin{lem}\label{lem:d-d^K}
Let $K$ be an arbitrary convex body containing 0 in its interior. For any compact set $A \subset \R^n$, one has
$$ r d^{(K)}(A) \leq d(A) \leq R d^{(K)}(A), $$
where $r,R>0$ are such that $rB_2^n \subset K \subset RB_2^n$.
\end{lem}

%\begin{proof}
%Notice that
%$$ A+d^{(K)}(A)RB_2^n \supset A+d^{(K)}(A)K \supset \conv(A). $$
%Hence, $d(A) \leq d^{(K)}(A)R$. On the other hand, one has
%$$ A + \frac{d(A)}{r} K \supset A + d(A)B_2^n \supset \conv(A). $$
%Hence, $d^{(K)}(A) \leq d(A)/r$.
%\end{proof}

It is also interesting to note a special property of $d^{(\conv(A))}(A)$:

\begin{lem}\label{lem:d^conv}
Let $A$ be a compact set in $\R^n$. If $0 \in \conv(A)$, then
$$ d^{(\conv(A))}(A)\le c(A). $$
If $0 \in A$, then
$$ d^{(\conv(A))}(A)\le \min\{1, c(A)\}. $$
\end{lem}

\begin{proof}
%Note that $0 \in \conv(A)$, thus $\conv(A) \subset (1+c(A)) \conv(A)$, but
%$$
%(1+c(A)) \conv(A) = A+c(A)\conv(A),
%$$
%where we used the fact that by definition of $c(A)$, $A + c(A) \conv(A)$ is convex.
If $0 \in \conv(A)$, then $\conv(A) \subset (1+c(A)) \conv(A)$. But,
$$
(1+c(A)) \conv(A) = A+c(A)\conv(A),
$$
where we used the fact that by definition of $c(A)$, $A + c(A) \conv(A)$ is convex. Hence, $d^{(\conv(A))}(A)\le c(A)$.

If $0 \in A$, in addition to the above argument, we also have
$$ \conv(A) \subset A + \conv(A). $$
Hence, $d^{(\conv(A))}(A)\le 1$.
\end{proof}

Note that the inequality in the above lemma cannot be reversed even with the cost of an additional multiplicative constant. Indeed, take the sets $A_k$ from Example  \ref{ex:sch},   then $c(A_k)=1$ but $d^{(\conv(A_k))}(A_k)$ tends to $0$.

Observe that $d, r, \rho$ and $v$ have some similarity in definition. Let us introduce the point-wise definitions of above notions: Consider $x\in \conv(A)$, define
\begin{itemize}
\item $d_A(x)=\inf\limits_{a\in A} |x-a|.$ \\
More generally, if $K$ is a compact convex set in $\R^n$ containing the origin,
\item $d^{(K)}_A(x)=\inf\limits_{a\in A} \|x-a\|_K.$
\item $r_A(x) = \inf \{R(T): T\subset A, x\in \conv(T) \} .$
\item $v^2_A (x)=  \inf\{\sum p_i |a_i|^2-|x|^2: x=\sum p_i a_i; p_i >0; \sum p_i=1, a_i \in A  \}.$
\item $w_A(x) =\inf\bigg\{\sum p_i |a_i-x|: \, x=\sum p_i a_i; p_i >0; \sum p_i=1, a_i \in A  \bigg\}$.
\item $\rho_A(x) =  \inf_{a \in A_x} |x-a|, $
where $$A_x = \{a \in A : \exists b \in \conv(A), \exists \theta \in (0,1) \mbox{ such that } x = (1-\theta)a + \theta b\}.$$
\end{itemize}

Below we present a Theorem due to Wegmann \cite{Weg80} which proves that $r, \rho$ and $v$ are equal for compact sets and that they are equal also to $d$ under an additional assumption. For the sake of completeness we will present the proof of  Wegmann \cite{Weg80}  which is simplified here for the case of compact sets.

%This was essentially shown by Wegmann \cite{Weg80}.

\begin{thm}[Wegmann \cite{Weg80}]\label{thm:weg}
Let $A$ be a compact set in $\R^n$, then
\ben
d(A)\le \rho(A) =w(A)= v(A)= r(A).
\een
Moreover if $v_A(x_0)=v(A)$, for some  $x_0$ in the relative interior of $\conv(A)$, then $d(A)=v(A)=w(A)=r(A)=\rho(A)$.
\end{thm}

\begin{proof}
1) First observe that $ d(A)\le \rho(A) \leq w(A)\leq v(A)\leq r(A)$ by easy arguments; in fact, this relation holds point-wise, i.e. $d_A(x) \le \rho_A(x)\le w_A(x)\leq v_A(x)\le r_A(x)$. 

Indeed the first inequality  follows directly from the definitions, because $A_x \subset A$. 

To prove the second inequality consider any convex decomposition of $x\in \conv(A)$, i.e. $x=\sum p_i a_i$,
with $p_i >0, \sum p_i=1, a_i \in A$.
Without loss of generality we may assume that $|x-a_1|\le |x-a_i|$ for all $i\le m$. Then
$$
\sum p_i |x-a_i| \ge  |x-a_1| \ge \rho_A(x),
$$
because $a_1 \in A_x$ (indeed, $x=p_1 a_1 + (1-p_1) \sum\limits_{i\ge 2} \frac{p_i}{1-p_1}  a_i$). 

The third inequality $w_A(x)\leq v_A(x)$ immediately follows from the Cauchy-Schwarz inequality.

To prove the fourth inequality let $T=\{a_1,\dots, a_m\} \subset A$  be such that   $x\in \conv(T)$. Let $p_1,\dots,p_m>0$ be such that $\sum p_i=1$ and $x=\sum p_ia_i$.  Let $c$ be the center of the smallest Euclidean ball containing $T$. Notice that  the minimum of $\sum p_i |x-a_i|^2$ is reached for $x=\sum p_i a_i$,  thus
$$
v_A^2(x)\le \sum p_i |x-a_i|^2 \le \sum p_i |c-a_i|^2 \le R^2(T),
$$
and we take infimum over all $T$ to finish the proof of the inequality.\\

2) Consider $x_0\in\conv(A)$. To prove the theorem we will first show that  $r_A(x_0) \le v(A)$. After this we will show that $v_A(x_0) \le \rho(A)$ and finally we will prove if $x_0$ is in the relative interior of $\conv(A)$ and maximizes $v_A(x)$, among $x\in\conv(A)$ then $d_A(x_0) \ge v(A)$.  \\

2.1) Let us prove that $r_A(x_0) \le v(A)$. Assume first that $x_0$ is an interior point of $\conv (A)$. 
Let us define the compact convex set $Q \subset \R^{n+1}$ by
$$
Q=\conv\{ (a,|a|^2); a\in A\}.
$$
Next we define the function $f: \conv(A) \to \R^+$ by $f(x)=\min\{y:  (x,y) \in Q\}$, note that
\ben
f(x)&=&\min\{y: (x,y)=\sum\lambda_i (a_i, |a_i|^2); \lambda_1,\dots,\lambda_m >0 \mbox{ and   }  a_1, \dots, a_m \in A\}\\
&=&\min\{\sum\lambda_i|a_i|^2: \lambda_1,\dots,\lambda_m >0 \mbox{ and   }  a_1, \dots, a_m \in A, \sum \lambda_i=1;  x=\sum\lambda_i a_i\}\\
&=&v^2_A(x)+|x|^2.
\een
Note that $(x_0, f(x_0))$ is a boundary point of $Q$ hence there exists a support hyperplane $H$ of $Q$ at $(x_0, f(x_0))$. Since $x_0$ is an interior point of $\conv (A)$, the hyperplane $H$ cannot be vertical because a vertical support plane would separate $x_0$ from boundary points of $\conv (A)$ and thus separate $(x_0, f(x_0))$ from boundary points of $Q$. Thus there exist $b \in \R^n$ and $\alpha \in \R$ such that $H=\{(x,t)\in\R^{n+1}: t=2 \langle b, x\rangle +\alpha\}$. Since $(x_0,f(x_0))\in H$ one has
\begin{equation}\label{eq:ba}
f(x_0) = 2 \langle b, x_0\rangle +\alpha
\end{equation}
and 
$$
f(x) \ge 2 \langle b, x\rangle +\alpha, \mbox{   for all   } x \in \conv (A).
$$
By definition of $f$, there exists $a_1,\dots, a_m\in A$ and
$\lambda_1, \dots, \lambda_m>0$, $\sum \lambda_i=1$ such that $x_0=\sum \lambda_i a_i$ and
$$
f(x_0)=\sum\lambda_i |a_i|^2 =\sum\lambda_i f(a_i).
$$
From the convexity of $Q$ we get that $(a_i, f(a_i))\in H\cap Q$, for any $i$; indeed we note that 
$$
f(x_0)=2 \langle b, x_0\rangle +\alpha=\sum_i\lambda_i(2 \langle b, a_i\rangle +\alpha)\le \sum_i \lambda_if(a_i)=f(x_0).
$$
Thus $2 \langle b, a_i\rangle +\alpha=f(a_i)$  for all $i$.
Let $T=\{a_1, \dots a_m\}$ and $W=\conv (T)$. Note that for any $x\in W \cap A$ we have
$$
|x|^2=f(x)=2\langle b, x \rangle +\alpha
$$
thus $\alpha+|b|^2 =|x-b|^2 \ge 0$. Define 
\be\label{def:weg-R}
R^2 = \alpha+|b|^2.
\ee 
Notice that for any $x\in \conv(A)$
\begin{equation}\label{eq:va}
v^2_A(x)=f(x)-|x|^2 \ge 2\langle b, x \rangle +\alpha - |x|^2=R^2-|b-x|^2,
\end{equation}
with equality if $x \in W$, in particular, $0\le v^2_A(x)=R^2-|b-x|^2 \le R^2$, for every $x\in W$. Consider the point $w \in W$ such that 
$$
v_A^2(w)=\max\limits_{x \in W} v^2_A(x) = \max\limits_{x \in W}(R^2-|b-x|^2) =R^2-\inf_{x\in W}|b-x|^2.
$$
Then one has $|b-w|=\inf_{x\in W}|b-x|$, which means $w$ is the projection of the point $b$ on the convex set $W$. This implies that, for every $x\in W$, one has
$\langle x-b,w-b\rangle\ge|w-b|^2$, thus
$$
|x-w|^2=|x-b|^2-2\langle x-b,w-b\rangle+|w-b|^2\le |x-b|^2-|w-b|^2\le R^2-|w-b|^2=v_A^2(w).
$$
%
%For a fixed $a\in W$ we define the function $\Phi: [0,1] \to \R^+$ by
%$$
%\Phi(\theta)=f((1-\theta) w +\theta a) -  |(1-\theta) w +\theta a|^2 =v_A^2((1-\theta)w+\theta a)
%$$
%Note that  $\Phi(0)\ge v_A^2(w)$.  Direct computations give
%$$
%\Phi(\theta)= -|a-w|^2 \theta^2 +2 \theta  \langle a-w,  b-w \rangle + v_A^2(w)
%$$
%Using that $\Phi$ has a maximal value at $0$, we get $\Phi'(0)= 2 \theta  \langle a-w,  b-w \rangle \le 0$ and thus
%$$
%0\le \Phi(1)=  -|a-w|^2  +2   \langle a-w,  b-w \rangle + v_A^2(w)\le  -|a-w|^2   + v_A^2(w).
%$$
%Thus for all $a\in W$ we get  $|a-w| \le v_A(w)$. 
We get $T\subset W \subset w+v_A(w)B_2^n$ and 
$$
R(T) \le v_A(w) = \max\limits_{x \in W} v_A(x).
$$
Using that $x_0\in W =\conv (T)$ and $T\subset A$, we conclude from the definition of $r_A$ that 
$$r_A(x_0)\le R(T)\le\max\limits_{x \in W} v_A(x)\le v(A).$$
If $x_0$ is a boundary point of $\conv (A)$, then using the boundary structure of the polytope $\conv(A)$ (see \cite[Theorem 2.1.2, p. 75 and Remark 3, p. 78]{Sch14:book}) 
$x_0$ belongs to the relative interior of an exposed face $F$ of $\conv( A)$.  By the definition of the notion of 
exposed face (see \cite[p. 75]{Sch14:book}) we get that if $x=\sum \lambda_i a_i$  for $a_i \in A$ and $\lambda_i>0$ with $\sum\lambda_i=1$, then $a_i \in A \cap F$. Thus
\begin{equation}\label{eq:face}
v_A(x_0)=v_{A \cap F}(x_0),  r_A(x_0)=r_{A \cap F}(x_0) \mbox{   and  }  \rho_A(x_0)=\rho_{A \cap F}(x_0).
\end{equation}
If $\mbox{dim}(F) =0$ then $x_0 \in A$ and thus all proposed inequalities are trivial, otherwise we can reproduce the above argument for $A \cap F$ instead of $A$.\\

2.2) Now we will prove that  $v_A(x_0) \le \rho(A)$.  Consider $b, \alpha$ and $R$ defined in \eqref{eq:ba} and \eqref{def:weg-R}. 
Using that $v_A(a)=0$, for every $a\in A$ and (\ref{eq:va}), we get $|b-a|\ge R$, for all $a\in A$. We will need to consider two cases
\begin{enumerate}
\item If $b\in \conv(A)$, then from the above $d_A(b)=\inf\limits_{a\in A} |b-a| \ge R$ thus
\begin{equation}\label{eq:vd}
 v_A(x_0) \le R\le d_A(b) \le \rho_A(b)\le \rho(A).
\end{equation}
\item If $b\not\in \conv(A)$, then there exists $y \in \partial(\conv(A)) \cap [w, b]$, thus $|b-y| \le |b-w|$. So, from (\ref{eq:va}) we have
$$
v_A^2(y) \ge R^2 -|b-y|^2 \ge R^2-|b-w|^2=v^2_A(w) \ge v^2_A(x_0),
$$
so it is enough to prove $v_A(y) \le \rho(A)$, where $y \in \partial(\conv(A))$.  Let $F$ be the face of $\conv(A)$ containing $y$ in its relative interior. Thus we can use the approach from (\ref{eq:face}) and reproduce the above argument for $A\cap F$ instead of $A$, in the end of which we will again get two cases (as above), in the first case we get $v_A(y)=v_{A\cap F}(y) \le \rho(A\cap F)\le \rho(A)$. In the second case, there exists $z \in \partial(\conv(A\cap F))$ such that $v_{A\cap F}(z) \ge  v_{A\cap F}(y)$ and we again reduce the dimension of the set under consideration. Repeating this argument we will arrive to the dimension $1$ in which the proof can be completed by verifying that $b \in \conv(A)$ (indeed, in this case $W=[a_1, a_2]$, $a_1, a_2 \in A$  and $|a_1 -b|=|a_2-b|$, thus $b=(a_1+a_2)/2 \in \conv(A)$) and thus  $v_A(x_0) \le \rho(A)$.
\end{enumerate}

2.3) Finally, assume $v_A(x_0)=v(A)$, where  $x_0$ is in the relative interior of $\conv(A)$. We may assume that $\conv(A)$ is $n$-dimensional (otherwise we would work in the affine subspace generated by $A$). Then using (\ref{eq:va}) we get that $v_A^2(x_0)=R^2- |b-x_0|^2$ and $v_A^2(a) \ge R^2- |b-a|^2,$ for all $a\in \conv(A)$, thus
$$
0 \le v_A^2(x_0)-v_A^2(a) \le |b-a|^2 - |b-x_0|^2, 
$$
for all $a \in \conv(A)$. So $|b-x_0| \le |b-a|$ for all $a \in \conv(A)$, this means that the  minimal distance  between $b$ and  $a \in \conv(A)$ is reached at $a=x_0$. Notice that if $b \not\in \conv(A)$  then $x_0$ must belong to $\partial(\conv(A))$, which contradicts our hypothesis. Thus $b \in \conv(A)$ and $x_0=b$, and we can use (\ref{eq:vd}) to conclude that
$v(A)=v_A(x_0) \le d_A(x_0) \le d(A)$.
\end{proof}

\begin{rem}\label{delaunay}
The method used in the proof of Theorem \ref{thm:weg} is reminiscent of the classical approach to Voronoi diagrams and Delaunay triangulation
(see, e.g., \cite[section 5.7]{Mat02:book}). Moreover the point $b$ constructed above is exactly the center of the ball  
circumscribed to the simplex of the Delaunay triangulation to which the point $x_0$ belongs.
\end{rem}

%%%%%%%%%%%%%%%%%%%%%%%%%%%%%%%%%%%%
%%%%%%%%%%%%%%%%%%%%%%%%%%%%%%%%%%%%NEW PROOF%%%%%%%%%%%%%%%%%%%%%%%%%%%%%
%%%%%%%%%%%%%%%%%%%%%%%%%%%%%%%%%%%%

Next we present a different proof of $r(A)=v(A)$ from Theorem \ref{thm:weg}, which essentially uses Remark \ref{delaunay}
and is more geometric. The proof will be deduced from the following proposition that better describes the geometric properties of the function $v_A$.

\begin{prop}\label{prop:vgeom} Let $A$ be a compact set in $\R^n$ and $x \in \conv(A)$.  
\begin{enumerate}
\item Then there exists an integer $1\le m\le n+1$, $m$ affinely independent points $a_1,\dots, a_m \in A$ and $m$ real numbers $p_1, \dots, p_m>0$ such that $\sum_{i=1}^{m} p_i=1$, $x=\sum_{i=1}^{m} p_ia_i$ and 
\ben
v^2_A (x)= \sum_{i=1}^{m} p_i |a_i|^2-|x|^2.
\een
\item Let $S=\{a_1, \dots, a_m\}$. Then there exists $c \in {\rm aff}S$ and $R_c>0$, such that $|a_i-c|=R_c$, for all $1\le i\le m$ and 
\ben
v^2_A (x)=R_c^2-|x-c|^2.
\een
Moreover $|a-c|\ge R_c$, for all $a\in A\cap {\rm aff}S$. 
\item For every $y\in\conv(S)$ there exists $q_1,\dots, q_m\ge0$ such that $\sum_{i=1}^{m} q_i=1$, $y=\sum_{i=1}^{m} q_ia_i$ and 
\ben
v^2_A (y)= \sum_{i=1}^{m} q_i |a_i|^2-|y|^2=R_c^2-|y-c|^2.
\een
\end{enumerate}
\end{prop}

\begin{proof} {\it 1.} Recall that 
$$v^2_A (x)=  \inf\left\{\sum_{i=1}^{m} \lambda_i |a_i|^2-|x|^2: m\in\N, x=\sum_{i=1}^{m}\lambda_i a_i; \lambda_i >0; \sum_{i=1}^{m} \lambda_i=1, a_i \in A  \right\}.$$
Following the standard proof of Carath\'eodory's theorem, we will show that for any decomposition of $x$ in the form $x=\sum\lambda_i a_i$, with $a_1, \dots, a_m$ being affinely dependent, the quantity $\sum \lambda_i |a_i|^2$ is not minimal. Thus the infimum in the definition of $v^2_A (x)$ may be reduced to affinely independent decompositions of $x$, thus with $m\le n+1$ points. Hence the infimum is taken on a compact set and is reached. 

So let $x=\sum\lambda_i a_i$ and assume that the sequence $a_1, \dots, a_m$ is affinely dependent then there exists a sequence of real numbers $\{\mu_i\}_{i=1}^m$, not all zeros, such that $\sum \mu_i a_i=0$ and $\sum \mu_i=0$. We note that (by multiplying, if needed, all $\mu_i$ by $-1$) we may also assume that
\begin{equation}\label{eq:poss}
\sum \mu_i|a_i|^2\ge0.
\end{equation}
And there is some $i$ such that $\mu_i>0$. Consider $k \in \{1,\dots, m\}$ such that
$$
\frac{\lambda_k}{\mu_k}=\min\{\frac{\lambda_i}{\mu_i}:\,\,\,\, \mu_i>0\}.
$$
Next, using that $a_k=-\sum\limits_{i\not=k}  \frac{\mu_i}{\mu_k} a_i$ we get
$$
x=\sum\limits_{i\not=k}\lambda_i a_i - \lambda_k \sum\limits_{i\not=k}  \frac{\mu_i}{\mu_k} a_i=\sum\limits_{i\not=k}\left(\lambda_i  -\lambda_k \frac{\mu_i}{\mu_k}\right) a_i,
$$
where $\left(\lambda_i -\lambda_k \frac{\mu_i}{\mu_k}\right) \ge 0$ for all $i$ and $\sum \left(\lambda_i -\lambda_k \frac{\mu_i}{\mu_k}\right) =1$, so we reduce the number of elements in sequence $\{a_i\}$. Thus, the only thing left is to show that
$$
\sum_{i\not=k}\left(\lambda_i  -\lambda_k \frac{\mu_i}{\mu_k}\right) |a_i|^2 \le  \sum_{i=1}^m \lambda_i  |a_i|^2.
$$
%or
%$$
%  -\frac{\lambda_k}{\mu_k}\sum\limits_{i\not=k} \mu_i   |a_i|^2 \le  \lambda_k  |a_k|^2
%$$
Using that $\mu_k>0$, the above is equivalent to $\sum \mu_i|a_i|^2\ge0$, which is exactly (\ref{eq:poss}). Therefore, we may assume that infimum in the definition of $v_A^2(x)$ is  is reached on affinely independent points and is actually a minimum. Hence, there exists an integer $1\le m\le n+1$, $m$ affinely independent points $a_1,\dots, a_m \in A$ and $m$ real numbers $p_1, \dots, p_m>0$ such that $\sum_{i=1}^{m} p_i=1$, $x=\sum_{i=1}^{m} p_ia_i$ and 
$v^2_A (x)= \sum_{i=1}^{m} p_i |a_i|^2-|x|^2.$\\

{\it 2.} One has $x=\sum_{i=1}^{m} p_ia_i$, with $p_i>0$ and $\sum_{i=1}^{m} p_i=1$ thus $x$ is in the relative interior of $\conv(S)$. Since $a_1,\dots, a_m$ are affinely independent, $\conv S$ is a $m$-dimensional simplex and there exists $c \in {\rm aff}S$ and $R_c>0$, such that $S \subset c+ R_c S^{n-1} $. Then $|a_i-c|=R_c$, for all $1\le i\le m$. Thus $|a_i|^2=R_c^2+2\langle c,a_i\rangle-|c|^2$, for all $1\le i\le m$. Hence
 \begin{eqnarray*}
 v^2_A (x)&=& \sum_{i=1}^{m} p_i |a_i|^2-|x|^2=\sum_{i=1}^{m} p_i (R_c^2-|c|^2+2\langle c,a_i\rangle)-|x|^2\\&=&R_c^2-|c|^2+2\langle c,x\rangle-|x|^2=R_c^2-|c-x|^2.
\end{eqnarray*}
 Assume now that there is $a\in A\cap {\rm aff}S$ such that $|a-c|< R_c$. Notice that we can select $k\in \{1, \dots, m\}$ such that  $x \in \conv \{ a, \{a_i\}_{i\not=k}\}$. Indeed, consider $a'=a+e_{n+1}\in\R^{n+1}$ and note that
the orthogonal projection of $\conv\{S, a'\}$ on ${\rm aff}S$ is equal to  $\conv\{S, a\}$
 and thus 
 $$
 \conv\{S\}  \subseteq \conv\{S, a\}= \bigcup_{k=1}^m \conv\{a, \{a_i\}_{i\not=k}\}.
 $$
Thus, there exists $\lambda_1,\dots, \lambda_m \ge0$, with $\sum_{i=1}^m \lambda_i=1$ such that $x = \sum_{i=1}^m \lambda_i \tilde{a}_i$, where $\tilde{a}_i=a_i$ for $i\neq k$ and $\tilde{a}_k=a$. Moreover, since $x$ is in the relative interior of $\conv(S)$ one has $\la_k>0$. Then
 \begin{eqnarray*}
\sum_{i=1}^{m} \lambda_i |\tilde{a}_i|^2-|x|^2 &=&  \sum\limits_{i=1}^{m} \lambda_i |\tilde{a}_i-c+c|^2-|x|^2=  \sum\limits_{i=1}^{m} \lambda_i |\tilde{a}_i-c|^2+2 \langle x-c, c\rangle +|c|^2-|x|^2\\
 &=&  \sum\limits_{i=1}^{m} \lambda_i |\tilde{a}_i-c|^2-|x-c|^2 < R_c^2 -|x-c|^2=v_A^2(x),
\end{eqnarray*}
which contradicts the minimality of the sequence $a_1, \dots, a_m$.\\

{\it 3.} Let $y \in \conv\{S\}$, then there exists $q_i\ge0, \sum q_i=1$ such that $y=\sum_{i=1}^{m} q_i a_i$. Consider another sequence $\{b_i\} \subset  A$, with  $y=\sum \lambda_i b_i$, and $\lambda_i>0, \sum \lambda_i=1$. 
%We will need to consider two cases:
%\noindent {\bf Case 1: $\{b_i\} \subset {\rm aff}S$.} Then, by 2. $|b_i-c|\ge R_c$, for all $i$. Thus
%$$
%v^2_A(y)=\sum\lambda_i|b_i|^2-|y|^2 =\sum \lambda_i|b_i-c|^2 - |c-y|^2 \ge R_c^2 -|c-y|^2=\sum \lambda_i|a_i|^2-|y|^2,
%$$
%and thus  $a_1, \dots, a_m$ is a minimizing sequence for $v^2_A(y)$.
%\noindent {\bf Case 2: $\{b_i\} \not\subset {\rm aff}S$.}  
Using the fact that $x, y \in {\rm aff} S$ we get,
as in {\it 2.} that $x=\mu_k y +\sum_{i\not=k} \mu_i a_i$, for some $\mu_i\ge0, \sum \mu_i=1$. Note that $\mu_k \not=0$, because $x$ is in the relative interior of $\conv S$. Thus
$$
x=\mu_k \sum \lambda_i b_i +\sum_{i\not=k} \mu_i a_i.
$$
The minimality of the sequence $S$ with respect to $v^2_A(x)$ implies that for any other convex combination $x=\sum \tilde{p}_i \tilde{a}_i$,  $\{\tilde{a}_i\} \subset A$, we get
$
\sum \tilde{p}_i|\tilde{a}_i - c|^2 \ge R_c.
$
Thus
$$
\sum \mu_k \lambda_i |b_i-c|^2 +\sum\limits_{i\not=k} \mu_i |a_i-c|^2\ge R_c^2.
$$
Using that $|a_i-c|=R_c$ and the fact that $\sum\limits_{i\not=k} \mu_i =1- \mu_k$ we get
$$
\sum  \lambda_i |b_i-c|^2  \ge R_c^2,
$$
which is exactly what we need to finish the proof. Indeed, again
$$
\sum\lambda_i|b_i|^2-|y|^2 =\sum \lambda_i|b_i-c|^2 - |c-y|^2 \ge R_c^2 -|c-y|^2=\sum q_i|a_i|^2-|y|^2,
$$
and thus $v^2_A (y)= \sum q_i |a_i|^2-|y|^2=R_c^2-|y-c|^2$ and $S$ is a minimizing sequence for $v^2_A(y)$.
\end{proof}

Now we are ready to use the above proposition to show that $v(A) \ge r(A)$.
For every $x \in \conv(A)$ let $S=\{a_1,\dots, a_m\}$ be the simplex obtained from Proposition \ref{prop:vgeom} and let $c$ and $R_c$ denote the center and the radius of the circumscribed ball of $S$. Then
$$
\sup\limits_{y\in \conv(S)} v^2_A(y)=R^2_c - \inf\limits_{y\in \conv(S)} |y-c|^2 =R^2_c-|c- w|^2,
$$
where $w=P_{\conv S}(c)$ denotes the projection of $c$ onto the convex set $\conv S$, i.e. the nearest point to $c$ from $\conv(S)$. 
%Note that if $c\in \conv(S)$ then 
%$$
%\sup\limits_{y\in \conv(S)} v^2_A(y) =R^2_c=R^2(S). 
%$$
%But actually, the same is true if $c\not\in \conv(S)$. Indeed, in this case 
For every $i$, one has $\langle a_i-w,c-w\rangle\le0$ thus $\langle a_i-c,c-w\rangle\le-|c-w|^2$, hence
$$
|a_i-w|^2=|a_i-c+c-w|^2=|a_i-c|^2+2\langle a_i-c,c-w\rangle+|c-w|^2\le R_c^2-|c-w|^2.
$$
Thus $S$ is contained in the ball of radius $\rho=\sqrt{R_c^2-|c-w|^2}=\sup_{y\in \conv(S)} v_A(y)\le v(A)$.
Hence $r_A(x)\le\rho\le v(A)$ which finishes the proof of $r(A) \le v(A)$.
\qed
\par\vspace{.1in}

\begin{table}[!h]
\begin{center}
\begin{tabular}{|c|l|l|l|l|}
\hline
{$\Rightarrow$} & {$d$} & $r$ & $c$ & $\Delta$  \\
\hline
$d$ &  = & N (Ex. \ref{ex:sch}, \ref{ex:delta-r}) & N (Ex. \ref{ex:sch}, \ref{ex:c}) & N (Ex. \ref{ex:dis}) \\
\hline
$r$ &  Y (Th. \ref{thm:weg}) & = & N (Ex. \ref{ex:c}, \ref{ex:rc}) & N (Ex. \ref{ex:dis}) \\
\hline
$c$ &  N (Ex. \ref{ex:c}, \ref{ex:delta-cd}) & N (Ex. \ref{ex:c}) & = & N (Ex. \ref{ex:dis}, \ref{ex:c})\\
\hline
$\Delta$ &  N (Ex. \ref{ex:delta-d}, \ref{ex:delta-cd}) & N (Ex. \ref{ex:sch}, \ref{ex:delta-r}) & N (Ex. \ref{ex:c}, \ref{ex:rc}) & = \\
\hline
\end{tabular}
\end{center}
\caption{When does convergence to 0 for one measure of non-convexity unconditionally imply the same for another?}\label{tab:1}
\end{table}

The above relationships (summarized in Table~\ref{tab:1}) are the only unconditional relationships that exist
between these notions in general dimension. To see this, we list below some examples that show why no other
relationships can hold in general.

\begin{ex}\label{ex:c}
By Lemma~\ref{lem:scaling}, we can scale a non convex set to get examples
where $c$ is fixed but $d,r$ and $\Delta$ converge to 0, for example, take $A_k=\{0; \frac{1}{k}\}$; 
or to get examples where $c$ goes to 0 but $d,r$ are fixed and $\Delta$ diverges, for example take $A_k=\{0, 1, \dots, k\}$. 
\end{ex}

%(d,r):
%Wegmann \cite{Weg80} gave counterexamples 
%(the point 0 together with the vertical line at 1) to show that $d(A)$ can be 1 while $r(A)$ is infinity,
%or that $d(A)$ can be arbitrarily close to 0 while $r(A)$ stays bounded away from 0. %(see Example 2.4 in \cite{DF04}).
%On the other hand, Theorem \ref{thm:weg} implies that $r(A_k)\ra 0$ implies $d(A_k)\ra 0$.
%
%

\begin{ex}\label{ex:delta-r}
 An example where $\Delta(A_k) \ra 0$, $d(A_k)\ra 0$ but $r(A_k)$ is bounded away
from 0 is given by a right triangle from which a piece is shaved off leaving a protruding edge, see Figure \ref{fig:dcr}.
%However, the following pairs of distances are not
%comparable in terms of their convergence to 0: $(d, \Delta), (r, \Delta), (c, \Delta), (c, r)$.
\end{ex}

\begin{figure}[h!]
\begin{center}
\includegraphics[scale=0.4]{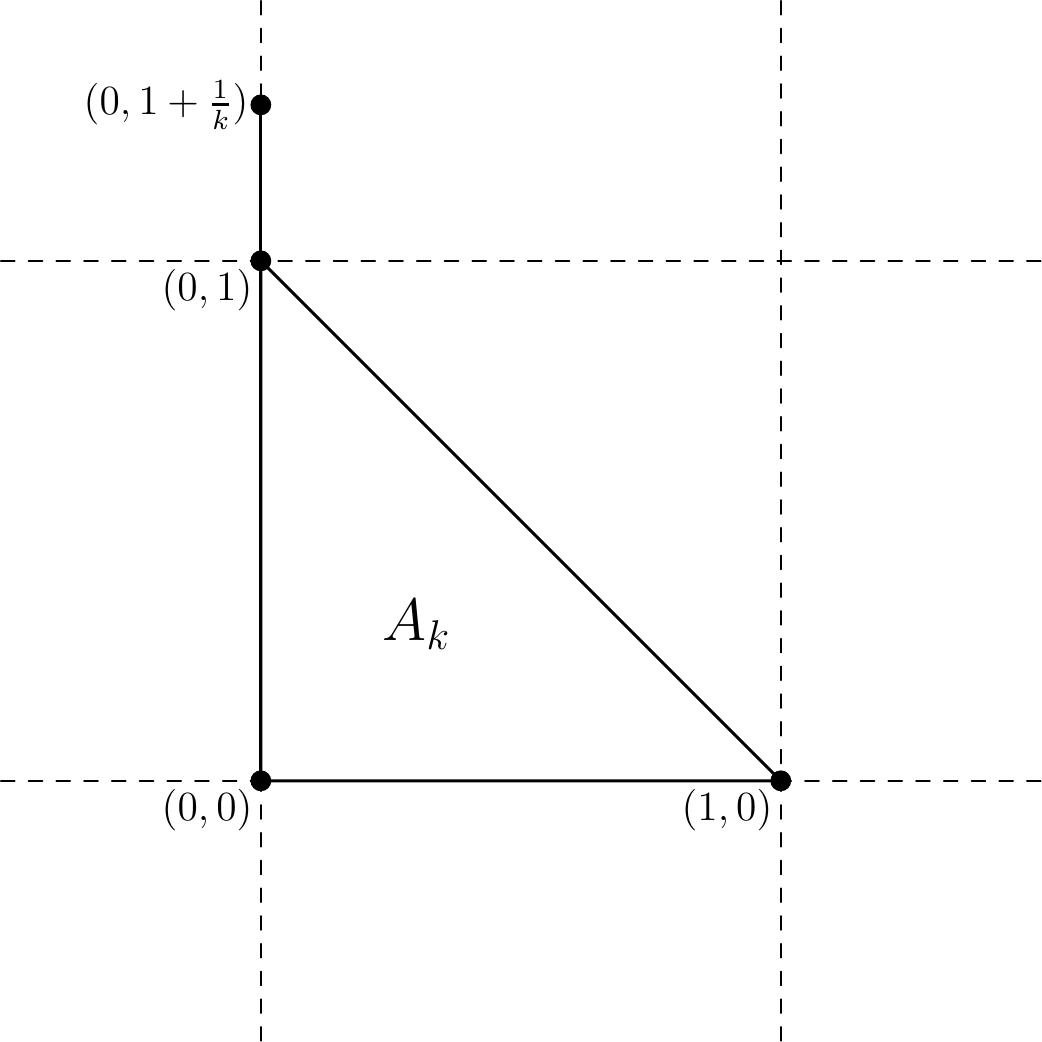}
\end{center}
\caption{$\Delta(A_k) \ra 0$  but $r(A_k)>\sqrt{2}/2$  (Example \ref{ex:delta-r}).}\label{fig:dcr}
\end{figure}

\begin{ex}\label{ex:delta-d}
 An example where $\Delta(A_k)\ra 0$ but both $c(A_k)$ and $d(A_k)$ are bounded away from 0
is given by taking a 3-point set with 2 of the points getting arbitrarily closer but staying away from the third, see Figure \ref{fig:DcD}.
\end{ex}

\begin{figure}[h!]
\begin{center}
\includegraphics[scale=0.4]{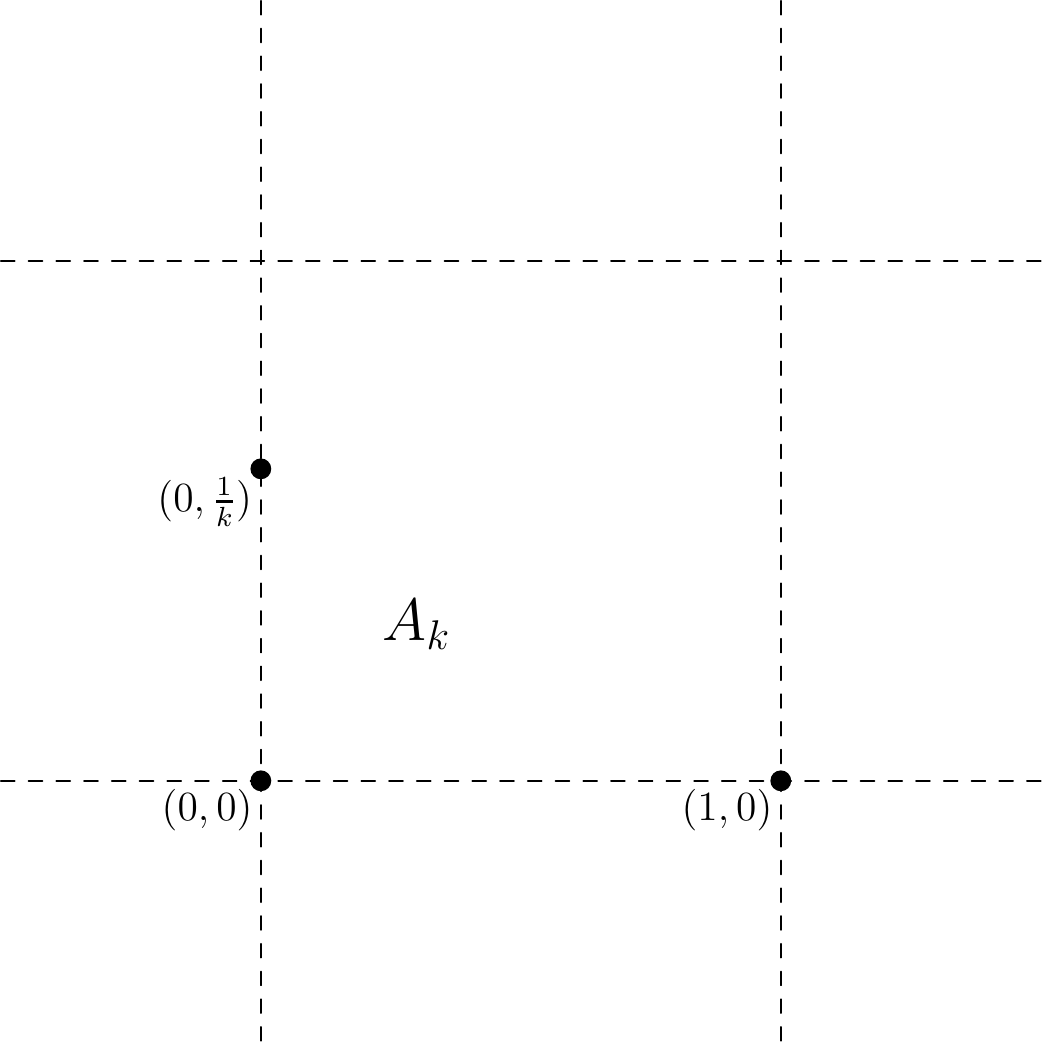}
\end{center}
\caption{$\Delta(A_k)\ra 0$ but  $c(A_k) \ge 1$ and $d(A_k)\ge 1/2$  (Example \ref{ex:delta-d}).}\label{fig:DcD}
\end{figure}

\begin{ex}\label{ex:delta-cd}
 An example where $\Delta(A_k) \ra 0$ and  $c(A_k)\ra 0$ but $d(A_k)>1/2$ can be found in  Figure \ref{fig:dcr-same}.
\end{ex}

\begin{figure}[h!]
\begin{center}
\includegraphics[scale=0.8]{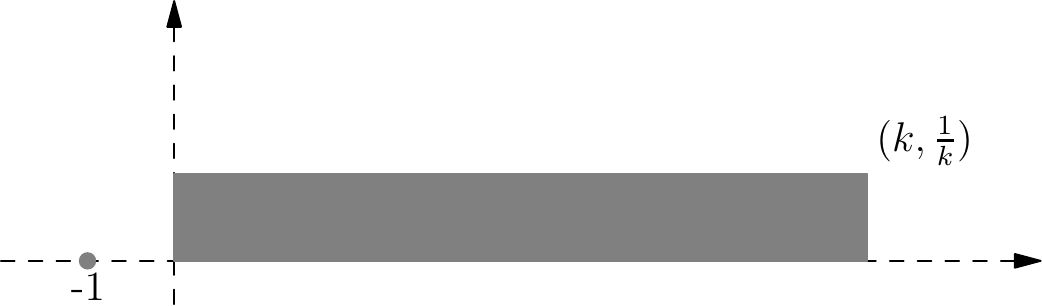}
\end{center}
\caption{$\vol_2(A_k) \ge 1$, $\Delta(A_k) \ra 0$ and  $c(A_k)\ra 0$ but $d(A_k)>1/2$  (Example \ref{ex:delta-cd}).}\label{fig:dcr-same}
\end{figure}

%\end{enumerate}

\subsection{Conditional relationships}
\label{sec:cond}

There are some more relationships between different notions of non-convexity that emerge if we impose some natural conditions on the sequence of sets (such as ruling out escape to infinity, or vanishing to almost nothing). 

A first observation of this type is that Hausdorff distance to convexity is dominated by Schneider's index of non-convexity if $A$ is contained in a ball of known radius.

\begin{lem}\label{lem:c-d}
For any compact set $A \subset \R^n$, 
\be\label{eq:d-c}
d(A)\leq R(A) c(A) .
\ee
\end{lem}

\begin{proof}
By translation invariance, we may assume that $A\subset R(A) B_2^n$. Then $0\in \conv(A)$, and it follows that
\ben
\conv(A) \subset \conv(A) + c(A)\conv(A) = A + c(A)\conv(A) \subset A + c(A)R(A)B_2^n.
\een
Hence $d(A)\leq R(A) c(A)$.
\end{proof}

This bound is useful only if $c(A)$ is smaller than $1$, because we already know that $d(A)\le r(A)\le R(A)$.

In dimension 1, all of the non-convexity measures are tightly connected.

\begin{lem}\label{lem:meas-1d}
Let $A$ be a compact set in $\R$. Then
\be\label{eq:c-d-1d}
r(A) = d(A) = R(A)c(A)\le \frac{\Delta(A)}{2}.
% \frac{\Delta(A) + \vol_1(A)}{2} c(A) 
\ee
%is half the size of the largest ``gap'' in $A$. 
\end{lem}

\begin{proof}
We already know that $d(A)\le r(A)$. Let us prove that $r(A)\le d(A)$. From the definition of $r(A)$ and $d(A)$, we have
$$r(A)=\sup_{x\in\conv(A)}\inf\left\{\frac{\beta-\alpha}{2}; \alpha, \beta\in A, \alpha\le x\le \beta\right\},
\quad d(A)=\sup_{y\in\conv(A)}\inf_{\alpha\in A}|y-\alpha|.$$
Thus we only need to show that for every $x\in\conv(A)$, there exists $y\in\conv(A)$ such that
$$\inf\left\{\frac{\beta-\alpha}{2}; \alpha, \beta\in A, \alpha\le x\le \beta\right\}\le\inf_{\alpha\in A}|y-\alpha|.$$
By compactness there exists $\alpha, \beta\in A$, with $\alpha\le x\le \beta$ achieving the infimum in the left hand side. Then we only need to choose $y=\frac{\alpha+\beta}{2}$ in the right hand side to conclude that $r(A)\le d(A)$. In addition, we get $(\alpha,\beta)\subset \conv(A)\setminus A$ thus $2r(A)=\beta-\alpha\le\Delta(A)$.\\

Now we prove that $d(A)=R(A)c(A)$. From Lemma \ref{lem:c-d}, we  have $d(A) \leq R(A) c(A)$. Let us prove that $R(A) c(A)\le d(A)$. By an affine transform, we may reduce to the case where $\conv(A)=[-1,1]$, thus $-1=\min(A)\in A$ and $1=\max(A)\in A$. Notice that $R(A)=1$ and denote $d:=d(A)$. By the definition of $d(A)$, one has $[-1,1]=\conv(A)\subset A+[-d,d]$. Thus using that $-1\in A$ and $1\in A$, we get
$$
A + d(A) \conv(A) =  A+[-d,d] \supset  (-1+[-d,d]) \cup [-1,1]\cup (1+[-d,d])=[-1-d, 1+d],
$$
we conclude that $A + d(A) \conv(A)\supset (1+d(A))\conv(A)$ and thus $R(A) c(A)=c(A)\le d(A)$.
\end{proof}

Notice that the inequality on $\Delta$ of Lemma \ref{lem:meas-1d} cannot be reversed as shown by Example \ref{ex:dis}.
The next lemma provides a connection between $r$ and $c$ in $\R^n$.

\begin{lem}\label{lem:c-r}
For any compact set $A \subset \R^n$, 
\be\label{eq:r-c}
r(A) \leq 2 \frac{c(A)}{1+c(A)} R(A).
\ee
\end{lem}

\begin{proof}
Consider $x^*$ the point in $\conv(A)$ that realizes the maximum in the definition of $\rho(A)$ (it exists since $\conv(A)$ is closed). Then, for every $a \in A_{x^*}$, one has $\rho(A) \leq |x^*-a|$. By definition,
$$ c(A) = \inf \{\lambda \geq 0 : \conv(A) = \frac{A + \lambda \conv(A)}{1+\lambda} \}. $$
Hence,
$$ x^* = \frac{1}{1+c(A)}a + \frac{c(A)}{1+c(A)}b, $$
for some $a \in A$ and $b \in \conv(A)$. Since $\frac{1}{1+c(A)} + \frac{c(A)}{1+c(A)} = 1$, one deduces that $a \in A_{x^*}$. Thus,
$$ \rho(A) \leq |x^*-a|. $$
But,
$$ x^* - a = \frac{1}{1+c(A)}a + \frac{c(A)}{1+c(A)}b - a = \frac{c(A)}{1+c(A)} (b-a). $$
It follows that
$$ \rho(A) \leq |x^*-a| = \frac{c(A)}{1+c(A)} |b-a| \leq \frac{c(A)}{1+c(A)} \diam(A) \leq 2 \frac{c(A)}{1+c(A)} R(A). $$
As shown by Wegmann (cf. Theorem \ref{thm:weg}), if $A$ is closed then $\rho(A) = r(A)$. We conclude that
$$ r(A) \leq 2 \frac{c(A)}{1+c(A)} R(A). $$
\end{proof}

Our next result says that the only reason for which we can find examples
where the volume deficit goes to 0, but the Hausdorff distance from convexity does not,
is because we allow the sets either to shrink to something of zero volume, or run off to infinity.

\begin{thm}\label{thm:delta-d}
Let $A$ be a compact set in $\R^n$ with nonempty interior. Then 
\be\label{eq:counter-size}
d(A)\le \left(\frac{n}{\vol_{n-1}(B_2^{n-1})}\right)^\frac{1}{n}\left(\frac{2R(A)}{\inr(\conv(A))}\right)^\frac{n-1}{n}\Delta(A)^\frac{1}{n}.
\ee
\end{thm}

\begin{proof}
From the definition of $d(A)$ there exists $x\in\conv(A)$ such that $\vol_n((x+d(A)B_2^n)\cap A)=0$. Thus 
$\Delta(A)\ge \vol_n(\conv(A) \cap (x + d(A) B_2^n))$. Let us denote $r=\inr(\conv(A))$. From the definition of $\inr(\conv(A))$, there exists $y\in\conv(A)$ such that $y+rB_2^n\subset \conv(A)$. Hence
$$\Delta(A)\ge \vol_n(\conv(x, y+rB_2^n) \cap (x + d(A) B_2^n))\ge \frac{1}{n}\vol_{n-1}(B_2^{n-1})\left(\frac{rd(A)}{2R(A)}\right)^{n-1}. $$
Let $\{z\}=[x,y]\cap (x + d(A) S^{n-1})$ be the intersection point of the sphere centered at $x$ and the segment $[x,y]$ and let $h$ be the radius of the $(n-1)$-dimensional sphere $S_h=\partial(\conv(x, y+rB_2^n)) \cap (x + d(A) S^{n-1}))$. Then $h=\frac{d(A)r}{|x-y|}$ and $\conv(x, y+rB_2^n) \cap (x + d(A) B_2^n)\supset\conv(x,S_h,z)$. Thus
$$\Delta(A)\ge \vol_n(\conv(x,S_h,z))=\frac{d(A)}{n}\vol_{n-1}(B_2^{n-1})h^{n-1}\ge\frac{d(A)^n}{n}\vol_{n-1}(B_2^{n-1})\left(\frac{r}{|x-y|}\right)^{n-1}.$$
\end{proof}

Observe that the first term on the right side in inequality \eqref{eq:counter-size} is just a dimension-dependent constant,
while the second term depends only on the ratio of the radius of the smallest Euclidean ball containing $A$
to that of the largest Euclidean ball inside it.

%ATTEMPT TO REMOVE BOUNDEDNESS ASSUMPTION THAT DID NOT WORK:
%If $\omega_n$ is the volume of the unit Euclidean ball, then 
%$
%\diam(A)\geq 2(|A|/\omega_n)^{\nth} .
%$
%Applying this to each set in the sequence $(A_k)$ and using the uniform lower bound on volume,
%we deduce that
%$\diam(A_k)\geq 2C/\omega_n^{\nth}$.
%If $\diam(A_k)\ra\infty$, then by setting $\tilde{A}_k=A_k/\diam(A_k)$, we see $\diam(\tilde{A}_k)=1$
%while $\vol_n(\tilde{A}_k)\geq c$ and $\Delta(\tilde{A}_k)\ra 0$.

%Of course, this is a purely qualitative result because of the compactness argument,
%and it would be interesting to obtain a quantitative version of it that bounds $d(A)$, 
%for a given set $A$ of volume at least $C$ and lying in a Euclidean ball of radius $R$, 
%in terms of $\Delta(A), C$ and $R$.

The next lemma enables to compare the inradius, the outer radius and the volume of convex sets. Such estimates were studied in \cite{BHS03},  \cite {San88} where, in some cases, optimal  inequalities were proved in dimension 2 and 3. 

\begin{lem}
Let $K$ be a convex body in $\R^n$. Then 
$$
\vol_n(K)\le(n+1)\vol_{n-1}(B_2^{n-1})\inr(K)(2R(K))^{n-1}.
$$
\end{lem}

\begin{proof}
From the definition of $\inr(K)$, there exists $y\in K$ such that $y+\inr(K)B_2^n\subset K$. Without loss of generality, we may assume that $y=0$ and that $\inr(K)=1$, which means that $B_2^n$ is the Euclidean ball of maximal radius inside $K$. This implies that $0$ must be in the convex hull of the contact points of $S^{n-1}$ and $\partial(K)$, because if it is not, then there exists an hyperplane separating $0$ from these contact points and one may construct a larger Euclidean ball inside $K$. 
Hence from Caratheodory,
%- either there are two opposite contact points between $S^{n-1}$ and $\partial(K)$ and then $K$ is inside the slab of with $2$. Moreover $K\subset\diam(K)B_2^n$ thus 
%$$\vol_n(K)\le \vol_n(\{x\in \diam(K)B_2^n: |x_1|\le1\})\le 2\diam(K)^{n-1}\vol_{n-1}(B_2^{n-1}).$$
 there exists $1\le k\le n$ and $k+1$ contact points $a_1,\dots,a_{k+1}$ so that $0\in\conv(a_1,\dots,a_{k+1})$ and $K\subset S=\{x: \langle x, a_i\rangle\le 1, \forall i\in\{1,\dots, k+1\}\}$. Since $0\in\conv(a_1,\dots,a_{k+1})$, there exists $\la_1, \dots, \la_{k+1}\ge0$ such that $\sum_{i=1}^{k+1}\la_ia_i=0$. Thus for every $x\in\R^n$, $\sum_{i=1}^{k+1}\la_i\langle x,a_i\rangle=0$ hence there exists $i$ such that $\langle x,a_i\rangle\ge0$. Hence
$$
S\subset\bigcup_{i=1}^{k+1}[0,a_i]\times\{x : \langle x, a_i\rangle=0\}.
$$
Moreover $K\subset\diam(K)B_2^n$ thus 
$$
K\subset S\cap \diam(K)B_2^n\subset\bigcup_{i=1}^{k+1}[0,a_i]\times\{x\in\diam(K)B_2^n : \langle x, a_i\rangle=0\}.
$$
Passing to volumes and using that $a_i\in S^{n-1}$, we get 
$$\vol_n(K)\le  (k+1)\vol_{n-1}(B_2^{n-1})(\diam(K))^{n-1}\le(n+1)\vol_{n-1}(B_2^{n-1})(2R(K))^{n-1}.
$$
%Let $b_1, \dots, b_{n+1}$ denote the vertices of $S$ arranged in such a way that $\langle a_i,b_j\rangle=\inr(K)^2$, for all $i\neq j$. We divide the simplex $S$ into $n+1$ simplices $S_i=\conv(0, b_1,\dots, b_{i-1}, b_{i+1}, \dots, b_{n+1})$. Then $S_i\cap\diam(K)B_2^n\subset [0,a_i]\times\{x\in\diam(K)B_2^n: \langle x, a_i\rangle=0\}$. 
%
%Finally
%$$\vol_n(S_i\cap \diam(K)B_2^n)\le \inr(K)(\diam(K))^{n-1}\vol_{n-1}(B_2^{n-1}).$$
%Therefore
%$$\vol_n(K)\le \sum_{i=1}^{n+1}\vol_n(S_i\cap \diam(K)B_2^n)\le (n+1)\vol_{n-1}(B_2^{n-1})\inr(K)(2R(K))^{n-1}.
%$$
\end{proof}

An immediate corollary of the above theorem and lemma is the following.

\begin{coro}\label{co:d-del}
Let $A$ be a compact set in $\R^n$. Then
$$
d(A)\le c_n\frac{R(A)^{n-1}}{\vol_n(\conv(A))^\frac{n-1}{n}}\Delta(A)^\frac{1}{n},
$$
where $c_n$ is an absolute constant depending on $n$ only. Thus for any sequence of compact sets $(A_k)$ in $\R^n$ such that 
$\sup_k R(A_k)<\infty$  and $\inf_k \vol_n(A_k)  >0$,
the convergence $\Delta(A_k)\ra 0$ implies that $d(A_k)\ra 0$.

\end{coro}

\begin{table}[!h]
\begin{center}
\begin{tabular}{|c|l|l|l|l|}
\hline
{$\Rightarrow$} & {$d$} & $r$ & $c$ & $\Delta$  \\
\hline
$d$ &  = & N (Ex.  \ref{ex:sch},  \ref{ex:ak}) & N (Ex.   \ref{ex:sch}, \ref{ex:ak}) & N (Ex. \ref{ex:hw}) \\
\hline
$r$ &  Y & = & N (Ex. \ref{ex:rc}) & N (Ex. \ref{ex:hw}) \\
\hline
$c$ &  Y (Lem. \ref{lem:c-d}) & Y (Lem. \ref{lem:c-r})& = & N (Ex. \ref{ex:hw})\\
\hline
$\Delta$ & Y  (Cor. \ref{co:d-del})& N (Ex. \ref{ex:sch},  \ref{ex:ak}) & N (Ex.  \ref{ex:sch}, \ref{ex:ak})& = \\
\hline
\end{tabular}
\end{center}
\caption{When does convergence to 0 for one measure of non-convexity imply the same for another
when we assume the sequence lives in a big ball and has positive limiting volume?}
\end{table}

From the preceding discussion, it is clear that $d(A_k)\ra 0$ is a much weaker statement than
either $c(A_k)\ra 0$ or $r(A_k)\ra 0$.

\par\vspace{.1in}

\begin{ex}\label{ex:hw} Consider a unit square with a set of points in the neighboring unit square, where the set of points becomes more dense as $k \to \infty$ (see Figure \ref{fig:dD}). This example shows  that the convergence in the Hausdorff sense
is weaker than convergence in the volume deficit sense even when the volume
of the sequence of sets is bounded away from 0.

\begin{figure}[h!]
\begin{center}
\includegraphics[scale=0.6]{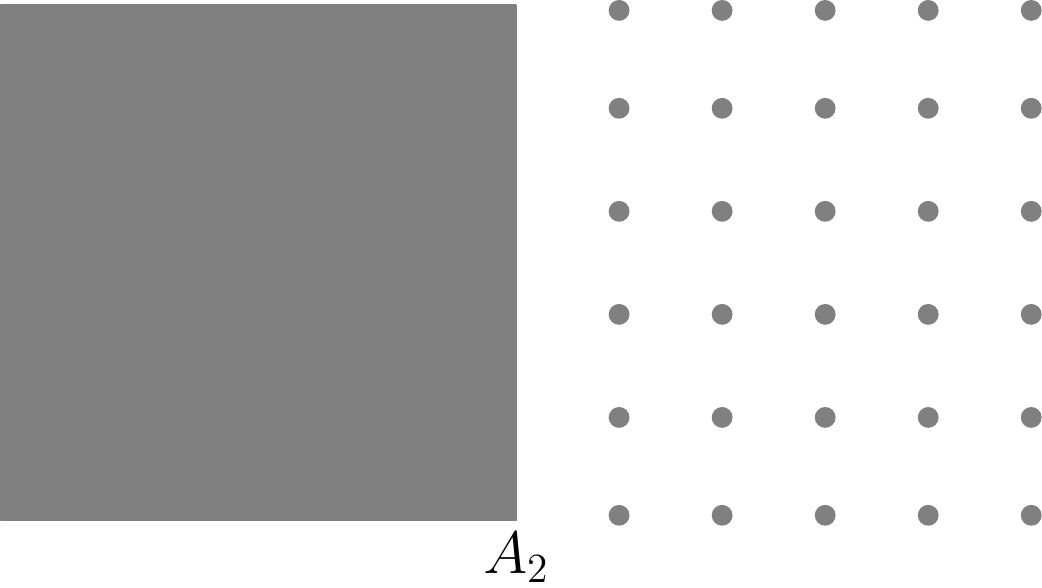} \hskip 2cm \includegraphics[scale=0.6]{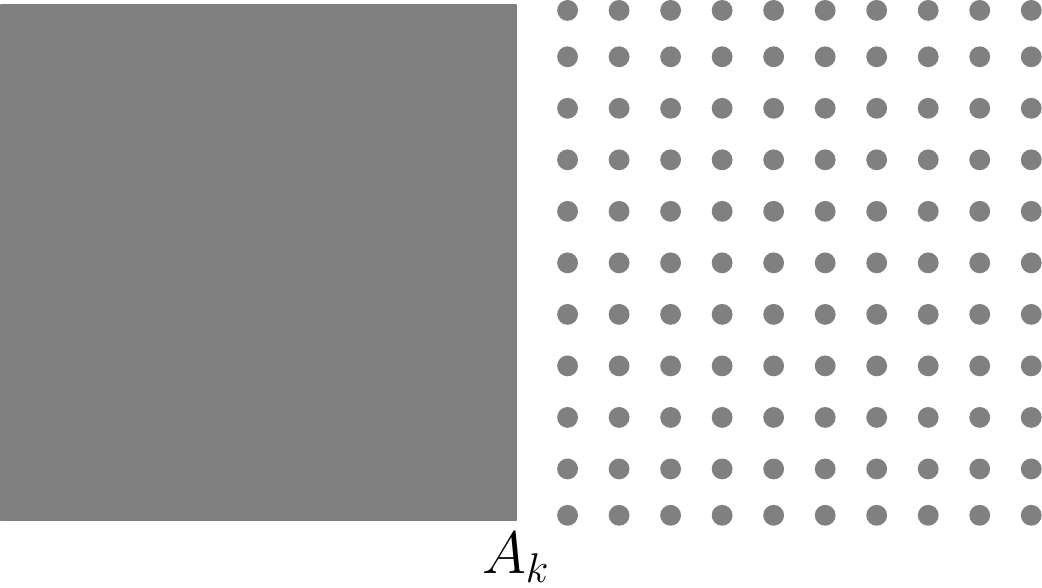}
\end{center}
\caption{$d(A_k) \ra 0$ and  $\vol_2(A_k)> c$ but $\Delta(A_k)>c$  (Example \ref{ex:hw}).}\label{fig:dD}
\end{figure}
\end{ex}
The following example shows that convergence in $\Delta$ does not imply convergence in $r$ nor $c$: 
\begin{ex} \label{ex:ak}
Consider the set $A_k = \{ (1-\frac{1}{k}, 0)\} \cup ([1,2]\times[-1,1])$ in the plane.
\end{ex}
Note that the Example \ref{ex:ak} also shows that convergence in $d$ does not imply convergence in $r$ nor $c$. 
The following example shows that convergence in $r$ does not imply convergence in $c$: 
\begin{ex} \label{ex:rc}
Consider the set $A_k=B_2^2\cup\{(1+1/k, 1/k) ; (1+1/k, -1/k)\}$ in the plane, the union of the Euclidean ball and 
two points close to it and close to each other (see Figure \ref{fig:cr}). 
Then we have $c(A_k)=1$ by applying the same argument as in Example \ref{ex:sch} to the point $(1+1/k,0)$.
%because taking an homothety of scale 1/2 is the best you can do if you want to cover the point $(1+1/k,0)$ with an homothety with center in a point of $A_k$. 
But for $r(A_k)$, we see that because of the roundness of the ball, one has $r(A_k)=\frac{\sqrt{k+1}}{\sqrt{2}k} \to 0$, when $k$ grows.

\begin{figure}[h!]
\begin{center}
\includegraphics[scale=0.6]{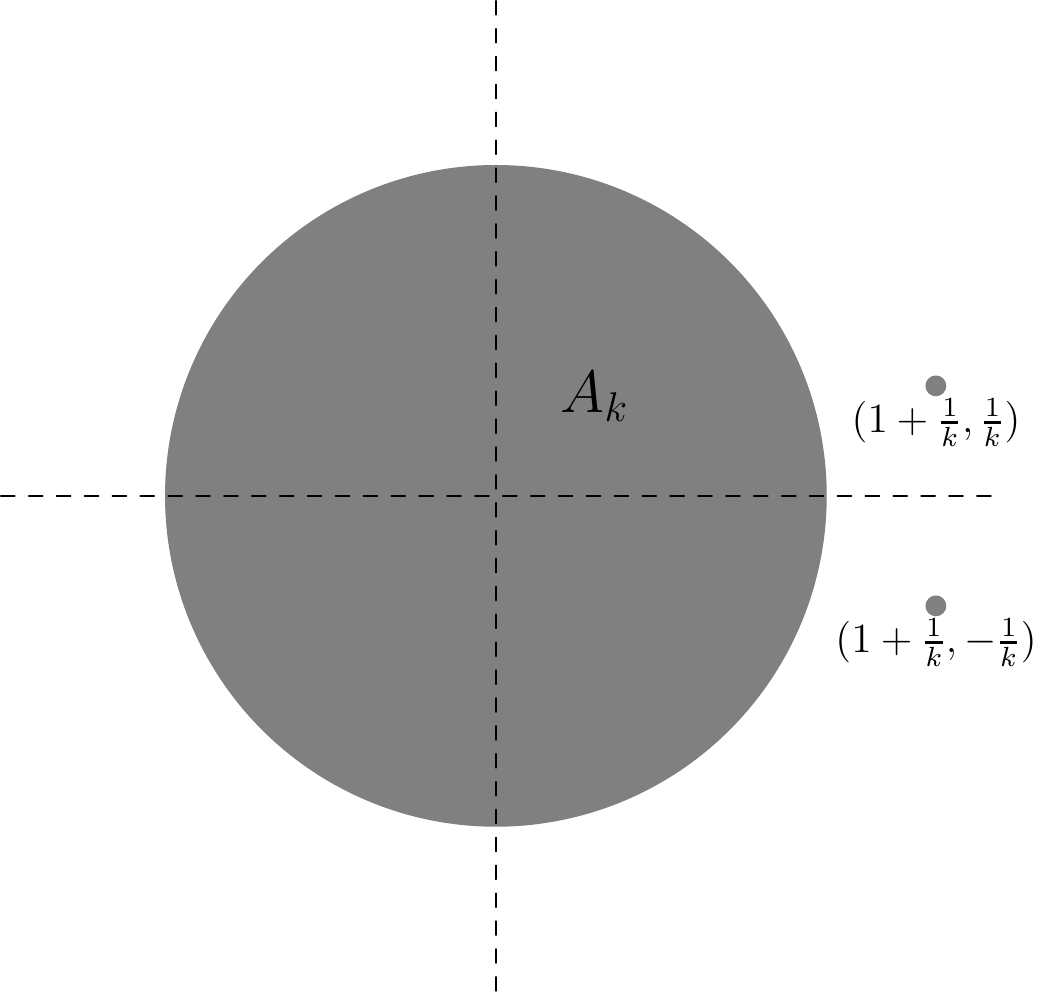} 
\end{center}
\caption{$c(A_k)=1$ but $r(A_k) \to 0$, when $k$ grows  (Example \ref{ex:rc}). }\label{fig:cr}
\end{figure}

\end{ex}

\section{The behavior of volume deficit}
\label{sec:Delta}

In this section we study the volume deficit. Recall its definition: for $A$ compact in $\R^n$,
\ben
\Delta(A)=\vol_n(\conv(A)\setminus A) = \vol_n(\conv(A))- \vol_n(A).
\een 

\subsection{Monotonicity of volume deficit in dimension one and for Cartesian products}
\label{sec:1d}

In this section, we observe that Conjecture \ref{strongconj} holds in dimension one and also for 
products of one-dimensional compact sets. In fact, more generally, we prove that Conjecture \ref{strongconj} passes to Cartesian products.

\begin{thm}\label{thm:1d}
Conjecture \ref{strongconj} holds in dimension one. In other words, if $k\ge2$ is an integer and $A_1, \dots, A_k$ are compact sets in $\R$, then
\begin{eqnarray}\label{eq:thm-fsa_0}
\vol_1\left(\sum_{i=1}^kA_i\right) \ge \frac{1}{k-1}\sum_{i=1}^k\vol_1\left(\sum_{j\in[k]\setminus\{i\}}A_j\right).
\end{eqnarray}
\end{thm}

\begin{proof}
We adapt a proof of Gyarmati, Matolcsi and Ruzsa \cite[Theorem 1.4]{GMR10} who established the same kind of inequality for finite subsets of the integers and cardinality instead of volume. The proof is based on set inclusions. Let $k \geq 1$. Set $S=A_1+ \cdots +A_k$ and for  $i \in[k]$, let $a_i=\min A_i$, $b_i=\max A_i$, 
$$S_i = \sum_{j\in[k]\setminus\{i\}}A_j,$$
$s_i=\sum_{j<i}a_j+\sum_{j>i}b_j$, $S_i^-=\{x\in S_i ; x\le s_i\}$ and $S_i^+=\{x\in S_i ; x> s_i\}$. For all $i \in [k-1]$, one has 
$$ S\supset (a_i+S_i^-)\cup(b_{i+1}+S_{i+1}^+).$$
Since $a_i+s_i=\sum_{j\le i}a_j+\sum_{j>i}b_j=b_{i+1}+s_{i+1}$, the above union is a disjoint union. Thus for $i \in [k-1]$
$$\vol_1(S) \ge \vol_1(a_i+S_i^-)+\vol_1(b_{i+1}+S_{i+1}^+) 
= \vol_1(S_i^-) + \vol_1(S_{i+1}^+).$$
Notice that $S_1^-=S_1$ and $S_k^+=S_k\setminus\{s_k\}$, thus adding the above $k-1$ inequalities we obtain
\begin{eqnarray*}
(k-1)\vol_1(S) & \ge & \sum_{i=1}^{k-1}\left(\vol_1(S_i^-) + \vol_1(S_{i+1}^+) \right) \\ 
& = & \vol_1(S_1^-) + \vol_1(S_k^+) + \sum_{i=2}^{k-1}\vol_1(S_i) \\
&=& \sum_{i=1}^k\vol_1(S_i).
\end{eqnarray*}
We have thus established Conjecture \ref{strongconj} in dimension 1.
\end{proof}

\begin{rem}
As mentioned in the proof, Gyarmati, Matolcsi and Ruzsa \cite{GMR10} earlier obtained a discrete version of Theorem~\ref{thm:1d}
for cardinalities of sums of subsets of the integers. There are also interesting upper bounds on cardinalities of sumsets
in the discrete setting that have similar combinatorial structure, see, e.g., \cite{GMR10, BB12, MMT12} and references therein.
Furthermore, as discussed in the introduction for the continuous domain, there are also discrete entropy analogues of these
cardinality inequalities, explored in depth in \cite{Ruz09:1, Tao10, MK10:isit, BB12, MMT12, HAT14, WWM14:isit, MWW17:1, MWW17:2} and references therein.
We do not discuss discrete analogues further in this paper.
\end{rem}

Now we prove that Conjecture \ref{strongconj} passes to Cartesian products.

\begin{thm}\label{thm:product}
Let $k,m\ge 2$ and $n_1, \dots, n_m\ge 1$ be integers. Let $n=n_1+\cdots +n_m$. For $1\le i\le k$ and $1\le l\le m$, let $A_i^l$ be some compact sets in $\R^{n_l}$. Assume that for any $1\le l\le m$ the $k$ compact sets $A_1^l,\dots, A_k^l\subset \R^{n_l}$ satisfy Conjecture \ref{strongconj}. 
For $1\le i\le k$, let $A_i=A_i^1\times \cdots\times A_i^m\subset\R^n=\R^{n_1}\times\cdots\times\R^{n_m}$. Then Conjecture \ref{strongconj} holds for $A_1,\dots, A_k$.
\end{thm}

\begin{proof}
Let $S=\sum_{i=1}^kA_i$ and let $S_i = \sum_{j\neq i} A_j$ then let us prove that 
$$(k-1)\vol_n(S)^\frac{1}{n}\ge \sum_{i=1}^k \vol_n(S_i)^\frac{1}{n}.
$$
For all $1\le i\le k$, one has 
$$S_i=\sum_{j\neq i} A_j=\sum_{j\neq i}\prod_{l=1}^mA_j^l=\prod_{l=1}^m\left(\sum_{j\neq i}A_j^l\right).$$
For $1\le i\le k$, denote $\sigma_i=(\vol_{n_l}(\sum_{j\neq i}A_j^l)^\frac{1}{n_l})_{1\le l\le m}\in\R^m$, and for $x=(x_l)_{1\le l\le m}\in\R^m$, denote $\|x\|_0=\prod_{l=1}^m |x_l|^\frac{n_l}{n}$. Then, using Minkowski's inequality for $\|\cdot\|_0$ (see, for example, Theorem 10 in \cite{HLP88:book}),  we deduce that
$$
\sum_{i=1}^k \vol_n(S_i)^\frac{1}{n}=\sum_{i=1}^k \prod_{l=1}^m\vol_{n_l}\left(\sum_{j\neq i}A_j^l\right)^\frac{1}{n}=\sum_{i=1}^k\|\sigma_i\|_0\le\left\| \sum_{i=1}^k\sigma_i\right\|_0=\prod_{l=1}^m\left(\sum_{i=1}^k\sigma_i^l\right)^\frac{n_l}{n}.
$$
Using that for any $1\le l\le m$ the $k$ compact sets $A_1^l,\dots, A_k^l\subset \R^{n_l}$ satisfy Conjecture \ref{strongconj}, we obtain
$$
\sum_{i=1}^k\sigma_i^l=\sum_{i=1}^k\vol_{n_l}\left(\sum_{j\neq i}A_j^l\right)^\frac{1}{n_l}\le (k-1)\vol_{n_l}\left(\sum_{i=1}^kA_i^l\right)^\frac{1}{n_l}.
$$
Thus
$$
\sum_{i=1}^k \vol_n(S_i)^\frac{1}{n}\le \prod_{l=1}^m\left((k-1)\vol_{n_l}\left(\sum_{i=1}^kA_i^l\right)^\frac{1}{n_l}\right)^\frac{n_l}{n}=(k-1)\vol_n(S)^\frac{1}{n}.
$$
\end{proof}

From Theorems \ref{thm:1d} and \ref{thm:product}, and the fact that Conjecture \ref{strongconj} holds for convex sets, 
we deduce that Conjecture \ref{strongconj} holds for Cartesian products of one-dimensional compact sets and convex sets.

\subsection{A counterexample in dimension $\ge 12$}
\label{sec:counter}

In contrast to the positive results for compact product sets, both the conjectures of 
Bobkov, Madiman and Wang \cite{BMW11} fail in general for even moderately high dimension.

\begin{thm}\label{thm:counter}
For every $k \geq 2$, there exists $n_k \in \N$ such that for every $n\ge n_k$ there is a compact set $A \subset \R^n$ such that $\vol_n(A(k+1))<\vol_n(A(k))$. Moreover, one may take
$$ n_k =\min\left\{n\in k\Z: n> \frac{\log(k)}{\log\left(1+\frac{1}{k}\right) - \frac{\log(2)}{k}}\right\}. $$
In particular, one has $n_2=12$, whence Conjectures~\ref{weakconj} and Conjecture~\ref{strongconj} are false in $\R^n$ for $n\ge12$.
\end{thm}

\begin{proof}
Let $k \geq 2$ be fixed and let $n_k$ be defined as in the statement of Theorem \ref{thm:counter} so that  
$$n_k> \frac{\log(k)}{\log\left(1+\frac{1}{k}\right) - \frac{\log(2)}{k}}$$ 
and $n_k=kd$, for a certain $d \in \N$. Let $F_1, \dots, F_k$ be $k$ linear subspaces of $\R^{n_k}$ of dimension $d$ orthogonal to each other such that $\R^{n_k} = F_1 \oplus \cdots \oplus F_k$. Set $A = I_1 \cup \cdots \cup I_k$, where for every $i \in [k]$, $I_i$ is a convex body in $F_i$. Notice that for every $l \geq 1$,
$$ \underset{l\ {\rm times}}{\underbrace{A + \cdots + A}} = \bigcup_{m_i \in \{0, \cdots, l\}, \sum_{i=1}^k m_i = l} (m_1 I_1 + \cdots + m_k I_k), $$
where we used the convexity of each $I_i$ to write the Minkowski sum of $m_i$ copies of $I_i$ as $m_i I_i$.
Thus
$$ k^{n_k} \vol_{n_k}(A(k)) = \vol_{n_k}(I_1 + \cdots + I_k) = \vol_{n_k}(I_1 \times \cdots \times I_k), $$
and 
\begin{eqnarray*}
(k+1)^{n_k} \vol_{n_k}(A(k+1)) & = & \vol_{n_k}((2I_1 + I_2 + \cdots + I_k) \cup \cdots \cup (I_1 + \cdots + I_{k-1} + 2I_k)) \\ & = & \vol_{n_k}((2I_1 \times I_2 \times \cdots \times I_k) \cup \cdots \cup (I_1 \times \cdots \times I_{k-1} \times 2I_k)) \\ & \leq & \vol_{n_k}(2I_1 \times I_2 \times \cdots \times I_k) + \cdots + \vol_{n_k}(I_1 \times \cdots \times I_{k-1} \times 2I_k) \\ & = & k2^d \vol_{n_k}(I_1 \times \cdots \times I_k) \\ & = & k^{n_k+1}2^d \vol_{n_k}(A(k)).
\end{eqnarray*}
The hypothesis on $n_k$ enables us to conclude that  $\vol_{n_k}(A(k+1)) < \vol_{n_k}(A(k))$. Now for $n\ge n_k$, we define $\tilde{A}=A\times[0,1]^{n-n_k}$. For every $l$, one has $\tilde{A}(l)=A(l)\times[0,1]^{n-n_k}$, thus $\vol_n(\tilde{A}(l)) = \vol_{n_k}(A(l))$. Therefore $\vol_n(\tilde{A}(k+1)) < \vol_n(\tilde{A}(k))$, which establishes that $\tilde{A}$ gives a counterexample in $\R^n$.

The sequence $\left\{\frac{\log(k)}{\log\left(1+\frac{1}{k}\right) - \frac{\log(2)}{k}}\right\}_{k \geq 2}$ is increasing and $\frac{\log(2)}{\log\left(1+\frac{1}{2}\right) - \frac{\log(2)}{2}} \approx 11.77$. Hence, Conjecture~\ref{weakconj} is false for $n \geq 12$.
\end{proof}

\begin{rem}
\begin{enumerate}
\item It is instructive to visualize the counterexample for $k=2$, which is done in Figure~\ref{fig:counter} by representing
each of the two orthogonal copies of $\R^6$ by a line.

\begin{figure}[h!]

\begin{center}
\includegraphics[scale=0.45]{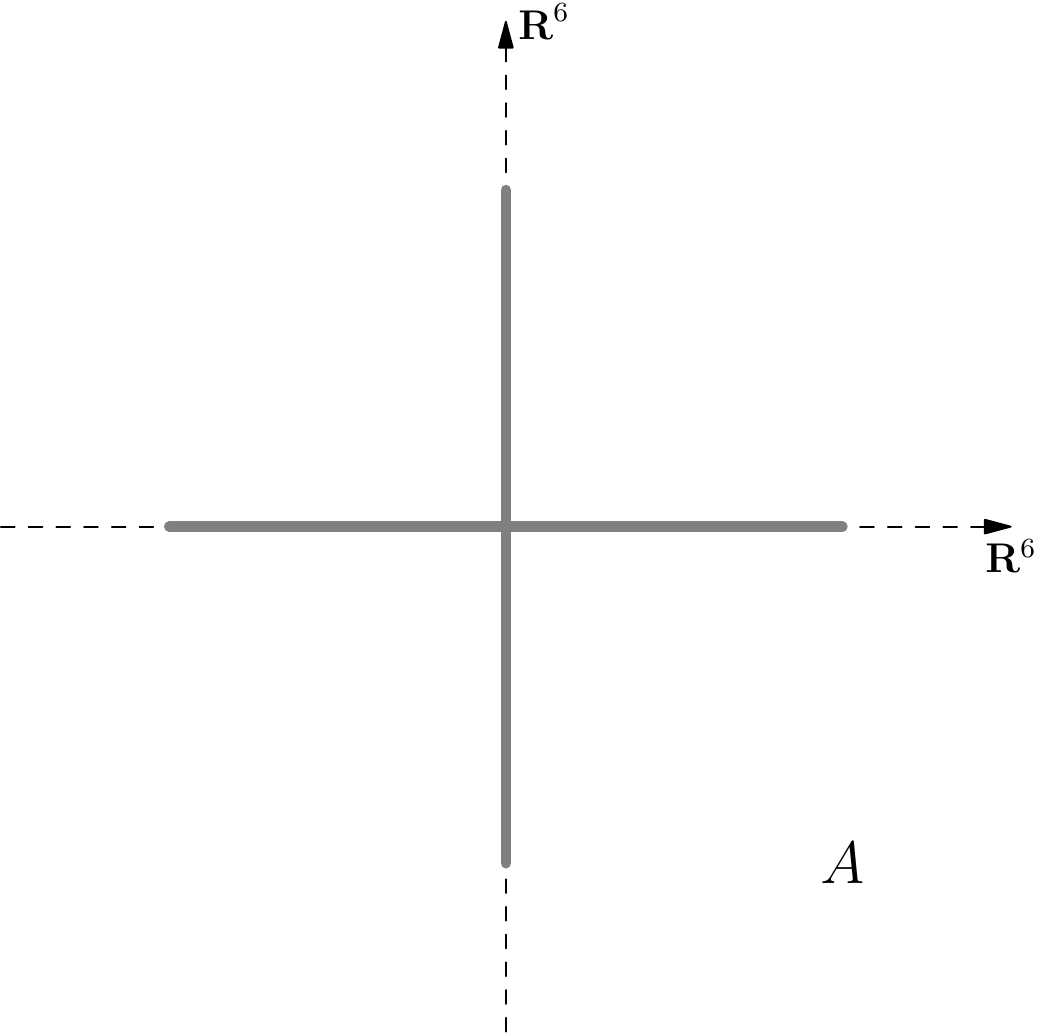} \,
\includegraphics[scale=0.45]{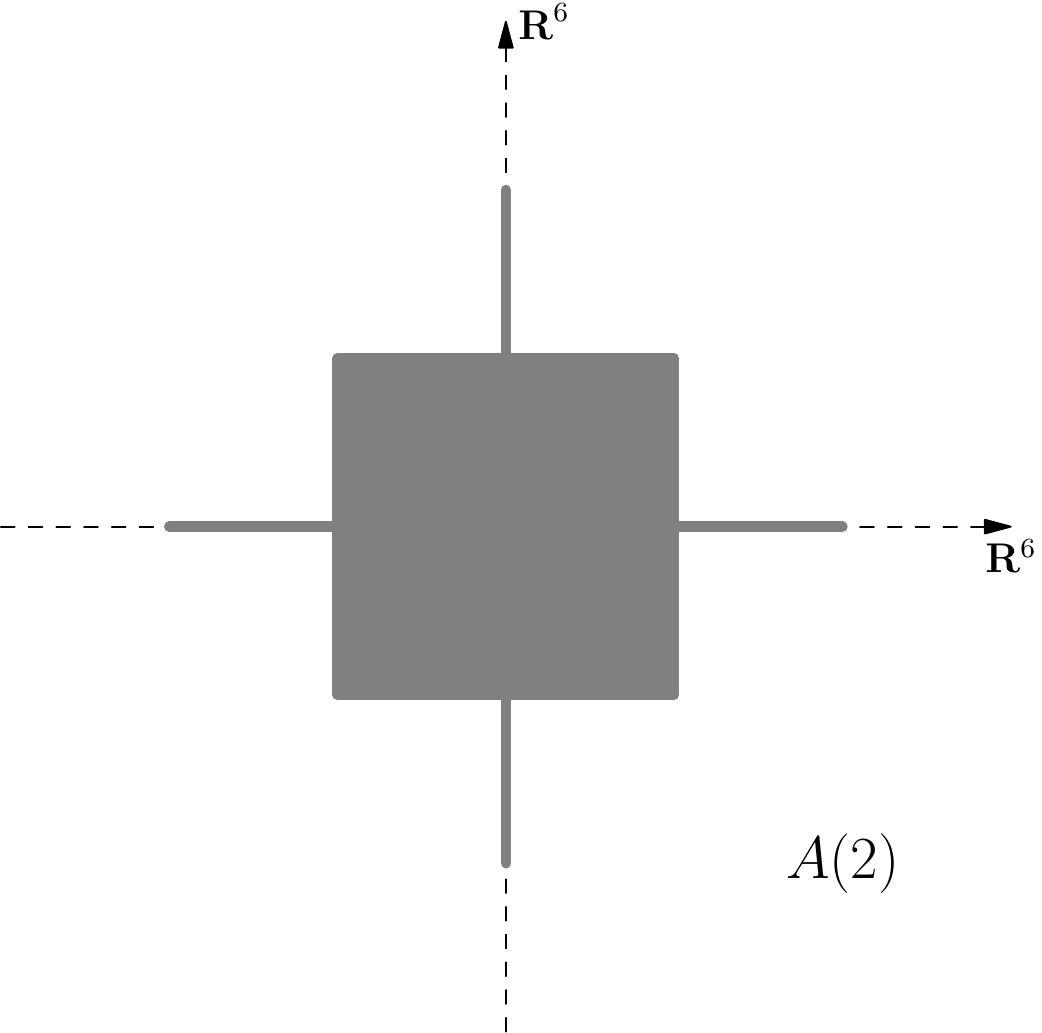}\,
\includegraphics[scale=0.45]{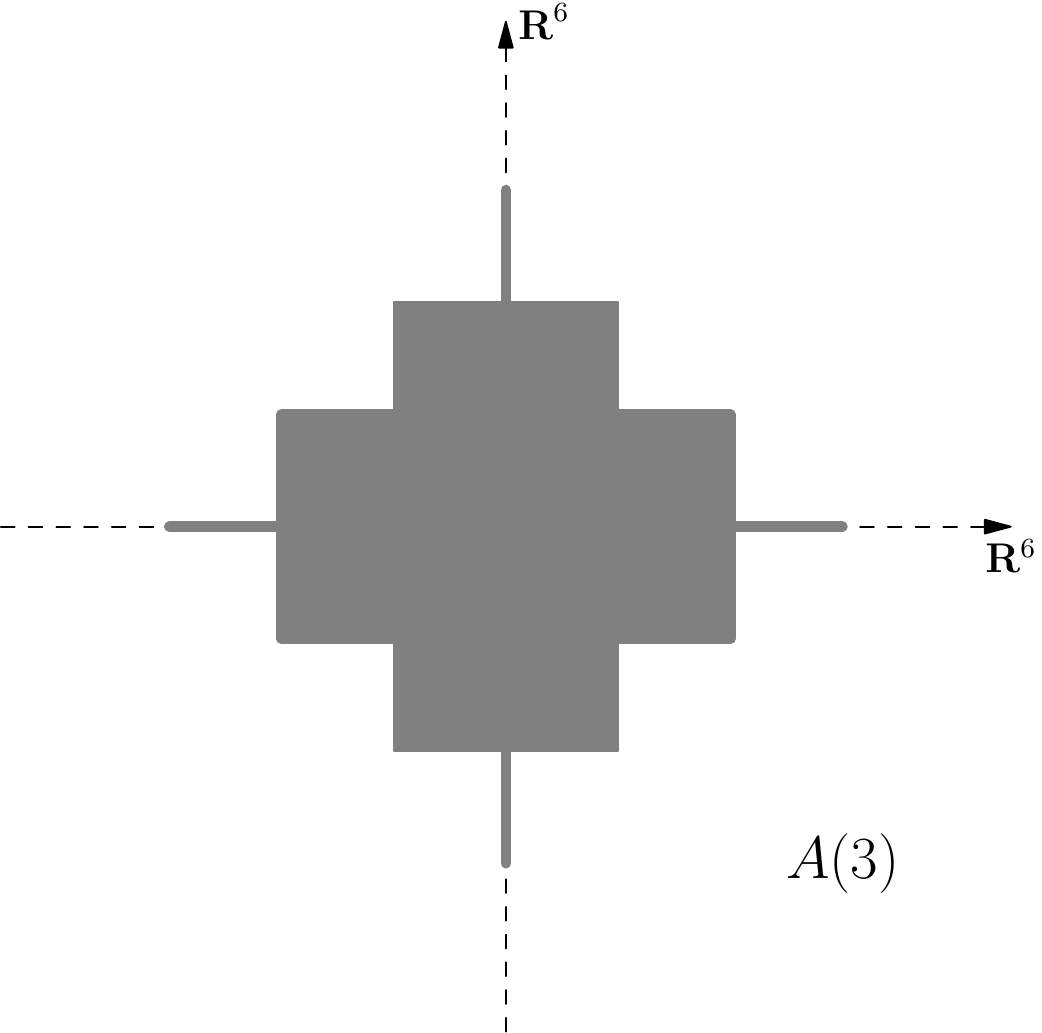}
\end{center}
\caption{A counterexample in $\R^{12}$.}\label{fig:counter}
\end{figure}

\item It was shown by Bobkov, Madiman and Wang \cite{BMW11} that Conjecture \ref{strongconj} is true for convex sets. 
The constructed counterexample is a union of convex sets and is symmetric and star-shaped.
\item Notice that in the above example one has $\vol_n(A(k-1))=0$. By adding to $A$ a ball with sufficiently small radius, one obtains a counterexample satisfying $\vol_n(A(k)) > \vol_n(A(k-1)) > 0$ and $\vol_n(A(k)) > \vol_n(A(k+1))$.
\item The counterexample also implies that Conjecture 1.1 in \cite{BMW11}, which suggests a fractional version
of Young's inequality for convolution with sharp constant, is false. It is still possible that it may be true for a 
restricted class of functions (like the log-concave functions).
\item Conjectures \ref{strongconj} and \ref{weakconj} are still open in dimension $n \in \{2, \ldots, 11\}$.
\end{enumerate}
\end{rem}

\subsection{Convergence rates for $\Delta$}
\label{sec:Delta-rate}

The asymptotic behavior of $\Delta(A(k))$ has been extensively studied by Emerson and Greenleaf \cite{EG69}.
In analyzing $\Delta(A(k))$, the following lemma about convergence of $A(k)$ to 0 in Hausdorff distance is useful.

\begin{lem}\label{lem:advance-sfs}
If $A$ is a  compact set in $\R^n$, 
\begin{eqnarray}\label{emerson}
\conv(A) \subset A(k) + \frac{n \,\diam(A)}{k}B_2^n.
\end{eqnarray}
\end{lem}

\begin{proof}
Using   invariance of  (\ref{emerson}) under the shifts of $A$, we may assume that $0 \in \conv(A)$,
$$ \conv(A) = \conv(A(k)) \subset (1+c(A(k))) \conv(A) = A(k) + c(A(k)) \conv(A). $$
Using $c(A(k)) \leq \frac{c(A)}{k}$ (see Theorem \ref{thm:c-rate} in Section~\ref{sec:c}), as well as $c(A) \leq n$ (see Theorem \ref{thm:sch75}), we deduce that
$$ \conv(A) \subset A(k) + \frac{n}{k} \conv(A). $$
To conclude, we note that since $0 \in \conv(A)$, one has $|x| \leq \diam(A)$ for every $x \in \conv(A)$. Hence, $\conv(A) \subset \diam(A) B_2^n$. 
Finally, we obtain
$$ \conv(A) \subset A(k) + \frac{n \,\diam(A)}{k}B_2^n. $$
%which is even better than the lemma since we removed the factor 2.
\end{proof}

Note that Lemma~\ref{lem:advance-sfs} is 
%essentially the same as
similar but weaker than the Shapley-Folkman-Starr theorem
discussed in the introduction, and which we will prove in Section~\ref{sec:d-rate}.
Lemma~\ref{lem:advance-sfs} was contained in \cite{EG69}, but with an extra factor of 2.

One clearly needs assumption beyond compactness to have asymptotic vanishing of $\Delta(A(k))$.
Indeed, a simple counterexample would be a finite set $A$ of points,
for which $\Delta(A(k))$ always remains at $\vol_n(\conv(A))$ and fails
to converge to 0. Once such an assumption is made, however, one has the following result.

\begin{thm}\cite{EG69}
%see Theorem 5.4 in \cite{EG69}.
Let $A$ be a compact set in $\R^n$ with nonempty interior. Then
\ben
\Delta(A(k)) \leq \frac{C}{k}  \vol_n(\conv(A)),
\een
for some constant $C$ possibly depending on $n$.
\end{thm}

\begin{proof}
%{\color{red} SOME PARTS NEED TO BE REWRITTEN?}

By translation-invariance, we may assume that $\delta B_2^n\subset A$ for some $\delta>0$. Then
$\delta B_2^n\subset A(k_0)$, and by taking $k_0 \geq \frac{n\, \diam(A)}{\delta}$,
we have 
$$\frac{n \,\diam(A)}{k_0}B_2^n\subset A(k_0).$$
Hence using (\ref{emerson}) we get
\ben
\conv(A)\subset A(k) + \frac{k_0}{k} A(k_0)=\frac{k+k_0}{k}A(k+k_0) ,
\een
so that by taking the volume we have
\ben
\vol_n(\conv(A))\leq \bigg( 1+ \frac{k_0}{k}\bigg)^n \vol_n(A(k+k_0)) ,
\een
and
\ben
\Delta(A(k+k_0)) \leq \bigg[ \bigg( 1+ \frac{k_0}{k}\bigg)^n -1 \bigg] \vol_n(A(k+k_0)) = O\bigg(\frac{1}{k}\bigg) \vol_n(\conv(A)).
\een
\end{proof}

%\subsection{Connection to additive combinatorics}
%
%
%As mentioned above in the proof of Theorem \ref{thm:1d}, the monotonicity of $\Delta(A(k))$ is related to similar questions in additive combinatorics and works of Lev \cite{Lev96} and Gyarmati, Matolcsi and Ruzsa \cite{GMR10}. Let us quote here more precisely their result.
%
%\begin{thm}[Lev; Gyarmati, Matolcsi and Ruzsa]\label{thm:GMR10}
%Let $A_1,\dots, A_k$ be finite non empty sets of integers. Let $S=A_1+\cdots +A_k$, $S_i=A_1+\cdots+A_{i-1}+A_{i+1}+\cdots+A_k$, for $1\le i\le k$. 
%Then 
%$$
%\#(S)\ge \frac{1}{k-1}\sum_{i=1}^k\#(S_i)-\frac{1}{k-1}.
%$$
%\end{thm}
%
%
%
%It was asked by Matolcsi and Ruzsa in Problem 1.7 of \cite{MR10} how to generalize the above theorem to the multidimensional case. They proposed that if $A_1, \dots, A_k$ are finite non empty sets of $\Z^n$ then a proper analog could be 
%$$
%\#(S)\ge\left( \frac{k^{n-1}}{(k-1)^n}-\varepsilon\right)\sum_{i=1}^k\#(S_i)
%$$
%if all sets are sufficiently large.
%
%
%

\section{Volume inequalities for Minkowski sums}
\label{sec:vol}

\subsection{A refined superadditivity of the volume for compact sets} %/Weaker forms of the conjectures
\label{sec:fsa}

In this section, we observe that if the exponents of $1/n$ in Conjecture~\ref{strongconj} are removed,
then the modified inequality is true (though unfortunately one can no longer directly relate this to 
a law of large numbers for sets).

\begin{thm}\label{thm:fsa}
Let $n\ge1$, $k\ge2$ be integers and let $A_1, \dots, A_k$ be $k$ compact sets in $\R^n$. Then 
\begin{eqnarray}\label{eq:thm-fsa}
\vol_n\left(\sum_{i=1}^kA_i\right) \ge \frac{1}{k-1}\sum_{i=1}^k\vol_n\left(\sum_{j\in[k]\setminus\{i\}}A_j\right).
\end{eqnarray}
\end{thm}

\begin{proof}
We use arguments similar to the proof of Theorem~\ref{thm:1d}.
Indeed, let us define the sets $S$ and $S_i$ in the same way as in the proof of Theorem~\ref{thm:1d}. 
Let $\theta\in S^{n-1}$ be any fixed unit vector and  let us define $a_i=\min\{\langle x,\theta\rangle ; x\in A_i\}$, 
$b_i=\max \{\langle x,\theta\rangle ; x\in A_i\}$, $s_i=\sum_{j<i}a_j+\sum_{j>i}b_j$, $S_i^-=\{x\in S_i ; \langle x,\theta\rangle\le s_i\}$ 
and $S_i^+=\{x\in S_i ; \langle x,\theta\rangle> s_i\}$. Then, the same inclusions hold true and thus we obtain
\begin{eqnarray*}
(k-1)\vol_n(S) 
& \ge & \sum_{i=1}^{k-1}\left(\vol_n(S_i^-) + \vol_n(S_{i+1}^+)\right) \\ 
& = & \vol_n(S_1^-) + \vol_n(S_k^+) + \sum_{i=2}^{k-1}\vol_n(S_i) \\
&=& \sum_{i=1}^k\vol_n(S_i).
\end{eqnarray*}
\end{proof}

Applying Theorem~\ref{thm:fsa} to $A_1=\cdots=A_k=A$ yields the following positive result.

\begin{cor}\label{cor:fsa-weak}
Let $A$ be a compact set in $\R^n$ and $A(k)$ be defined as in (\ref{defAk}). Then
\begin{eqnarray}\label{weak}
\vol_n(A(k)) \ge \left(\frac{k-1}{k}\right)^{n-1}\vol_n(A(k-1)).
\end{eqnarray}
\end{cor}

In the following proposition, we improve Corollary~\ref{cor:fsa-weak} under additional assumptions on the set $A \subset \R^n$, for $n \geq 2$.

\begin{prop}\label{asymptote}
Let $A$ be a compact subset of $\R^n$ and $A(k)$ be defined as in (\ref{defAk}). If there exists a hyperplane $H$ such that $\vol_{n-1}(P_H(A)) = \vol_{n-1}(P_H(\conv(A)))$, where $P_H(A)$ denotes the orthogonal projection of $A$ onto $H$, then
$$ \vol_n(A(k)) \ge \frac{k-1}{k}\vol_n(A(k-1)). $$
\end{prop}

\begin{proof}
By assumption, $\vol_{n-1}(P_H(A))=\vol_{n-1}(P_H(\conv(A)))$. Thus, for every $k \geq 1$, $\vol_{n-1}(P_H(A(k)))=\vol_{n-1}(P_H(\conv(A)))$. Indeed, one has $A \subset A(k) \subset \conv(A)$. Thus, $P_H(A) \subset P_H(A(k)) \subset P_H(\conv(A))$. Hence,
$$ \vol_{n-1}(P_H(A)) \leq \vol_{n-1}(P_H(A(k))) \leq \vol_{n-1}(P_H(\conv(A))) = \vol_{n-1}(P_H(A)). $$
It follows by the Bonnesen inequality (concave Brunn-Minkowski inequality, see \cite{BF87:book, Ohm55}) that for every $k \geq 2$,
\begin{eqnarray*}
\vol_n(A(k)) & = & \vol_n\left( \frac{k-1}{k} A(k-1) + \frac{1}{k} A \right) \\ & \geq & \frac{k-1}{k} \vol_n(A(k-1)) + \frac{1}{k} \vol_n(A) \geq \frac{k-1}{k} \vol_n(A(k-1)).
\end{eqnarray*}
\end{proof}

\begin{rem}
\begin{enumerate}
\item By considering the set $A=\{0, 1\}$ and $\delta_{\frac{1}{2}}$ the Dirac measure at $\frac{1}{2}$, one has
$$ \delta_{\frac{1}{2}}(A(2)) = 1 > 0 = \delta_{\frac{1}{2}}(A(3)). $$
Hence Conjecture~\ref{weakconj} does not hold in general for log-concave measures in dimension~1.
\item If $A$ is countable, then for every $k \geq 1$, $\vol_n(A(k)) = 0$, thus the sequence $\{\vol_n(A(k))\}_{k \geq 1}$ is constant and equal to $0$. 
\item If there exists $k_0 \geq 1$ such that $A(k_0) = \conv(A)$, then for every $k \geq k_0$, $A(k) = \conv(A)$. Indeed, 
\begin{eqnarray*}
(k_0+1)A(k_0+1) &=& k_0 A(k_0) + A = k_0 \conv(A) + A \\
& \supset & \conv(A) + k_0 A(k_0)  =  (k_0+1)\conv(A).
\end{eqnarray*}
It follows that $A(k_0+1)=\conv(A)$. We conclude by induction. Thus, in this case, the sequence $\{\vol_n(A(k))\}_{k \geq 1}$ is stationary to $\vol_n(\conv(A))$, for $k \geq k_0$. 
%\item Since $A(m) \subset A(ml)$ for every $m,l \geq 1$, one has that $\{\vol_n(A(ml^{k-1}))\}_{k \geq 1}$ is non-decreasing for every $m,l \geq 1$.
\item It is natural to ask if the refined superadditivity of volume
can be strengthened to fractional  superadditivity as defined in Definition~\ref{def:fsa} below.
While this appears to be a difficult question in general, 
it was shown recently in \cite{BMW18} %of F.~Barthe, M.~Madiman and L.~Wang
that fractional superadditivity is true in the case of compact subsets of $\R$. % when dealing with up to $k=4$ sets but only partial results were obtained for higher $k$.
\end{enumerate}
\end{rem}

\subsection{Supermodularity of volume for convex sets}
\label{sec:smod}

If we restrict to convex sets, an even stronger inequality is true from which
we can deduce Theorem~\ref{thm:fsa} for convex sets.

\begin{thm}\label{thm:smod}
Let $n\in\Nat$. For compact convex subsets $B_1, B_2, B_3$ of $\R^{n}$, one has
\be\label{set}
\vol_n(B_1+B_2+B_3) + \vol_n(B_1) \geq \vol_n(B_1+B_2) + \vol_n(B_1+B_3)  .
\ee
\end{thm}

%We will provide two separate proofs of Theorem~\ref{thm:smod}.
We first observe that Theorem~\ref{thm:smod} is actually equivalent to a 
formal strengthening of it, namely  Theorem~\ref{thm:smod2} below.
Let us first recall the notion of a supermodular set function. 

\begin{defn}
A set function $f:2^{[k]}\ra\R$ is {\it supermodular} if 
\be\label{supmod:defn}
f(\setS\cup\setT)+f(\setS\cap\setT) \geq f(\setS) + f(\setT)
\ee
for all subsets $\setS, \setT$ of $[k]$.
\end{defn}

\begin{thm}\label{thm:smod2}
Let $B_1, \ldots, B_k$ be compact convex subsets of $\R^n$, and
define 
\be\label{def:setfn-v}
v(\setS)=\vol_n\bigg(\sum_{i\in\setS} B_i \bigg) 
\ee
for each $\setS\subset [k]$.
Then $v:2^{[k]}\ra [0,\infty)$ is a supermodular set function.
\end{thm}

Theorem~\ref{thm:smod2} implies Theorem~\ref{thm:smod}, namely
\be\label{vset}
\vol_n(B_1+B_2+B_3) + \vol_n(B_1) \geq \vol_n(B_1+B_2) + \vol_n(B_1+B_3)  
\ee
for compact convex subsets $B_1, B_2, B_3$ of $\R^{n}$,
since the latter is a special case of Theorem~\ref{thm:smod2} when $k=3$.
To see the reverse, apply the inequality \eqref{vset} to
\ben
B_1=\sum_{i\in\setS\cap\setT} A_i , \quad B_2=\sum_{i\in\setS\setminus\setT} A_i, \quad B_3=\sum_{i\in\setT\setminus\setS} A_i .
\een

%The first 
Our proof of Theorem~\ref{thm:smod} combines a property of determinants that seems to have been
first explicitly observed by Ghassemi and Madiman \cite{MG17} with
a use of optimal transport inspired by Alesker, Dar and Milman \cite{ADM99}.
Let us prepare the ground by stating these results.

\begin{lem}\label{lem:det}\cite{MG17} 
Let $K_1, K_2$ and $K_3$ be $n\times n$ positive-semidefinite matrices. Then
\begin{equation*}
\det(K_1 + K_2 + K_3) + \det(K_1) \geq \det(K_1 + K_2) +\det(K_1 + K_3)  .
\end{equation*}
\end{lem}

We state the deep result of \cite{ADM99} directly for $k$ sets instead of for two sets as in \cite{ADM99}
(the proof is essentially the same, with obvious modifications). 

\begin{thm}[Alesker-Dar-Milman \cite{ADM99}]\label{thm:ADM}
Let $A_1,\ldots,A_{k} \subset \R^{n}$ be open, convex sets
with $|A_i|=1$ for each $i\in [k]$.
Then there exist $C^1$-diffeomorphisms $\psi_i:A_1\ra A_i$
preserving Lebesgue measure, such that 
\ben
\sum_{i\in [k]} \lam_i A_i = \bigg\{ \sum_{i\in [k]} \lam_i \psi_i(x): x\in A_1  \bigg\} ,
\een
for any $\lam_1,\ldots,\lam_{k} >0$.
\end{thm}

\noindent{\it Proof of Theorem~\ref{thm:smod}.}%First 
By adding a small multiple of the
Euclidean ball $B_2^n$ and then using the continuity of $\varepsilon\mapsto \vol_n(B_i+\varepsilon B_2^n)$
as $\varepsilon\ra 0$, we may assume that each of the $B_i$ satisfy $\vol_n(B_i)>0$.
Then choose $\lam_i$ such that $B_i=\lam_i A_i$ with $|A_i|=1$, so that
\ben\begin{split}
\vol_n(B_1+B_2+B_3) \, &=\, \vol_n\bigg( \sum_{i \in [3]} \lam_i A_i \bigg) 
\,=\, \int  1_{\sum_{i \in [3]} \lam_i A_i}(x) dx \\
&=\, \int  1_{\big\{ \sum_{i \in [3]} \lam_i \psi_i(y): y\in A_1  \big\} }(x) dx ,
%&\eqb \int  1_{A_1}(y) d\bigg[\sum_{i\in [M]} \lam_i \psi_i(y)\bigg] \\
\end{split}\een
using Theorem~\ref{thm:ADM}. 
Applying a change of coordinates using the diffeomorphism $x=\sum_{i\in [3]} \lam_i \psi_i(y)$,
\ben\begin{split}
V &:=\vol_n(B_1+B_2+B_3) = \int  1_{A_1}(y) \det\bigg(\sum_{i \in [3]} \lam_i D\psi_i\bigg)(y) dy \\
&\geq \int_{A_1}  \det[(\lam_1 D\psi_1+\lam_2 D\psi_2)(y)] + \det[(\lam_1 D\psi_1+\lam_3 D\psi_3)(y)] - \det[\lam_1 D\psi_1(y)] dy \\
&= \int 1_{A_1}(y) d[(\lam_1 \psi_1+\lam_2 \psi_2)(y)] + \int 1_{A_1}(y) d[(\lam_1 \psi_1+\lam_3 \psi_3)(y)]  - \int 1_{A_1}(y) d[\lam_1 \psi_1(y)] dy \\
&= \int 1_{\{ \lam_1 \psi_1(y)+\lam_2 \psi_2(y): y\in A_1\} }(z) dz + \int 1_{\{ \lam_1 \psi_1(y)+\lam_3 \psi_3(y): y\in A_1\} }(z') dz' \\
& \quad\quad\quad\quad\quad - \int 1_{\{ \lam_1 \psi_1(y): y\in A_1\} }(z'') dz'' 
\end{split}\een
where the inequality follows from Lemma~\ref{lem:det}, 
and the last equality is obtained by 
making multiple appropriate coordinate changes.
Using Theorem~\ref{thm:ADM} again, 
\ben\begin{split}
\vol_n(B_1+B_2+B_3) \, 
&\geq \int  1_{\lam_1 A_1+\lam_2 A_2}(z) dz + \int  1_{\lam_1 A_1+\lam_3 A_3}(z) dz - \int  1_{\lam_1 A_1}(z) dz \\
&= \vol_n(B_1+B_2) + \vol_n(B_1+B_3) - \vol_n(B_1) .
\end{split}\een
\qed

For the purposes of discussion below, it is useful to collect some well known
facts from the theory of supermodular set functions. 
Observe that if $v$ is supermodular and $v(\emptyset)=0$, then considering disjoint 
$\setS$ and $\setT$ in \eqref{supmod:defn} implies that 
$v$ is superadditive. In fact, a more general structural result
is true. To describe it, we need some terminology.

\begin{defn}\label{def:frac}
Given a collection $\collS$ of subsets of $[k]$, a function $\alpha:\collS \to \R^+$,
is called a {\em fractional partition}, if for each $i\in [k]$, we have
$\sum_{\setS\in \collS:i\in \setS} \as = 1$.
\end{defn}

The reason for the terminology is that this notion extends the familiar notions of 
a partition of sets (whose indicator function  can be defined
precisely as in Definition~\ref{def:frac} but with range restriction to $\{0,1\}$) 
by allowing fractional values. 
An important example of a fractional partition of $[k]$ is the collection 
$\collS_m=\binom{[k]}{m}$ of all subsets of size $m$,
together with the coefficients $\as=\binom{k-1}{m-1}^{-1}$.

\begin{defn}\label{def:fsa}
A function $f:2^{[k]}\ra\R$ is {\it fractionally superadditive} if
for any fractional partition $(\collS, \beta)$,
\ben
f([k]) \geq \sumS \bs f(\setS) .
\een
\end{defn}

The following theorem has a long history and is implicit in results from
cooperative game theory in the 1960's but to our knowledge, it was first explicitly stated
by Moulin Ollagnier and Pinchon \cite{MP82}. 

\begin{thm}\label{prop:BS}\cite{MP82}
If $f:2^{[k]}\ra\R$ is supermodular and $f(\emptyset)=0$, then $f$ is fractionally superadditive.
\end{thm}

A survey of the history of Theorem~\ref{prop:BS}, along with various
strengthenings of it and their proofs, and discussion of several applications,
can be found in \cite{MT10}.
If $\{A_i, i\in [k]\}$ are compact convex sets and $u(s)=\vol_n(\sum_{i\in s} A_i)$ as defined in \eqref{def:setfn-v},
then $u(\emptyset)=0$ and Theorem~\ref{thm:smod2} says that $u$  is supermodular,
whence Theorem~\ref{prop:BS} immediately implies
that $u$ is fractionally superadditive. 

\begin{cor}\label{cor:vol}
Let $B_1, \ldots, B_k$ be compact convex subsets of $\R^n$ and let $\beta$ 
be any fractional partition using a collection $\collS$ of subsets of $[k]$. Then
\ben
\vol_n\bigg(\sum_{i\in [k]} B_i\bigg) \geq \sumS \bs \vol_n\bigg(\sum_{i\in s} B_i\bigg) .
\een
\end{cor}

Corollary~\ref{cor:vol} implies that for each $m<k$,
\be\label{eq:vol-deg}
\vol_n\bigg(\sum_{i\in [k]} B_i\bigg) \geq \binom{k-1}{m-1}^{-1} \sum_{|\setS|=m}  \vol_n\bigg(\sum_{i\in s} B_i\bigg) .
\ee
Let us discuss whether these inequalities contain anything novel. On the one hand,
if we consider the case $m=1$ of inequality \eqref{eq:vol-deg}, the resulting inequality is not new
and in fact implied by the Brunn-Minkowski inequality:
\ben
\vol_n\bigg(\sum_{i\in [k]} B_i\bigg) \geq \big[ \sum_{i\in [k]} \vol_n(B_i)^\nth \big]^n 
\geq \sum_{i\in [k]} \vol_n(B_i)  .
\een
%\be\label{sup-appl}\begin{split}
%\vol_n(A+B)^\nth &\geq \big[ \vol_n(A) + \vol_n(B) \big]^\nth 
%\ge 2^{\nth -1} \big[\vol_n(A)^\nth + \vol_n(B)^\nth\big] ,
%\end{split}\ee
%which is the , but with an extra factor that ranges
%from 1 to $\half$ as the dimension goes from 1 to $\infty$. In other words,
%the Brunn-Minkowski inequality is stronger than the superadditivity \eqref{sup-appl}.
On the other hand, applying the inequality \eqref{eq:vol-deg} to $m=k-1$ yields
precisely Theorem~\ref{thm:fsa} for convex sets $B_i$, i.e.,
\be\label{f-sup-appl}\begin{split}
\vol_n \bigg(\sum_{i\in [k]}  B_i\bigg) &\geq \frac{1}{k-1} \sum_{i\in [k]}  \vol_n\bigg(\sum_{j\neq i}  B_j\bigg).
\end{split}\ee
Let us compare this with what is obtainable 
from the refined Brunn-Minkowski inequality for convex sets proved in \cite{BMW11},
which says that
\be\label{bm-sup-appl}\begin{split}
\vol_n\bigg(\sum_{i\in [k]}  B_i\bigg) &\geq \bigg(\frac{1}{k-1}\bigg)^n \bigg[\sum_{i\in [k]}  \vol_n\big(\sum_{j\neq i}  B_j\big)^\nth \bigg]^n .
\end{split}\ee
Denote the right hand sides of \eqref{f-sup-appl} and \eqref{bm-sup-appl} by $R_{\eqref{f-sup-appl}}$
and $R_{\eqref{bm-sup-appl}}$. Also set 
\ben
c_i=\vol_n\bigg(\sum_{j\neq i}  B_j\bigg)^\nth ,
\een
and write $c=(c_1, \ldots, c_k)\in [0,\infty)^k$, so that
%\ben
$R_{\eqref{f-sup-appl}}^{\nth} =(k-1)^{-\nth} \|c\|_n$
%\quad{\rm and}\quad
and
$R_{\eqref{bm-sup-appl}}^{\nth} = (k-1)^{-1} \|c\|_1$.
%\een
Here, for $m\geq 1$, $\|c\|_m = \left( \sum_{i=1}^k c_i^m \right)^{\frac{1}{m}}$. In other words,
\ben
\bigg[\frac{R_{\eqref{f-sup-appl}}}{R_{\eqref{bm-sup-appl}}} \bigg]^\nth
= (k-1)^{1-\nth} \frac{\|c\|_n}{\|c\|_1} .
\een
Let us consider $n=2$ for illustration. Then we have
\ben
\bigg[\frac{R_{\eqref{f-sup-appl}}}{R_{\eqref{bm-sup-appl}}} \bigg]^\half
= \sqrt{k-1} \frac{\|c\|_2}{\|c\|_1} ,
\een
which ranges between $\sqrt{1-\frac{1}{k}}$ and $\sqrt{k-1}$,
since $\|c\|_2/\|c\|_1 \in [k^{-\half} ,1]$. In particular, neither bound is uniformly better;
so the  inequality \eqref{eq:vol-deg} and Corollary~\ref{cor:vol} do indeed have
some potentially useful content.

Motivated by the results of this section, it is natural to ask if the volume of Minkowski 
sums is supermodular even without the convexity assumption on the sets involved,
as this would strengthen Theorem~\ref{thm:fsa}. In fact, this is not the case.

\begin{prop}\label{prop:smod-counter}
There exist compact sets $A, B, C \subset \R$ such that
\ben
\vol_1(A+B+C) + \vol_1(A) < \vol_1(A+B) + \vol_1(A+C) .
\een
\end{prop}

\begin{proof}
Consider $A=\{0,1\}$ and $B=C=[0,1]$. Then,
$$ \vol_1(A+B+C) + \vol_1(A) = 3 < 4 = \vol_1(A+B) + \vol_1(A+C). $$
\end{proof}

On the other hand, the desired inequality is true in dimension 1 if the set $A$ is convex. 
More generally, in dimension 1, one has the following result.

\begin{prop}\label{prop:smod-1d}
If $A, B, C\subset \R$ are compact, then
\ben
\vol_1(A+B+C) +\vol_1(\conv(A))\ge \vol_1(A+B) + \vol_1(A+C) .
\een
\end{prop}

\begin{proof}
Assume, as one typically does in the proof of the one-dimensional Brunn-Minkowski inequality,  
that $\max B=0=\min C$. (We can do this without loss of generality since translation does not 
affect volumes.) This implies that $B\cup C\subset B+C$, whence 
\ben
(A+B)\cup(A+C)=A+(B\cup C)\subset A+B+C. 
\een
Hence
\ben\begin{split}
\vol_1(A+B+C) &\ge \vol_1((A+B)\cup(A+C))\\
&= \vol_1(A+B)+\vol_1(A+C)-\vol_1((A+B)\cap(A+C)) .
\end{split}\een
We will show that $(A+B)\cap(A+C)\subset \conv(A)$, which together with the preceding
inequality yields the desired conclusion
$\vol_1(A+B+C)\ge \vol_1(A+B)+\vol_1(A+C)- \vol_1(\conv(A))$.

To see that $(A+B)\cap(A+C)\subset \conv(A)$, consider $x\in(A+B)\cap(A+C)$. 
One may write $x=a_1+b=a_2+c$, with $a_1,a_2\in A$, $b\in B$ and $c\in C$. 
Since $\max B=0=\min C$ one has $b\le 0\le c$ and one deduces that $ a_2\le x\le a_1$ 
and thus $x\in \conv(A)$. This completes the proof.
%So we have $(A+B)\cap(A+C)\subset \conv(A)$ and so $|(A+B)\cap(A+C)|\le |\conv(A)|$. 
\end{proof}

\begin{rem}
\begin{enumerate}
\item One may wonder if Proposition \ref{prop:smod-1d} extends to higher dimension. More particularly, we do
not know if the supermodularity inequality 
$$\vol_n(A+B+C) +\vol_n(A)\ge \vol_n(A+B) + \vol_n(A+C) $$
holds true in the case where $A$ is  convex and $B$ and $C$ are any compact sets.

\item It is also natural to ask in view of the results of this section whether the fractional superadditivity
\eqref{conjdimn} of $\vol_n^{1/n}$ for convex sets proved in \cite{BMW11} follows from a more general supermodularity
property, i.e., whether
\be\label{eq:q-supmod}
\vol_n^{1/n}(A+B+C) + \vol_n^{1/n}(A) \geq \vol_n^{1/n}(A+B) + \vol_n^{1/n}(A+C)
\ee
for convex sets $A, B, C\subset \R^n$. It follows from results of \cite{MG17} that such a result
does not hold (their counterexample to the determinant version of \eqref{eq:q-supmod} corresponds in our context
to choosing ellipsoids in $\R^2$). Another simple explicit counterexample is the following:
Let 
$A = [0,2] \times [0, 1/2]$, 
$B= [0,1/2] \times [0,2]$,
and $C=\eps B_2^2$, with $\eps > 0$. Then,
$$ \vol_2(A)^{1/2} = 1, \quad \vol_2(A+B+C)^{1/2} = \sqrt{25/4 + 10\eps + \pi \eps^2}, $$
$$ \vol_2(A+B)^{1/2} = 5/2, \quad \vol_2(A+C)^{1/2} = \sqrt{1 + 5\eps + \pi \eps^2}. $$
Hence,
\ben\begin{split}
\vol_2(A+B+C)^{1/2} + \vol_2(A)^{1/2} &= 1 + 5/2 + 2\eps + o(\eps) \\
\vol_2(A+B)^{1/2} + \vol_2(A+C)^{1/2} &= 1 + 5/2 + (5/2)\eps + o(\eps)
\end{split}\een
For $\eps$ small enough, this yields a counterexample to \eqref{eq:q-supmod}.

\item It is shown in \cite{MG17} that the entropy analogue of Theorem \ref{thm:smod} 
does not hold, i.e., there exist independent real-valued random variables $X, Y, Z$ with log-concave distributions
such that
\ben
e^{2h(X+Y+Z)}+e^{2h(Z)} < e^{2h(X+Z)}+e^{2h(Y+Z)}.
\een

%\item It is shown in \cite{Ruz97} that there exist variants of the Brunn-Minkowski
%inequality where a lower bound for $|A+B|$ is given in terms of $|A|, |B|$ and $|\conv(B)|$;
%these are shown to be better than the Brunn-Minkowski inequality when $B$ is non-convex
%and $A$ has large volume. 

\end{enumerate}
\end{rem}

\section{The behavior of Schneider's non-convexity index}
\label{sec:c}

In this section we study Schneider's non-convexity index. Recall its definition: for $A$ compact in $\R^n$,
\ben
c(A) = \inf \{ \lam\geq 0: A+\lam\, \conv(A) \text{ is convex}  \}.
\een 

\subsection{The refined monotonicity of Schneider's non-convexity index}
\label{sec:c-mono}

In this section, our main result is that Schneider's non-convexity index $c$ satisfies a strong kind of monotonicity in any dimension. 

%We start  by proving a lemma.
%
%\begin{lem}\label{lem:sumconvex}
%Let $n\ge1$ and $k\ge2$ be integers and let $A_1, \dots, A_k$ be $k$ sets in $\R^n$. If there exists $i\neq j\in[k]$ such that $\sum_{l\in[k]\setminus\{i\}}A_l$ and  $\sum_{l\in[k]\setminus\{j\}}A_l$ are convex then  $\sum_{l\in[k]}A_l$ is convex.
%\end{lem}
%
%\begin{proof}
%We have 
%\begin{eqnarray*}
%\sum_{l\in[k]}A_l &=& A_i+\sum_{l\in[k]\setminus\{i\}}A_l=A_i+\conv\left(\sum_{l\in[k]\setminus\{i\}}A_l\right)=A_i+\sum_{l\in[k]\setminus\{i\}}\conv(A_l)\\
%&\supset& \conv(A_j)+\sum_{l\in[k]\setminus\{j\}}A_l=\conv(A_j)+\conv\left(\sum_{l\in[k]\setminus\{j\}}A_l\right)=\conv\left(\sum_{l\in[k]}A_l\right).
%\end{eqnarray*}
%Since the other inclusion is obvious, we deduce that $\sum_{l\in[k]}A_l$ is convex.
%\end{proof}

We state the main theorem of this section, and will subsequently deduce corollaries 
asserting monotonicity in the Shapley-Folkman-Starr theorem from it.

\begin{thm}\label{thm:fracsubofc}
Let $n\ge1$ and let $A, B, C$ be subsets of $\R^n$. Then
\ben
c(A+B+C) \leq \max\{c(A+B), c(B+C)\}.
\een
\end{thm}

\begin{proof}
Let us denote $\lam=\max\{c(A+B), c(B+C)\}$. Then 
\begin{eqnarray*}
A+B+C+\lam\conv(A+B+C)&=&A+B+\lam\conv(A+B)+C+\lam\conv(C)\\
&=&(1+\lam)\conv(A+B)+C+\lam\conv(C)\\
&\supset& (1+\lam)\conv(A)+B+\lam\conv(B)+C+\lam\conv(C)\\
&=&(1+\lam)\conv(A)+ (1+\lam)\conv(B+C)\\
&=& (1+\lam)\conv(A+B+C).
\end{eqnarray*}
Since the opposite inclusion is clear, we deduce that $A+B+C+\lam\conv(A+B+C)$ is convex, which means that $c(A+B+C) \leq \lam=\max\{c(A+B), c(B+C)\}.$
\end{proof}

Notice that the same kind of proof also shows that if $A+B$ and $B+C$ are convex then $A+B+C$ is also convex. Moreover, Theorem~\ref{thm:fracsubofc} has an equivalent formulation for  $k\ge2$ 
subsets of $\R^n$, say $A_1, \dots, A_k$: if $s, t\subset [k]$
with $s\cup t=[k]$, then
\be\label{eq:c-multi}
c\left(\sum_{i\in[k]}A_i\right) \leq \max \bigg\{ c\left(\sum_{i\in s}A_i\right), c\left(\sum_{i\in t}A_i\right)  \bigg\}.
\ee
To see this, apply Theorem~\ref{thm:fracsubofc}  to
\ben
B=\sum_{i\in\setS\cap\setT} A_i , \quad A=\sum_{i\in\setS\setminus\setT} A_i, \quad C=\sum_{i\in\setT\setminus\setS} A_i .
\een

From the inequality \eqref{eq:c-multi}, the following corollary, expressed in a more
symmetric fashion, immediately follows.

\begin{cor}\label{cor:fracsubofc}
Let $n\ge1$ and $k\ge2$ be integers and let $A_1, \dots, A_k$ be $k$ sets in $\R^n$. Then
$$
c\left(\sum_{l\in[k]}A_l\right) \leq \max_{i\in[k]}c\left(\sum_{l\in[k]\setminus\{i\}}A_l\right).
$$
\end{cor}

The $k=2$ case of Corollary~\ref{cor:fracsubofc} follows directly from the definition of $c$
and was observed by Schneider in \cite{Sch75}.
Applying Corollary~\ref{cor:fracsubofc} for $A_1=\cdots =A_k=A$, where $A$ is a fixed subset of $\R^n$, 
and using the scaling invariance of $c$, one deduces that the sequence $c(A(k))$ is non-increasing.
In fact, for identical sets, we prove something even stronger in the following theorem.

\begin{thm}\label{thm:c-quant-mono}
Let $A$ be a subset of $\R^n$ and $k\ge2$ be an integer. Then 
\ben
c\left(A(k)\right)\le\frac{k-1}{k}c\left(A(k-1)\right).
\een
\end{thm}

\begin{proof}
Denote $\lam=c\left(A(k-1)\right)$. Since $\conv(A(k-1))=\conv(A)$, from the definition of $c$, one knows that 
$A(k-1)+\lam\conv(A)=\conv(A)+\lam\conv(A)=(1+\lam)\conv(A)$. Using that $A(k)=\frac{A}{k}+\frac{k-1}{k}A(k-1)$, one has 
\begin{eqnarray*}
A(k)+\frac{k-1}{k}\lam\conv(A) &=& \frac{A}{k}+\frac{k-1}{k}A(k-1)+\frac{k-1}{k}\lam\conv(A)\\
&=&\frac{A}{k}+\frac{k-1}{k}\conv(A)+\frac{k-1}{k}\lam\conv(A)\\
&\supset&\frac{\conv(A)}{k}+\frac{k-1}{k}A(k-1)+\frac{k-1}{k}\lam\conv(A)\\
&=&\frac{\conv(A)}{k}+\frac{k-1}{k}(1+\lam)\conv(A)\\
&=&\left(1+\frac{k-1}{k}\lam\right)\conv(A).
\end{eqnarray*}
Since the other inclusion is trivial, we deduce that $A(k)+\frac{k-1}{k}\lam\conv(A)$ is convex which proves that 
$$c(A(k))\le \frac{k-1}{k}\lam=\frac{k-1}{k}c\left(A(k-1)\right).$$
\end{proof}

%Recall that for any set $A$ in $\R^n$, 
%$c(A)= \inf \{ \lam\geq 0: A_\lam \text{ is convex}  \}$, 
%where $A_\lam=\frac{A+\lam\conv(A)}{1+\lam}$. 
%We can also ask the following natural question. Define
%\ben
%\lam_k(A)= \inf \{ \lam\geq 0: A_\lam \supset A(k) \} ,
%\een
%We should be able to show using Shapley-Folkman that $\lam_k(A)\ra c(A)$ as $k\ra\infty$.
%Do we also have monotonicity of $\lam_k(A)$ in $k$?
%
%Consider $A=\{0,1\}$. Then, for $\lambda \in (0,1)$,
%\ben
%A_{\lambda} = \bigg[0, \frac{\lambda}{1+\lambda}\bigg] \cup \bigg[\frac{1}{1+\lambda}, 1\bigg]. 
%\een
%Notice that 
%$\frac{\lambda}{1+\lambda} < 1/2 < \frac{1}{1+\lambda}$. 
%Notice also that if $k$ is even, then $1/2 \in A(k)$ and if $k$ is odd then $1/2 \notin A(k)$. 
%Thus, if $k$ is even then $\lambda_k(A)=1$, and if $k$ is odd then $\lambda_k(A)<1$. 
%We conclude that the sequence $(\lambda_k(A))$ is not monotone.
%
%A remark:
%Since $A(k) \subset \frac{A + (k-1)\conv(A)}{k}$, then $\lambda_k(A) \leq k-1$. And by definition, $\lambda_k(A) \leq c(A)$. Hence, 
%\ben
%\lambda_k(A) \leq \min(k-1, c(A)).
%\een

\begin{rem}
\begin{enumerate}
\item We do not know if $c$ is fractionally subadditive; for example, we do not know if 
$2\,c(A+B+C)\le c(A+B)+c(A+C)+c(B+C)$. We know it with a better constant if $A=B=C$, 
as a consequence of Theorem \ref{thm:c-quant-mono}. We also know it if we take a large enough number of sets; this is a consequence of the Shapley-Folkman lemma (Lemma~\ref{lem:SF}).
\item The Schneider index $c$ (as well as any other measure of non-convexity) cannot be submodular.
This is because, if we consider $A=\{0,1\}$, $B=C=[0,1]$, then $c(A+B)=c(A+C)=c(A+B+C)=0$ but $c(A)>0$, hence
\ben
c(A+B+C) + c(A) > c(A+B) + c(A+C).
\een
\end{enumerate}
\end{rem}

\subsection{Convergence rates for Schneider's non-convexity index}
\label{sec:c-rate}

We were unable to find any examination in the literature of rates, or indeed, even of 
sufficient conditions for convergence as measured by $c$.

Let us discuss convergence in the Shapley-Folkman-Starr theorem
using the Schneider non-convexity index. In dimension 1, we can get an
$O(1/k)$ bound on $c(A(k))$ by using the close relation \eqref{eq:c-d-1d} between
$c$ and $d$ in this case. In general dimension, the same bound also holds: by applying Theorem~\ref{thm:c-quant-mono} inductively, we get the following theorem.

\begin{thm}\label{thm:c-rate}
Let $A$ be a compact set in $\R^n$. Then
\ben
c(A(k)) \leq \frac{c(A)}{k} .
\een
In particular, $c(A(k))\ra 0$ as $k\ra\infty$.
\end{thm}

%\begin{proof}
%By repeated application of Theorem~\ref{thm:c-quant-mono}, 
%\ben\begin{split}
%c(A(k)) &\leq \frac{k-1}{k} c(A(k-1)) 
%\leq \frac{k-1}{k} \frac{k-2}{k-1} c(A(k-2)) \leq \ldots \\
%&\leq \frac{k-1}{k} \frac{k-2}{k-1}\ldots \frac{1}{2} c(A) 
%= \frac{1}{k} c(A) .
%\end{split}\een
%\end{proof}

Let us observe that the $O(1/k)$ rate of convergence cannot be improved,
either for $d$ or for $c$. To see this simply consider the case where $A= \{0,1\}\subset \R$.
Then $A(k)$ consists of the $k+1$ equispaced points $j/k$, where $j\in \{0, 1, \ldots, k\}$,
and $c(A(k))=2d(A(k))= 1/k$ for every $k\in\Nat$.

\section{The behavior of the effective standard deviation $v$}
\label{sec:r}

In this section we study the effective standard deviation $v$. Recall its definition: for $A$ compact in $\R^n$,
\ben
v^2(A) =  \sup_{x\in\conv(A)}\inf\{\sum p_i |a_i-x|^2: x=\sum p_i a_i; p_i >0; \sum p_i=1, a_i \in A  \}.
\een

\subsection{Subadditivity of $v^2$}

% The following alternate definition of $r$ is useful.
%{\color{red} DO WE NEED THIS? EVEN IF NOT, IS IT USEFUL GENERAL INFO TO INCLUDE?}
%\begin{lem}\label{lem:r-simplex}
%For any compact set $A\subset \R^n$,
%\ben
%r(A) = \sup_{T\subset A, |T|=n+1, \relint(\conv(A))} R(T) .
%\een
%\end{lem}

Cassels \cite{Cas75} showed that $v^2$ is subadditive.

\begin{thm}[\cite{Cas75}]\label{thm:v-subadd}
Let $A,B$ be compact sets in $\R^n$. Then,
\ben
v^2(A+B) \leq v^2(A)+v^2(B).
\een
\end{thm}

\begin{proof}

Recall that $v(A)=\sup_{x\in\conv(A)}v_A(x)$, where 
$$v_A^2(x)=\inf\{\sum_{i\in I} \lambda_i |a_i-x|^2:  (\lambda_i,a_i)_{i\in I}\in\Theta_A(x)\}, $$
and $\Theta_A(x)=\{(\lambda_i,a_i)_{i\in I}: I\ \hbox{finite},\ x=\sum \lambda_i a_i; \lambda_i>0; \sum \lambda_i=1, a_i \in A  \}.$
Thus 
$$v(A+B)=\sup_{x\in\conv(A+B)}v_{A+B}(x)=\sup_{x_1\in\conv(A)}\sup_{x_2\in\conv(B)}v_{A+B}(x_1+x_2).$$
And one has 
$$v_{A+B}^2(x_1+x_2)=\inf\{\sum_{i\in I} \nu_i |c_i-x_1-x_2|^2:  (\nu_i,c_i)_{i\in I}\in\Theta_{A+B}(x_1+x_2)\}.$$
For $(\lambda_i,a_i)_{i\in I}\in\Theta_A(x_1)$ and $(\mu_j,b_j)_{j\in J}\in\Theta_B(x_2)$ one has 
$$(\lambda_i\mu_j,a_i+b_j)_{(i,j)\in I\times J}\in\Theta_{A+B}(x_1+x_2), $$ 
and 
%\ben\label{eq:var-add}\begin{split}
%&\sum_{(i,j)\in I\times J}\lambda_i\mu_j |a_i+b_j-x_1-x_2|^2 \\
%&=\sum_{i\in I} \lambda_i |a_i-x_1|^2+\sum_{j\in J} \mu_j |b_j-x_2|^2 + 2\sum_{(i,j)\in I\times J}\lambda_i\mu_j \langle a_i-x_1, b_j-x_2\rangle \\
%&=\sum_{i\in I} \lambda_i |a_i-x_1|^2+\sum_{j\in J} \mu_j |b_j-x_2|^2 + 2 \langle \sum_{i\in I} \lambda_i a_i-x_1, \sum_{j\in J}\mu_j b_j-x_2\rangle \\
%&= \sum_{i\in I} \lambda_i |a_i-x_1|^2+\sum_{j\in J} \mu_j |b_j-x_2|^2 .
%\end{split}\een
\begin{equation}\label{eq:var-add}
\begin{aligned}
&\sum_{(i,j)\in I\times J}\lambda_i\mu_j |a_i+b_j-x_1-x_2|^2 \\
&=\sum_{i\in I} \lambda_i |a_i-x_1|^2+\sum_{j\in J} \mu_j |b_j-x_2|^2 + 2\sum_{(i,j)\in I\times J}\lambda_i\mu_j \langle a_i-x_1, b_j-x_2\rangle \\
&=\sum_{i\in I} \lambda_i |a_i-x_1|^2+\sum_{j\in J} \mu_j |b_j-x_2|^2 + 2 \langle \sum_{i\in I} \lambda_i a_i-x_1, \sum_{j\in J}\mu_j b_j-x_2\rangle \\
&= \sum_{i\in I} \lambda_i |a_i-x_1|^2+\sum_{j\in J} \mu_j |b_j-x_2|^2 .
\end{aligned}
\end{equation}
Thus 
\ben\begin{split}
v_{A+B}^2(x_1+x_2)&\le \inf_{(\lambda_i,a_i)_{i\in I}\in\Theta_A(x_1)}\inf_{(\mu_j,b_j)_{j\in J}\in\Theta_B(x_2)}\sum_{i\in I} \lambda_i |a_i-x_1|^2+\sum_{j\in J} \mu_j |b_j-x_2|^2 \\
&=v_A^2(x_1)+v_B^2(x_2).
\end{split}\een
Taking the supremum in $x_1\in\conv(A)$ and $x_2\in\conv(B)$, we conclude.
%One has
%\ben\begin{split}
%v^2(A+B) &= \sup_{c^* \in \conv(A+B)} \inf_{\{X: \E X=c^*, X\in A+B\}} V(X) \\
%&= \sup_{\substack{c_1^* \in \conv(A)\\ c_2^* \in \conv(B)}} \inf_{\{(X_1, X_2): \E X_1=c_1^*, \E X_2=c_2^*, X_1\in A, X_2\in B\}} V(X_1 + X_2) 
%\end{split}\een
%Taking the infimum on a larger set only makes the quantity larger so that 
%\ben\begin{split}
%v^2(A+B) &\leq \sup_{c_1^* \in \conv(A), c_2^* \in \conv(B)} \inf_{\{X_1, X_2: \langle X_1, X_2\rangle=0,  \E X_1=c_1^*, \E X_2=c_2^*, X_1\in A, X_2\in B\}} V(X_1) + V(X_2) \\
%&= \sup_{c_1^* \in \conv(A), c_2^* \in \conv(B)} \inf_{\{X_1: \E X_1=c_1^*, X_1\in A\}} \bigg[ V(X_1) + \inf_{\{X_2: X_2\perp X_1, \E X_2=c_2^*, X_2\in B\}} V(X_2) \bigg] \\
%&= \sup_{c_1^* \in \conv(A), c_2^* \in \conv(B)} 
%\bigg[ \inf_{\{X_1: \E X_1=c_1^*, X_1\in A\}} V(X_1) + \inf_{\{X_2: \E X_2=c_2^*, X_2\in B\}} V(X_2) \bigg] \\
%&= \sup_{c_1^* \in \conv(A)}  \inf_{\{X_1: \E X_1=c_1^*, X_1\in A\}} V(X_1) + \sup_{c_2^* \in \conv(B)} \inf_{\{X_2: \E X_2=c_2^*, X_2\in B\}} V(X_2) \\
%&= v^2(A) + v^2(B) .
%\end{split}\een
\end{proof}

Observe that we may interpret the proof probabilistically. Indeed, a key point in the proof is the identity \eqref{eq:var-add},
which is just the fact that the variance of a sum of independent random variables is the sum of the individual variances 
(written out explicitly for readability).

\subsection{Strong fractional subadditivity for large $k$}

In this section, we prove that the effective standard deviation $v$ satisfies a strong fractional subadditivity when considering sufficient large numbers of sets.

\begin{thm}\label{v-large-k}
Let $A_1, \dots, A_k$ be compact sets in $\R^n$, with $k \geq n + 1$. Then,
$$ v\left(\sum_{i \in [k]} A_i\right) \leq \max_{I \subset [k]: |I| \leq n} \min_{i \in [k] \setminus I} v\left(\sum_{j \in [k] \setminus \{i\}} A_j \right). $$
\end{thm}

\begin{proof}
Let $x \in \conv(\sum_{i\in [k]} A_i)$, where $k \geq n + 1$. By using the Shapley-Folkman lemma (Lemma~\ref{lem:SF}), there exists a set $I$ of at most $n$ indexes such that
$$
x \in \sum_{i\in I} \conv(A_i) + \sum_{i \in [k] \setminus I} A_i.
$$
Let $i_0 \in [k] \setminus I$. In particular, we have
$$ x \in \conv\bigg(\sum_{i \in [k] \setminus \{i_0\}} A_i\bigg) + A_{i_0}. $$
Hence, by definition of the convex hull,
$$ x = \sum_m p_m a_m + a_{i_0} = z + a_{i_0}, $$
where $z=\sum_m p_m a_m$, $\sum_m p_m =1$, $a_m \in \sum_{i \in [k] \setminus \{i_0\}} A_i$ and $a_{i_0} \in A_{i_0}$. Thus, by denoting $A_{\{i_0\}} = \sum_{i \in [k] \setminus \{i_0\}} A_i$, we have
\begin{eqnarray*}
v_{A_{\{i_0\}}}^2(z) & = & \inf \bigg\{ \sum_m p_m |a_m - z|^2 : z = \sum_m p_m a_m; \sum_m p_m = 1; a_m \in A_{ \{i_0\} } \bigg\} \\ 
& = & \inf \bigg\{ \sum_m p_m |a_m + a_{i_0} - (z + a_{i_0})|^2 : z = \sum_m p_m a_m; \sum_m p_m = 1; a_m \in A_{ \{i_0\}} \bigg\} \\ 
& \geq & \inf \bigg\{ \sum_m p_m |a_m^* - (z + a_{i_0})|^2 : z + a_{i_0} = \sum_m p_m a_m^* ; \sum_m p_m = 1 ; a_m^* \in \sum_{i \in [k]} A_i \bigg\} \\ 
& = & v_{\sum_{i \in [k]} A_i}^2(x).
\end{eqnarray*}
Taking supremum over all $z \in \conv(\sum_{i \in [k] \setminus \{i_0\}} A_i)$, we deduce that
$$ v_{\sum_{i \in [k]} A_i}(x) \leq v\bigg(\sum_{i \in [k] \setminus \{i_0\}} A_i\bigg). $$
Since this is true for every $i_0 \in [k] \setminus I$, we deduce that
$$ v_{\sum_{i \in [k]} A_i}(x) \leq \min_{i \in [k] \setminus I} v\bigg(\sum_{j \in [k] \setminus \{i\}} A_j\bigg). $$
Taking the supremum over all set $I \subset [k]$ of cardinality at most $n$ yields
$$ v_{\sum_{i \in [k]} A_i}(x) \leq \max_{I \subset [k]: |I| \leq n} \min_{i \in [k] \setminus I} v\bigg(\sum_{j \in [k] \setminus \{i\}} A_j\bigg). $$
We conclude by taking the supremum over all $x \in \conv(\sum_{i\in [k]} A_i)$.

\end{proof}

An immediate consequence of Theorem~\ref{v-large-k} is that if $k \geq n + 1$, then
$$ v\left(\sum_{i \in [k]} A_i\right) \leq \max_{i \in [k]}  v\left(\sum_{j \in [k] \setminus \{i\}} A_j \right). $$
By iterating this fact as many times as possible (i.e., as long as the number of sets is at least $n+1$), we obtain
the following corollary.

\begin{cor}\label{cor:v-large-k}
Let $A_1, \dots, A_k$ be compact sets in $\R^n$, with $k \geq n + 1$. Then,
$$ v\left(\sum_{i \in [k]} A_i\right) \leq \max_{I \subset [k]: |I| = n} v\left(\sum_{j \in I} A_j \right). $$
\end{cor}

In the case where $A_1 = \cdots = A_k = A$, we can repeat the above argument with $k\ge c(A)+1$ to prove that in this case,
$$ v(A(k)) \leq \frac{k-1}{k} v(A(k-1)), $$
where $c(A)$ is the Schneider non-convexity index of $A$. Since $c(A) \leq n$, and $c(A) \leq n-1$ when $A$ is connected, we deduce the following monotonicity property for the effective standard deviation.

\begin{cor}
\begin{enumerate}
\item In dimension 1 and 2, the sequence $v(A(k))$ is non-increasing for every compact set $A$.
\item In dimension 3, the sequence $v(A(k))$ is non-increasing for every compact and connected set $A$.
\end{enumerate}
\end{cor}

\begin{rem}
It follows from the above study that if a compact set $A \subset \R^n$ satisfies $c(A) \leq 2$, then the sequence $v(A(k))$ is non-increasing. One can see that if a compact set $A \subset \R^n$ contains the boundary of its convex hull, then $c(A) \leq 1$; for such set $A \subset \R^n$, the sequence $v(A(k))$ is non-increasing.
\end{rem}

\subsection{Convergence rates for $v$}

It is classical that one has convergence in $v$ at good rates.

\begin{thm}[\cite{Cas75}]\label{thm:cassels-v-gen}
Let $A_1, \ldots, A_k$ be compact sets in $\R^n$. Then
\ben
v(A_1 +\cdots + A_k) \leq \sqrt{\min\{k,n\}} \, \max_{i\in [k]} v(A_i) .
\een
\end{thm}

\begin{proof}
Firstly, by using subadditivity of $v^2$ (Theorem \ref{thm:v-subadd}), one has
$$ v^2(A_1 + \cdots + A_k) \leq k \max_{i \in [k]} v^2(A_i). $$
Hence, $v(A_1 + \cdots + A_k) \leq \sqrt{k} \max_{i \in [k]} v(A_i)$.

If $k\geq n+1$, we can improve this bound using Corollary~\ref{cor:v-large-k}, which gives us
\ben\begin{split}
v^2\left(\sum_{i \in [k]} A_i\right) 
&\leq  \max_{I \subset [k]: |I| = n} v^2\left(\sum_{j \in I} A_j \right)\\
&\leq \max_{I \subset [k]: |I| = n} \sum_{j \in I} v^2(A_j)\\
&\leq n \max_{i \in I} v^2(A_i) \leq n \max_{i \in [k]} v^2(A_i) ,
\end{split}\een
again using subadditivity of $v^2$ for the second inequality.
\end{proof}

By considering $A_1=\cdots=A_k=A$, one obtains the following convergence rate.

\begin{cor}\label{cor:vsfs}
Let $A$ be a compact set in $\R^n$. Then,
\ben
v(A(k))\leq \min\bigg\{ \frac{1}{\sqrt{k}}, \frac{\sqrt{n}}{k} \bigg\}  v(A).
\een 
\end{cor}

\section{The behavior of the Hausdorff distance from the convex hull}
\label{sec:d}

In this section we study the Hausdorff distance from the convex hull. Recall its definition: for $K$ being a compact convex set containing 0 in its interior and $A$ compact in $\R^n$,
\ben
d^{(K)}(A)=  \inf\{r>0: \conv(A)\subset A+rK\} .
\een

\subsection{Some basic properties of the Hausdorff distance}
\label{ss:d-mono-gen}

The Hausdorff distance is subadditive.

\begin{thm}\label{d-additive}
Let $A, B$ be compact sets in $\R^n$, and $K$ be an arbitrary convex body containing 0 in its interior. Then 
\ben
d^{(K)}(A+B) \leq d^{(K)}(A)+d^{(K)}(B) .
\een
\end{thm}

\begin{proof}
The convexity of $K$ implies that
\ben
A+B + (d^{(K)}(A)+d^{(K)}(B)) K = A + d^{(K)}(A)K + B + d^{(K)}(B)K,
\een
but since $A + d^{(K)}(A)K \supset \conv(A)$ and $B + d^{(K)}(B)K \supset \conv(B)$ by definition, we have
\ben
A+B + (d^{(K)}(A)+d^{(K)}(B)) K \supset \conv(A) +\conv(B) = \conv(A+B). 
\een
\end{proof}

We can provide a slight further strengthening of Theorem \ref{d-additive} when dealing with Minkowski sums of more than 2 sets, by following an argument similar to that used for Schneider's non-convexity index.

\begin{thm}\label{thm:d-3sum}
Let $A, B, C$ be compact sets in $\R^n$, and $K$ be an arbitrary convex body containing 0 in its interior. Then 
\ben
d^{(K)}(A+B+C) \leq d^{(K)}(A+B)+d^{(K)}(B+C).
\een
\end{thm}

\begin{proof}
Notice that
\ben\begin{split}
&A+B+C+\big(d^{(K)}(A+B)+d^{(K)}(B+C)\big)K \\
&= A+B+d^{(K)}(A+B) K + C + d^{(K)}(B+C) K \\
&\supset \conv(A+B)+C+d^{(K)}(B+C) K\\
&\supset \conv(A)+B+C+d^{(K)}(B+C) K\\
&\supset \conv(A)+ \conv(B+C)\\
&= \conv(A+B+C).
\end{split}\een
\end{proof}

In particular, Theorem~\ref{thm:d-3sum} implies that 
\ben
d^{(K)}\left(\sum_{l\in[k]}A_l\right) \leq 2 \max_{i\in[k]}d^{(K)}\left(\sum_{l\in[k]\setminus\{i\}}A_l\right) ,
\een
and, when the sets are the same,
\be\label{eq:d-part-mono}
d^{(K)}(A(k))\leq 2\frac{k-1}{k} d^{(K)}(A(k-1)).
\ee
While not proving monotonicity of $d^{(K)}(A(k))$, the inequality \eqref{eq:d-part-mono} does provide a bound on 
extent of non-monotonicity in the sequence in general dimension.

\subsection{The Dyn--Farkhi conjecture} 

Dyn and Farkhi \cite{DF04} conjectured that
\be\label{eq:sub-d}
d^2(A+B) \leq d^2(A)+d^2(B).
\ee
The next theorem shows that the above conjecture is false in $\R^n$ for $n\ge 3$.

\begin{thm}\label{thm:DF}
Let $q\geq 0$. The inequality
$$
d^q(A+B) \le d^q(A)+d^q(B),
$$
holds for all compact sets $A, B \subset {\mathbb R}^3$ if and only if $q \le 1$.
\end{thm}

\begin{proof} 
We have already seen that the inequality holds for $q=1$ and thus the inequality holds when $0\le q\le1$. 
Let $q\geq 0$ be such that the inequality holds for all compact sets $A$ and $B$. Let $A=A_1 \cup A_2$, 
where $A_1$ and $A_2$ are intervals such that  $A_1=[(0,0,0), (1,0,-f)]$ and $A_2=[(0,0,0), (1,0, f)]$, and $f>0$ is a large  number to be selected.
Let $B=B_1 \cup B_2$, where $B_1$ and $B_2$ are intervals such that  $B_1=[(0,0,0), (1,-f,0)]$ and $A_2=[(0,0,0), (1,f, 0)]$.
Note that $(1,0,0)$ belongs to both ${\rm conv}(A)$ and ${\rm conv}(B)$. It is easy to see, using two dimensional considerations that
$$d(B)=d(A)=d_A(1,0,0)=\frac{f}{\sqrt{1+f^2}}\le 1.$$ 
Next we notice that 
$$
A+B=\bigcup \limits_{i,j \in \{1,2\}} (A_i + B_j).
$$
 Thus, the points in $A+B$ can be parametrized by
$$
t  (1, \pm f, 0) + s(1,0 \pm f),
$$
where $t,s \in [0,1]$. We note that $(2,0,0) \in {\rm conv}(A+B)$ and 
$$
d(A+B)\ge d_{A+B}(2,0,0)=\min\limits_{t,s \in [0,1]}\sqrt{ (2-(t+s))^2 +f^2(s^2+t^2)}=\frac{2f}{\sqrt{f^2+2}}
$$
Note that if $f \to \infty$, this tends to $2$. So assuming that the inequality 
$$
d^q(A+B) \le d^q(A)+d^q(B),
$$
holds implies that $2^q\le 2$ thus $q \le 1$.
\end{proof}

\begin{rem}
\begin{enumerate}
\item 
Note that the above example is also valid if we consider $\ell_p$, $p\ge 1$ metric instead of the $\ell_2$ metric. 
Indeed $d^{(B_p^3)}(A)=d^{(B_p^3)}(B)\le 1$ and we may compute $\ell_p$ distance from $(2,0,0)$ to $A+B$ as
$$
\min\limits_{t,s \in [0,1]}\left((2-(t+s))^p +f^p(s^p+t^p)\right)^\frac{1}{p} .
$$
If $f \to \infty$, then to minimize the above, we must, again, select $s, t$ to be close to zero,
and thus the distance is at least $2$.  This shows that if the inequality
$$
(d^{(B_p^3)})^q(A+B) \le (d^{(B_p^3)})^q(A)+(d^{(B_p^3)})^q(B)
$$
holds for all $A, B \subset \R^3$, then $q \le 1$.

\item As shown by Wegmann \cite{Weg80}, if the set $A$ is such that the supremum in the definition of $v(A)$ 
is achieved at a point in the relative interior of $\conv(A)$, then $d(A)=v(A)$. Thus Theorem~\ref{thm:v-subadd} implies the following statement:
If $A, B$ are compact sets in $\R^n$ such that the supremum in the definition of $v(A)$ 
is achieved at a point in the relative interior of $\conv(A)$, and likewise for $B$, then 
\ben
d^2(A+B) \leq d^2(A)+d^2(B) .
\een

%{\color{red} QUESTION:} Can we understand this condition better (e.g., give nicer sufficient conditions on $A$ and $B$)?

\item We emphasize that the conjecture is still open in the case $A=B$. 
In this case, the Dyn-Farkhi conjecture is equivalent to 
\ben
d\left(\frac{A+A}{2}\right) \leq \frac{d(A)}{\sqrt{2}}.
\een
If $c_n$ is the best constant such that $d(\frac{A+A}{2}) \leq c_n d(A)$  for all compact sets $A$ in dimension $n$, 
then one has 
\ben 
c_n\ge \sqrt{\frac{n-1}{2n}}
\een
for $n\ge2$. This can be seen from the example where $A=\{a_1,\cdots, a_{n+1}\}$ is a set of $n+1$ vertices of a regular simplex in $\R^n$, $n\ge2$. 
For this example, it is not difficult to see that $d(A)=|g-a_1|$, where $g=(a_1+\cdots +a_{n+1})/(n+1)$ is the center of mass of $A$ and $d(\frac{A+A}{2})=|g-\frac{a_1+a_2}{2}|$. Then, one easily concludes that 
\ben
\frac{d(\frac{A+A}{2})}{d(A)}=\frac{|g-\frac{a_1+a_2}{2}|}{|g-a_1|}=\sqrt{\frac{n-1}{2n}}.
\een
Thus we get $\sup_n c_n \ge\frac{1}{\sqrt{2}}$, while the Dyn-Farkhi conjecture amounts to $\sup_n c_n \le\frac{1}{\sqrt{2}}$. 

\item Notice that there is another interpretation of $d(A)$ as the largest empty circle of $A$,
i.e., the radius of the circle of largest radius, centered at a point in $\conv(A)$ and containing no point of $A$ in its interior (see \cite{Sch08},
where the relevance of this notion for planning new store locations and toxic waste dump locations is explained). Indeed this radius is equal to
\ben
\sup\{R; \exists x\in\conv(A);  |x-a|\ge R, \forall a\in A\}=\sup\{R; \sup_{x\in\conv(A)} \inf_{a\in A}|x-a|\ge R\}=d(A).
\een

%\item It is interesting to note that $(d^{(K)})^2(A+B) \leq (d^{(K)})^2(A)+(d^{(K)})^2(B)$ is not true for a general convex, symmetric body $K$. This can be seen as a consequence of connections  to discrepancy theory (see Remark \ref{rem:sharp}) or directly by considering $K=B_1^n=\{x \in \R^n: \sum |x_i| \le 1\}$ and $A_i=\{-e_i, e_i\}$ for $i=1,\dots, n$, where $e_i$ is a standard basis of $\R^n$. Then $d^{(K)}(A_i)=1$, but $d^{(K)}\left( \sum A_i\right)=n$.

\end{enumerate}

\end{rem}

\subsection{Strong fractional subadditivity for large $k$}

In this section, similarly as for the effective standard deviation $v$, we prove that the Hausdorff distance from the convex hull $d^{(K)}$ 
satisfies a strong fractional subadditivity when considering sufficient large numbers of sets.

\begin{thm}\label{thm:d-large-k}
Let $K$ be an arbitrary convex body containing 0 in its interior. Let $A_1, \dots, A_k$ be compact sets in $\R^n$, with $k \geq n + 1$. Then,
$$ d^{(K)}\left(\sum_{i \in [k]} A_i \right) \leq \max_{I \subset [k]: |I| \leq n} \min_{i \in [k] \setminus I} d^{(K)}\left(\sum_{j \in [k] \setminus \{i\}} A_j \right). $$
\end{thm}

\begin{proof}
Let $x \in \conv(\sum_{i \in [k]} A_i)$. By using the Shapley-Folkman lemma (Lemma~\ref{lem:SF}), there exists a set $I \subset [k]$ of cardinality at most $n$ such that
$$ x \in \sum_{i \in I} \conv(A_i) + \sum_{i \in [k] \setminus I} A_i. $$
Let $i_0 \in [k] \setminus I$. In particular, we have
$$ x \in \sum_{i \in [k] \setminus \{i_0\}} \conv(A_i) + A_{i_0}. $$
Thus,
$$ x = \sum_{i \in [k] \setminus \{i_0\}} x_i + x_{i_0} = z + x_{i_0}, $$
for some $x_i \in \conv(A_i)$, $i \in [k] \setminus \{i_0\}$, and some $x_{i_0} \in A_{i_0}$, where $z = \sum_{i \in [k] \setminus \{i_0\}} x_i$. Hence,
\begin{eqnarray*}
d^{(K)}_{\sum_{i \in [k] \setminus \{i_0\}} A_i}(z) & = & \inf_{a \in \sum_{i \in [k] \setminus \{i_0\}} A_i} \|z-a\|_K \\ & = & \inf_{a \in \sum_{i \in [k] \setminus \{i_0\}} A_i} \|z + x_{i_0} - (a + x_{i_0}) \|_K \\ & \geq & \inf_{a^* \in \sum_{i \in [k]} A_i} \|z + x_{i_0} - a^* \|_K \\ & = & d^{(K)}_{\sum_{i \in [k]} A_i}(x).
\end{eqnarray*}
Taking supremum over all $z \in \conv(\sum_{i \in [k] \setminus \{i_0\}} A_i)$, we deduce that
$$ d^{(K)}_{\sum_{i \in [k]} A_i}(x) \leq d^{(K)}(\sum_{i \in [k] \setminus \{i_0\}} A_i). $$
Since this is true for every $i_0 \in [k] \setminus I$, we deduce that
$$ d^{(K)}_{\sum_{i \in [k]} A_i}(x) \leq \min_{i \in [k] \setminus I} d^{(K)}(\sum_{j \in [k] \setminus \{i\}} A_j). $$
Taking the supremum over all set $I \subset [k]$ of cardinality at most $n$ yields
$$ d^{(K)}_{\sum_{i \in [k]} A_i}(x) \leq \max_{I \subset [k]: |I| \leq n} \min_{i \in [k] \setminus I} d^{(K)}(\sum_{j \in [k] \setminus \{i\}} A_j). $$
We conclude by taking the supremum over all $x \in \conv(\sum_{i\in [k]} A_i)$.
\end{proof}

In the case where $A_1 = \cdots = A_k = A$, we can use the above argument to prove that for $k \geq c(A) + 1$,
$$ d^{(K)}(A(k)) \leq \frac{k-1}{k} d^{(K)}(A(k-1)), $$
where $c(A)$ is the Schneider non-convexity index of $A$. Since $c(A) \leq n$, and $c(A) \leq n-1$ when $A$ is connected, 
we deduce the following monotonicity property for the Hausdorff distance to the convex hull.

\begin{cor}
Let $K$ be an arbitrary convex body containing 0 in its interior. Then,
\begin{enumerate}
\item In dimension 1 and 2, the sequence $d^{(K)}(A(k))$ is non-increasing for every compact set  $A$.
%{\color{red} I do not think we need to write $A \subset \R^n$, $n=1,2$}
\item In dimension 3, the sequence $d^{(K)}(A(k))$ is non-increasing for every compact and connected set $A$.
\end{enumerate}
\end{cor}

\begin{rem}
It follows from the above study that if a compact set $A \subset \R^n$ satisfies $c(A) \leq 2$, then the sequence $d^{(K)}(A(k))$ is non-increasing. 
One can see that if a compact set $A \subset \R^n$ contains the boundary of its convex hull, then $c(A) \leq 1$; for such set $A \subset \R^n$, 
the sequence $d^{(K)}(A(k))$ is non-increasing.
\end{rem}

It is useful to also record a simplified version of Theorem~\ref{thm:d-large-k}.

\begin{cor}\label{cor:d-large-k}
Let $K$ be an arbitrary convex body containing 0 in its interior. Let $A_1, \dots, A_k$ be compact sets in $\R^n$, with $k \geq n + 1$. Then,
$$ d^{(K)}\left(\sum_{i \in [k]} A_i\right) \leq \max_{I \subset [k]: |I| = n} d^{(K)}\left(\sum_{i \in I} A_i \right) \leq n \max_{i\in [k]} d^{(K)}(A_i) .$$
\end{cor}

\begin{proof}
By Theorem~\ref{thm:d-large-k}, provided $k>n$, we have in particular
\ben 
d^{(K)}\left(\sum_{i \in [k]} A_i\right) \leq \max_{i\in [k]} d^{(K)}\left(\sum_{j \neq i} A_j \right) .
\een
Iterating the same argument as long as possible, we have that
\ben 
d^{(K)}\left(\sum_{i \in [k]} A_i\right) \leq \max_{I \subset [k]: |I| = n} d^{(K)}\left(\sum_{j \in I} A_j \right) ,
\een
which is the first desired inequality. Applying the subadditivity property of $d^{(K)}$ (namely, Theorem~\ref{d-additive}),
we immediately have the second desired inequality.
\end{proof}

While Corollary~\ref{cor:d-large-k} does not seem to have been explicitly written down before, it seems to have
been first discovered by V. Grinberg (personal communication).

\subsection{Convergence rates for $d$}
\label{sec:d-rate}

Let us first note that having proved convergence rates for $v(A(k))$, we automatically
inherit convergence rates for $d^{(K)}(A(k))$ as a consequence of Lemma \ref{lem:d-d^K}, Theorem \ref{thm:weg} and Corollary \ref{cor:vsfs}.

\begin{cor}\label{d-rate-v-rate} Let $K$ be an arbitrary convex body containing 0 in its interior. For any compact set $A\subset \R^n$,
\ben
d^{(K)}(A(k))\leq \frac{1}{r}\min\bigg\{ \frac{1}{\sqrt{k}}, \frac{\sqrt{n}}{k} \bigg\}  v(A),
\een 
where $r>0$ is such that $rB_2^n \subset K$.
\end{cor}

For Euclidean norm (i.e., $K=B_2^n$), this goes back to \cite{Sta69, Cas75}.

Although we have a strong convergence result for $d^{(K)}(A(k))$ as a consequence of that for $v(A(k))$, we give below another estimate of $d^{(K)}(A(k))$ in terms of $d^{(K)}(A)$, instead of $v(A)$.

\begin{thm}\label{d-easy-rate} For any compact set $A\subset \R^n$,
\ben
d^{(K)}(A(k)) \leq \min\left\{1,\frac{\lceil c(A) \rceil}{k}\right\}d^{(K)}(A).
\een 
\end{thm}

\begin{proof}
As a consequence of Theorem \ref{d-additive}, we always have $d^{(K)}(A(k)) \leq d^{(K)}(A)$. Now consider $k \geq c(A) + 1$, and notice that
\ben
kA(k) + \lceil c(A) \rceil d^{(K)}(A)K \supset (k- \lceil c(A) \rceil )A(k - \lceil c(A) \rceil) + \lceil c(A) \rceil \conv(A) = \conv(kA(k)).
\een
Hence $d^{(K)}(kA(k)) \leq \lceil c(A) \rceil d^{(K)}(A)$, or equivalently, $d^{(K)}(A(k)) \leq \frac{\lceil c(A) \rceil d^{(K)}(A)}{k}$.
\end{proof}

Using the fact that $c(A) \leq n$ for every compact set $A \subset \R^n$, we deduce that
$$ d^{(K)}(A(k)) \leq \min\left\{1,\frac{n}{k}\right\}d^{(K)}(A). $$

%The analysis of the behavior of $d$ for self-averages of simplices (which was begun in Remark 3 of
%Section~\ref{ss:d-mono-gen}) can be extended to show that if $S$ is a simplex in dimension $n$,
%\ben
%d(S(k)) \geq \frac{C\sqrt{n}}{n+k} .
%\een
%Note $\frac{2\sqrt{n}}{n+k} \leq 1/\sqrt{k}$ by the AM-GM inequality,
%but $\frac{C\sqrt{n}}{n+k}$ cannot always be smaller than $1/k$.
%This means that it is impossible to have a dimension-independent bound of the form
%\ben
%d(A(k))\leq C\frac{d(A)}{k} 
%\een
%for arbitrary sets $A\subset \R^n$, which one might have hoped for given the dimension-independent rate for $c$.

\section{Connections to discrepancy theory}
\label{sec:discrep}

The ideas in this section have close connections to the area known sometimes as ``discrepancy theory'',   
which has arisen independently in the theory of Banach spaces, combinatorics, and computer science.
It should be emphasized that there are two distinct but related areas that go by the name of discrepancy theory. 
The first, discussed in this section and sometimes called
``combinatorial discrepancy theory'' for clarity, was likely originally motivated by questions related to absolute versus unconditional 
versus conditional convergence for series in Banach spaces. The second, sometimes called
``geometric discrepancy theory'' for clarity, is related to how well a finite set of points can approximate a uniform distribution on
(say) a cube in $\R^n$. Our discussion here concerns the former; the interested reader may consult \cite{Tra14:book} for more on the latter.
When looked at deeper, however, combinatorial discrepancy theory is also related to the ability to discretely approximate ``continuous''
objects. For example, a famous result of Spencer \cite{Spe85} says that given any collection $\{S_1,\ldots,S_n\}$ of subsets of  $[n]$, it is possible to color the 
elements of $[n]$ with two colors (say, red and blue) such that
\ben
\bigg| |S_i\cap R| - \frac{|S_i|}{2} \bigg|\leq 3\sqrt{n},
\een
for each $i\in [n]$, where $R\subset [n]$ is the set of red elements. As explained for example by Srivastava \cite{Sri13:blog}
\begin{quote}
In other words, it is possible to partition $[n]$ into two subsets 
so that this partition is very close to balanced on {\it each one} of the test sets $S_i$. 
Note that a ``continuous'' partition which splits each element exactly in half will be exactly balanced on each $S_i$; 
the content of Spencer's theorem is that we can get very close to this ideal situation with an actual, discrete partition 
which respects the wholeness of each element.
\end{quote}
Indeed, Srivastava also explains how the recent celebrated results of Marcus, Spielman and Srivastava \cite{MSS15:1, MSS15:2}
that resulted in the solution of the Kadison-Singer conjecture may be seen from a discrepancy point of view.

For any $n$-dimensional Banach space $E$ with norm $\|\cdot\|_E$, define the functional 
\ben
V(k, E)= \max_{x_1, \ldots, x_k: \|x_i\|=1 \,\forall i\in [k]} \,\, \min_{(\eps_1, \ldots, \eps_k)\in \{-1,1\}^k} \, \bigg\| \sum_{i\in [k]} \eps_i x_i\bigg\|_E.
\een
In other words, $V(k,E)$ answers the question: for any choice of $k$ unit vectors in $E$, how small are we guaranteed to be able to make the 
signed sum of the unit vectors by appropriately choosing signs? The question of what can be said about the numbers $V(k,E)$ 
was first asked\footnote{See \cite[p. 496]{Kle63:book} where this question is stated as one in a collection of then-unsolved problems.} by A. Dvoretzky in 1963. 
Let us note that the same definition also makes sense when $\|\cdot\|$ is a nonsymmetric norm (i.e., satisfies $\|ax\|=a\|x\|$ for $a>0$,
positive-definiteness and the triangle inequality), and we will discuss it in this more general setting.

%It is useful to observe that one has a partial monotonicity property for the quantities $V(k,E)$.  {\color{red} The proof of lemma below is clearly incorrect, It is easy to see that $V(1, E) \le V(3, E)$, just by considering the same vector $3$ times, BUT I could not create anything more interesting out of it. The remark after it is correct}: 

%\begin{lem}\label{lem:V-mono}
%For any $k\in\Nat$ and any nonsymmetric norm,
%\ben
%V(k,E)\leq V(k+2, E).
%\een
%\end{lem}
%\begin{proof}
%Suppose $\{x_1^*, \ldots, x_k^*\}$ is an optimal configuration, i.e., this is a set of unit vectors in $E$ that achieve the maximum in the definition of $V(k,E)$.
%Then by considering the set $\{x_1^*, \ldots, x_k^*, a, a\}$ for any unit vector $a$, we see that
%\begin{eqnarray*}
%V(k+2, E) &\geq& \min_{(\eps_1, \ldots, \eps_{k+2})\in \{-1,1\}^{k+2}} \|\eps_1 x_1^*+  \ldots  +\eps_k x_k^*+ \eps_{k+1}a + \eps_{k+2} a\| \\
%&\geq& \min_{(\eps_1, \ldots, \eps_{k+2})\in \{-1,1\}^{k+2}} \left(\|\eps_1 x_1^*+  \ldots  +\eps_k x_k^*\|- \| \eps_{k+1}a + \eps_{k+2} a\| \right)\\
%&=& \min_{(\eps_1, \ldots, \eps_{k})\in \{-1,1\}^{k}} \|\eps_1 x_1^*+  \ldots  +\eps_k x_k^*\|\\
%&=& V(k,E).
%\end{eqnarray*}
%\end{proof}

It is a central result of discrepancy theory \cite{GS80, BG81} that when $E$ has dimension $n$, it always holds\footnote{The fact 
that $V(k,E)\leq n$ appears to be folklore and the first explicit mention of it we could find is in \cite{GS80}.} that $V(k,E) \leq n$.
To make the connection to our results, we observe that this fact actually  follows from Corollary~\ref{cor:d-large-k}.

\begin{thm}\label{thm:discrep-main}
Suppose $A_1, \ldots, A_k \subset K$, where $K$ is a convex body in $\R^n$ containing 0 in its interior (i.e., the unit ball of a non-symmetric norm $\|\cdot \|_K$),
and suppose $0\in \conv(A_i)$ and $\dim(A_i)=1$ for each $i\in [k]$. Then there exist vectors $a_i\in A_i$ ($i\in [k]$) such that
\ben
\bigg\| \sum_{i\in [k]} a_i\bigg\|_K  \leq n .
\een
In particular, if $K$ is symmetric,  then by choosing $A_i=\{x_i, -x_i\}$, with $\|x_i\|_K=1$, one immediately has $V(k, E_K)\leq n$
for  $E_K=(\R^n, \|\cdot\|_K)$.
\end{thm}

\begin{proof}
We simply observe that since $0\in \conv(\sum_{i\in [k]} A_i)$, there exists a point $a_0\in\sum_{i\in [k]} A_i$ such that
$$
\|a_0\|_K\leq \sup\limits_{x \in \conv(\sum\limits_{i\in [k]} A_i)} \inf\limits_{a \in \sum_{i\in [k]} A_i}  \|a-x \|_K=d^{(K)}(\sum\limits_{i\in [k]} A_i) \leq n \max_{i\in [k]} d^{(K)}(A_i),$$
where the last inequality follows from Corollary \ref{cor:d-large-k}. Moreover, using that for each $i \in [k]$,
$A_i \subseteq K$ and $K$ is convex, we get  $\conv(A_i) \subseteq K$. Thus by Lemmata~\ref{lem:d^K-d^L} and \ref{lem:d^conv}, $ d^{(K)}(A_i) \le  d^{(\conv(A_i))}(A_i) \leq c(A_i)\leq 1$,
where the last inequality uses Theorem~\ref{thm:sch75} and the assumption that $\dim(A_i)=1$.
\end{proof}

\begin{rem}
B\'ar\'any and Grinberg \cite{BG81} proved  Theorem~\ref{thm:discrep-main} without the condition $\dim(A_i)=1$.
They also proved it for symmetric bodies $K$ under the weaker condition that 
$0\in \conv(\sum_{i\in [k]} A_i)$; we will recover this fact for  symmetric bodies, without restriction on the dimension,
as a consequence of Theorem \ref{thm:discrep-super} below. 
\end{rem}

%The above method can be used to prove theorem for $K=B_2^n$ and general $A_i \subset B_2^n$, indeed in this case $d^{B_2^n}(A_i) \le 1$.

\begin{rem}\label{rem:sharp}
As pointed out in \cite{BG81},  Theorem~\ref{thm:discrep-main} is sharp. By taking $E=\ell_1^n$ and $x_i$ to be the $i$-th standard basis vector $e_i$ of $\R^n$,
we see that for any choice of signs, $\big\| \sum_{i\in [n]} \eps_i x_i\big\|= n$, which implies that $V(n, \ell_1^n)=n$.
\end{rem}

\begin{rem}
It is natural to think that the sequence $V(k,E)$ may be monotone with respect to $k$. Unfortunately, this is not true. 
Swanepoel \cite{Swa00} showed that $V(k,E)\leq 1$ for every {\it odd} $k$ and every $2$-dimensional Banach space $E$.
Consequently, we have $V(1, \ell_1^2)=1$ and $V(3, \ell_1^2)\leq 1$,
whereas we know from Remark~\ref{rem:sharp} that $V(2, \ell_1^2)=2$. 
%whereas it follows from Theorem \ref{thm:discrep-main} discussed below that $V(k,E)=2$ when $k$ is even
%and $E=\ell_1^2$ (hence the sequence $(V(k,\ell_1^2):k\in\Nat)$ alternately takes the value $2$ and a value not exceeding $1$)\footnote{The question of what precisely is the best bound
%on $V(k,E)$ when $k$ and $n$ have opposite parity is interesting and open. It was conjectured independently by Swanepoel and Grinberg
%that in this case, $V(k,E)\leq n-1$; this remains open even for the case where $n=3$ and $k=6$.}.
%{\color{red} This is not correct. We only know from the remark after Theorem \ref{thm:discrep-main} that
%$V(2,\ell_1^2)=2$. It would be good to figure out what $V(k,\ell_1^2)$ is for other even $k$ (a good guess is that it is always 2 in this case, but we need a proof).}
\end{rem}

%Moreover, by Lemma~\ref{lem:V-mono} combined with the upper bound of Theorem~\ref{thm:discrep-main}, 
%we have $V(k, \ell_1^n)=n$ whenever $k, n$ have the same parity\footnote{The question of what precisely is the best bound
%on $V(k,E)$ when $k$ and $n$ have opposite parity is interesting and open. It was conjectured independently by Swanepoel and Grinberg
%that in this case, $V(k,E)\leq n-1$; this remains open even for the case where $n=3$ and $k=6$.}.

Not surprisingly, for special norms, better bounds can be obtained. In particular  (see, e.g., \cite[Theorem 2.4.1]{AS00:book} or \cite[Lemma 2.2]{Bec83:2}),   $V(k, \ell_2^n)\leq \sqrt{n}$. 
We will present a proof of this and more general facts in Theorem \ref{thm:discrep-super}.
But first let us discuss a quite  useful observation about the quantity $V(k,E)$: it is an isometric invariant,
i.e., invariant under nonsingular linear transformations of the unit ball. A way to measure the extent of isometry  is
using the Banach-Mazur distance $d_{BM}$: Let $E$, $E'$ be two $n$-dimensional normed spaces. The Banach-Mazur distance between them is defined as
$$
d_{BM} (E, E')=\inf\{\|T\| \cdot \|T^{-1}\|; T:E \to E'  \mbox{ isomorphism}\}.
$$
Thus $d_{BM}(E,E') \ge 1$ and $d_{BM}(E,E') = 1$ if and only if $E$ and $E'$ are isometric. We also remind that the above notion have a geometrical interpretation. Indeed if we denote by $B(X)$ a unit ball of 
Banach space $X$, then $d_{BM}(E,E')$ is a minimal positive number such that there exists a linear transformation $T$ with:
$$
B(E) \subseteq T(B(E')) \subseteq  d_{BM}(E,E') B(E).
$$
\begin{lem}\label{lem: V-BM}
If $d_{BM}(E,E')=1$, then
\ben
V(k,E)=V(k,E').
\een
\end{lem}

\begin{proof} Consider an invertible  linear transformation $T$ such that $T(B(E))=B(E')$ and thus $\| y\|_E = \|Ty\|_{E'}$, then
\begin{eqnarray*}
V(k, E)&= & \max_{x_1, \ldots, x_k: \|x_i\|_E=1 \,\forall i\in [k]} \,\, \min_{(\eps_1, \ldots, \eps_k)\in \{-1,1\}^k} \, \bigg\| \sum_{i\in [k]} \eps_i x_i\bigg\|_E\\
&= & \max_{x_1, \ldots, x_k: \|T x_i\|_{E'}=1 \,\forall i\in [k]} \,\, \min_{(\eps_1, \ldots, \eps_k)\in \{-1,1\}^k} \, \bigg\| T \left(\sum_{i\in [k]} \eps_i x_i \right)\bigg\|_{E'}\\
&= & \max_{y_1, \ldots, y_k: \|y_i\|_{E'}=1 \,\forall i\in [k]} \,\, \min_{(\eps_1, \ldots, \eps_k)\in \{-1,1\}^k} \, \bigg\| \sum_{i\in [k]} \eps_i y_i \bigg\|_{E'}
\end{eqnarray*}
\end{proof}

%{\color{red} QUESTION: Can we show that $V(k,E)\leq V(k,E') d_{BM}(E,E')$? This would be a very nice generalization that would also imply
%the conjecture at the bottom of this section. The one of the main problem in showing this is that it is NOT so clear how to ouse inequality between norms in the definition of 
%$V(k,E)$, indeed, I do not think that we can replace $\|x_i\|=1$ with $\|x_i\| \le 1$, or more precisely, say  $E \subseteq E'$ what can we say about $V(k,E)$ and  $V(k,E')$.}

  %\begin{thm}\label{thm:discrep-eucl}
%Suppose $A_1, \ldots, A_k \subset B_2^n$,
%and suppose $0\in \conv(\sum_{i\in [k]} A_i)$. Then there exist vectors $a_i\in A_i$ ($i\in [k]$) such that
%\ben
%\bigg\| \sum_{i\in [k]} a_i\bigg\|  \leq \sqrt{n}.
%\een
%In particular, by choosing $A_i=\{x_i, -x_i\}$, with $\|x_i\|=1$, one immediately has $V(k, \ell_2^n)\leq \sqrt{n}$.
%\end{thm}

%\begin{proof}
%As in the proof of  Theorem \ref{thm:discrep-main} we observe that since $0\in \conv(\sum_{i\in [k]} A_i)$, there exists a point $a\in\sum_{i\in [k]} A_i$ such that
%$$
%\|a\| \leq d (\sum\limits_{i\in [k]} A_i)  \le \max_{I \subset [k]:\card(I) = n} d\left(\sum_{j \in I} A_j \right),
%$$
%where the last inequality follows from Corollary \ref{cor:d-large-k}. Now we can apply Theorems \ref{thm:weg} and \ref{thm:v-subadd} to get
%$$
% \max_{I \subset [k]:\card(I) = n} d\left(\sum_{j \in I} A_j \right)  \le  \max_{I \subset [k]:\card(I) = n} v\left(\sum_{j \in I} A_j \right)  \le \max_{I \subset [k]:\card(I) = n} \sqrt{ \sum_{j \in I} v^2( A_j) } \le \sqrt{n}.
%$$
%\end{proof}

 Now  we would like to use the ideas of the proof  of Theorem \ref{thm:discrep-main} together with  Lemma \ref{lem: V-BM} to prove the following statement that will help us to provide sharper bounds for $V(k,E)$ for intermediate norms.
 
 \begin{thm}\label{thm:discrep-super}
Suppose $A_1, \ldots, A_k \subset K$, where $K$ is a symmetric convex body in $\R^n$ (i.e., the unit ball of a norm $\|\cdot \|_K$), 
%containing $0$ in its interior (i.e., the unit ball of a non-symmetric norm $\|\cdot \|_K$),
and suppose $0\in \conv(\sum_{i\in [k]} A_i)$. Then there exist vectors $a_i\in A_i$ ($i\in [k]$) such that
\ben
\bigg\| \sum_{i\in [k]} a_i\bigg\|_K  \leq \sqrt{n}\, d_{BM}(E,\ell_2^n),
\een
where $E=(\R^n, \|\cdot\|_K)$. In particular, by choosing $A_i=\{x_i, -x_i\}$, with $\|x_i\|_K=1$, one immediately has 
\ben
V(k,E)\leq \sqrt{n} \, d_{BM}(E,\ell_2^n).
\een
\end{thm}
\begin{proof}
Let $d=d_{BM}(E,\ell_2^n)$, then we may assume, using Lemma \ref{lem: V-BM}, that $B_2^n \subset K \subset d B_2^n$. Next,  as in the proof of  Theorem \ref{thm:discrep-main} we observe that since $0\in \conv(\sum_{i\in [k]} A_i)$, there exists a point $a\in\sum_{i\in [k]} A_i$ such that
$$
\|a\|_K \leq d^{(K)} \left(\sum\limits_{i\in [k]} A_i\right)  \le \max_{I \subset [k]: |I| = n} d^{(K)}\left(\sum_{j \in I} A_j \right),
$$
where the last inequality follows from Corollary \ref{cor:d-large-k}. Next, we apply Lemma \ref{lem:d^K-d^L} together with $B_2^n \subset K$ to get
$$
\max_{I \subset [k]: |I| = n} d^{(K)}\left(\sum_{j \in I} A_j \right) \le \max_{I \subset [k]: |I| = n} d\left(\sum_{j \in I} A_j \right).
$$
Now we can apply Theorems \ref{thm:weg} and \ref{thm:v-subadd} to get
$$
 \max_{I \subset [k]: |I| = n} d\left(\sum_{j \in I} A_j \right)  \le  \max_{I \subset [k]: |I| = n} v\left(\sum_{j \in I} A_j \right)  \le \max_{I \subset [k]: |I| = n} \sqrt{ \sum_{j \in I} v^2( A_j) } \le d\sqrt{n},
$$
where the last inequality follows from the fact that $v(A_i)=r(A_i)$ is bounded by $d$ since $A_i\subset K \subset d B_2^n$.
\end{proof}
We note that it follows from F. John Theorem  (see, e.g., \cite[page 10]{MS86:book}) that 
$d_{BM}(E, \ell_2^n) \le \sqrt{n}$ for any $n$-dimensional Banach space $E$.
Thus we have the following corollary, which recovers a result of \cite{BG81}.

\begin{cor}\label{cor:discrep-super}
Suppose $A_1, \ldots, A_k \subset K$, where $K$ is a convex symmetric body in $\R^n$,
and suppose $0\in \conv(\sum_{i\in [k]} A_i)$. Then there exist vectors $a_i\in A_i$ ($i\in [k]$) such that
\ben
\bigg\| \sum_{i\in [k]} a_i\bigg\|_K  \leq n.
\een
 In particular, by choosing $A_i=\{x_i, -x_i\}$, with $\|x_i\|_K=1$, one immediately has 
\ben
V(k,E)\leq n,
\een
where $E=(\R^n, \|\cdot\|_K)$.
\end{cor}

%The fact that $V(k,E)\leq n$ appears to be folklore and the first explicit mention of it we could find is in \cite{GS80}.

It is well known that $d_{BM}(\ell_p^n, \ell_2^n)=n^{|\frac{1}{p}-\frac{1}{2}|}$ for $p\ge 1$ (see, e.g., \cite[page 20]{MS86:book}). 
Thus  Theorem \ref{thm:discrep-super} gives:

\begin{cor}
For any $p\geq 1$ and any $n\in\Nat$,
$$
V(k, \ell_p^n) \le n^{\frac{1}{2}+|\frac{1}{p}-\frac{1}{2}|}.
$$
\end{cor}

In particular, we recover the classical fact that $V(k, \ell_2^n)\leq \sqrt{n}$, which can be found, e.g., in \cite[Theorem 2.4.1]{AS00:book}.
V. Grinberg (personal communication) informed us of the following elegant and sharp bound generalizing this fact that he obtained in unpublished work:
if $A_i$ are subsets of $\R^n$ and $D=\max_i \diam(A_i)$, then
\be\label{eq:grinberg}
d\bigg(\sum_{i\in [k]} A_i\bigg) \leq \frac{D}{2} \sqrt{n}.
\ee
The special case of this when each $A_i$ has cardinality 2 is due to Beck \cite{Bec83:2}. Let us note that the inequality \eqref{eq:grinberg} improves upon
the bound of $\sqrt{n}\max_i v(A_i)$ that is obtained in the Shapley-Folkman theorem by combining Theorems~\ref{thm:weg} and \ref{thm:cassels-v-gen}.

Finally let us note that the fact that the quantities $V(k,E)$ are $O(n)$ for general norms and $O(\sqrt{n})$ for Euclidean norm
is consistent with the observations in Section~\ref{sec:d-rate} that the rate of convergence of $d^{(K)}(A(k))$ for a compact set $A\subset \R^n$
is $O(n/k)$ for general norms and $O(\sqrt{n}/k)$ for Euclidean norm (i.e., $K=B_2^n$). 
%\footnote{It is tempting to think that all results for discrepancy immediately follow from the earlier results in this section.
%However this is not the case-- for example, Corollary~\ref{d-rate-v-rate} or its generalization to non-identical summands (first stated in \cite{Sta69}) 
%prove a bound on Hausdorff distance from the convex hull of the Minkowski sum in terms of the inner radius of the summands, which does not recover Grinberg's result that
%$V(k, \ell_2^n)\leq \sqrt{n}$ since the inner radius is in general larger.} 

We do not comment further on the relationship of our study with discrepancy theory, which contains many interesting results and questions
when one uses {\it different} norms to pick the original unit vectors, and to measure the length of the signed sum (see, e.g., \cite{BF81:1, Gia97, Nik13}). The interested reader
may consult the books \cite{Cha00:book, Mat10:book, CST14:book} for more in this direction, including discussion of algorithmic issues and applications
to theoretical computer science. There are also connections to the Steinitz lemma \cite{Bar08}, which was originally discovered
in the course of extending the Riemann series theorem (on the real line being the set of possible limits by rearrangements of 
a conditionally convergent sequence of real numbers) to sequences of vectors (where it is called the L\'evy-Steinitz theorem,
and now understood in quite general settings, see, e.g., \cite{Sof08}).

\section{Discussion}
\label{sec:disc}

%\noindent{\color{red} Make whatever summarizing remarks we want}
%
%\noindent{\color{red} Possibly collect all open questions mentioned in the paper and state again in one place?}

Finally we mention some notions of non-convexity that we do not take up in this paper:
\begin{enumerate}
\item Inverse reach: The notion of reach was defined by Federer \cite{Fed59}, %(cf. Durand \cite{Dur31}), 
and plays a role in geometric measure theory. For a set $A$ in $\R^n$, the reach of $A$ is defined as
\ben
\text{reach}(A)=\sup \{ r>0: \forall y\in A+rB_2^n,\, \text{there exists a unique}\, x\in A\, \text{nearest to}\, y \}.
\een 
A key property of reach is that $\text{reach}(A)=\infty$ if and only if $A$ is convex; consequently one may think
of 
\ben
\iota(A)=\text{reach}(A)^{-1}
\een 
as a measure of non-convexity.
Th\"ale \cite{Tha08} presents a comprehensive survey of the study of 
sets with positive reach (however, one should take into account the
cautionary note in the review of this article on MathSciNet).
%some additional results of interest may be found in Colesanti and Manselli \cite{CM10}.
\item Beer's index of convexity: First defined and studied by Beer \cite{Bee73:1}, this quantity
is defined for a compact set $A$ in $\R^n$ as the probability that 2 points drawn uniformly from $A$ at random
``see'' each other (i.e., the probability that the line segment connecting them is in $A$). Clearly this
probability is 1 for convex sets, and 0 for finite sets consisting of more than 1 point. Since our study has
been framed in terms of measures of non-convexity, it is more natural to consider
\ben
b(A)= 1-\P \{ [X,Y]\subset A\} ,
\een
where $X, Y$ are i.i.d. from the uniform measure on $A$, and $[x,y]$ denotes the line segment connecting
$x$ and $y$.
\item Convexity ratio: The convexity ratio of a set $A$ in $\R^n$ is defined as the ratio of the volume
of a largest convex subset of $A$ to the volume of $A$; it is clearly 1 for convex sets and can be arbitrarily
close to 0 otherwise. For dimension 2, this has been studied, for example,
by Goodman \cite{Goo81}. Balko et al. \cite{BJVW14} discuss this notion in general dimension, and also
give some inequalities relating the convexity ratio and Beer's index of convexity. Once again, to get 
a measure of non-convexity, it is more natural to consider
\ben
\kappa(A)=1-\frac{\vol_n(L(A))}{\vol_n(A)} ,
\een
where $L(A)$ denotes a largest convex subset of $A$.
\end{enumerate}

These notions of non-convexity are certainly very interesting, but they behave quite 
differently from the notions we have explored thus far.
For example, if $b(A)=0$ or $\kappa(A)=0$, the compact set $A$ may not be convex, but differ from a convex set 
by a set of measure zero. For example, if $A$ is the union of a unit Euclidean ball and a point separated from it,
then 
\be\label{eq:bk0}
b(A)=\kappa(A)=0 ,
\ee 
even though $A$ is compact but non-convex. 
Even restricting to compact {\it connected} sets does not help-- just connect the disc with a point by a segment,
and we retain \eqref{eq:bk0} though $A$ remains non-convex.

It is possible that further restricting to connected open sets
%(or compact sets with connected interior) 
is the right thing to do here-- this may yield a characterization of convex sets using $b$ and $\kappa$,
but it still is not enough to ensure stability of such a characterization. For example, $b(A)$ small 
would not imply that $A$ is close to its convex hull even for this restricted class of sets, because we can take 
the previous example of a point connected to a disc by a segment and just slightly fatten the segment. 
%Nonetheless may be b(A)=0 still implies A convex in this case, because any such counterexample would still give small positive value of b.

%Not sure if it kills A+A example or not...

Generalizing this example leads to a curious phenomenon. Consider $A=B_2^n \cup \{x_1,...,x_N\}$,
where $x_1, \ldots, x_N$ are points in $\R^n$ well separated %(say 20) 
from each other and the origin.
Then $b(A)=\kappa(A)=0$, but we can send $b(\frac{A+A}{2})$ and $\kappa(\frac{A+A}{2})$ 
arbitrarily close to 1 by making $N$ go to infinity
(since isolated points are never seen for $A$ but become very important for the sumset).
This is remarkably bad behavior indeed, since it indicates an extreme violation of the monotone decreasing
property of $b(A(k))$ or $\kappa(A(k))$ that one might wish to explore, already in dimension 2.

Based on the above discussion, it is clear that the measures $\iota, b, \kappa$ of non-convexity are 
more sensitive to the topology of the set than the functionals we considered in most of this paper. 
Thus it is natural that the behavior of these additional measures for Minkowski sums should be 
studied with a different global assumption than in this paper (which has focused on what can be said for compact sets). We hope to investigate this question in future work.

\noindent Matthieu Fradelizi \\
Universit\'e Paris-Est, LAMA (UMR 8050), UPEMLV, UPEC, CNRS,
F-77454, France\\
E-mail address: matthieu.fradelizi@u-pem.fr 

\vspace{0.8cm}

\noindent Mokshay~Madiman\\
University of Delaware\\
Department of Mathematical Sciences\\
501 Ewing Hall \\
Newark, DE 19716, USA\\
E-mail address: madiman@udel.edu

\vspace{0.8cm}

\noindent Arnaud Marsiglietti \\
Center for the Mathematics of Information\\
California Institute of Technology\\
1200 E California Blvd, MC 305-16\\
Pasadena, CA 91125, USA \\
E-mail address: amarsigl@caltech.edu

\vspace{0.8cm}

\noindent Artem Zvavitch \\
Department of Mathematical Sciences \\
Kent State University \\
Kent, OH 44242, USA \\
E-mail address: zvavitch@math.kent.edu

\end{document}